\documentclass[a4paper,11pt]{article}
\usepackage{bbm}
\usepackage{mathrsfs}
\usepackage{amsthm}
\usepackage{amsmath}
\usepackage{amssymb}
\usepackage{graphicx}
\usepackage{epsfig}
\usepackage{enumerate}
\usepackage{tikz,hyperref}
\usepackage{appendix}
\usepackage{comment}
\usepackage{extarrows}
\usepackage{CJK}
\numberwithin{equation}{section}
\newtheorem{theorem}{Theorem}[section]
\newtheorem{corollary}[theorem]{Corollary}

\newtheorem{lemma}[theorem]{Lemma}
\newtheorem{proposition}[theorem]{Proposition}

\theoremstyle{definition}
\newtheorem{definition}[theorem]{Definition}
\newtheorem{remark}[theorem]{Remark}

\newcommand{\R}{\mathbf{R}}
\newcommand{\C}{\mathbf{C}}
\newcommand{\N}{\mathbf{N}}
\newcommand{\Z}{\mathbf{Z}}
\newcommand{\B}{\mathcal{B}}

\DeclareMathOperator{\GL}{GL}
\DeclareMathOperator{\U}{U}
\DeclareMathOperator{\cl}{cl}
\DeclareMathOperator{\Cr}{Cr}
\DeclareMathOperator{\Sp}{Sp}
\DeclareMathOperator{\Mas}{Mas}
\DeclareMathOperator{\Graph}{Gr}
\DeclareMathOperator{\Lag}{Lag}
\DeclareMathOperator{\im}{im}
\DeclareMathOperator{\dom}{dom}
\DeclareMathOperator{\SF}{sf}
\DeclareMathOperator{\Id}{Id}
\DeclareMathOperator{\Lo}{Long}

\DeclareMathOperator{\Span}{span}
\DeclareMathOperator{\codim}{codim}
\DeclareMathOperator{\const}{constant}

\usepackage{float}

\begin{document}
 \title{The Rabinowitz minimal periodic solution conjecture on partially convex reversible Hamiltonian systems and brake subharmonics}
\author{ Yuting Zhou \thanks{This work was partially supported by the National Natural Science Foundation of China (Grant No. 12301232). }
	\\
 {\em School of Science, Tianjin University of  Technology, Tianjin 300384, P.R. China,}\\
{\em E-mail address: nkzhouyt@mail.nankai.edu.cn}
\\
}
\date{}
\maketitle

\noindent \hrulefill
\medskip
\begin{abstract}\maketitle\baselineskip=18pt
Under weaker regularity and compactness assumptions, we find the mountain-pass essential point, which is a novel  extension of the classical  Ambrosetti-Rabinowitz mountain pass theorem.
We study the reversible superquadratic autonomous Hamiltonian systems whose Hamiltonian $H(p,q)$ is strictly convex in the position $q\in\R^n$ and prove that for every $T>0$, the system has a $T$-periodic brake
solution $\bar x$ with minimal period $T$, provided the Hessian $H_{pp}(\bar x(t))\in\R^{n\times n}$ is semi-positive definite for $t\in\R$ or $n=1$.
For brake subharmonics of general reversible nonautonomous Hamiltonian systems, we also get some new results.
\end{abstract}

\noindent \hrulefill

\tableofcontents

\section{Introduction and Main Results}\label{s:introduction}
\subsection{Extension of the mountain-pass essential point}
In this paper we first give an extended critical point theorem of the improved Ambrosetti-Rabinowitz
mountain pass theorem in \cite{EkHo87}, which is a generalization of the existence of the mountain-pass essential point in loc. cit. and its ``localization" in \cite{GhoPre89}.
Next follows the novel dual action functional and the corresponding quadratic form.
Then we apply them to partially  strictly convex Hamiltonian systems with a certain symmetry, i.e., reversibility, including both autonomous and nonautonomous cases.
Together with new iteration inequalities for Maslov-type indices in Section \ref{s:Maslov-index}, we prove the Rabinowitz minimal period conjecture about
periodic brake solutions of partially convex reversible Hamiltonian systems.

In \cite[Theorem 2.1]{AmRa76}, Ambrosetti and Rabinowitz established  the celebrated mountain pass theory, which shows the existence of critical points for unbounded nonlinear functionals on Banach spaces and consequently can be used to  seek solutions for nonlinear differential equations by variational methods. From then on, many
mathematicians  gave stronger versions  and applications; see \cite{rabinowitz1978,Hofer84,Hofer85} and so on.
In \cite{GhoPre89}, Ghoussoub and Preiss used Ekeland's variational principle instead of the deformation lemma to prove a general mountain pass theory, which gives more information about the location of critical points.

To give our new critical point theorem,
 we first recall the concept of mountain-pass point.
\begin{definition}[cf. {\cite[Definition 1(ii)]{Hofer85}}]\label{d:mountain-pass-point}
	Let $E$ be a real Banach space and let $f\colon E\to \mathbf{R}$ be a continuous and G{\^a}teaux-differential function.
	Let $\bar u$ be a critical point of  $f$, i.e., $f'(\bar u)=0$. We say that $\bar u$ is a mountain-pass point if, for every open neighborhood $\mathcal{U}$ of $\bar u$ in $E$, the set
	\[\{u\in\mathcal{U}| f(u)<f(\bar u)\}\]
	is neither empty nor path connected.
\end{definition}
Denote the set of critical points of $f$ belonging to the critical value $d$ by \[\Cr(f,d)=\{u\in E; f(u)=d\ \  \text{and}\  \  f'(u)=0\}.\]
We also define $M(f,d)=\{u\in\Cr(f,d);u\ \text{is a local minimum for } f \}$.
Denote the closure of a set $S\subset E$ by $\cl(S)$.
Fix a real number $d$ and consider the sublevel set of $f$\[\dot f^d:=\{u\in E; f(u)<d\}.\]
If $e\in \dot f^d$ and $u\in \Cr(f,d)$, we write $e\xrightarrow[f]{} u$ provided $u\in \cl (A)$, where $A$ is the path component of $e$ in $\dot f^d$.

At first, the author found that using a
simple and useful observation in \cite{EkHo87},
we can prove Rabinowitz minimal periodic solution conjecture for strictly convex reversible  Hamiltonian systems via the celebrated mountain-pass essential point theorem in loc. cit. Nevertheless, it seems that the key lemma (\cite[Lemma 2]{Hofer85}) has not been given a complete proof in the literature. 
So the author read \cite[\S 4.1]{Ek90} again, then seeks out a way to prove a more general theorem  needing weaker regularity and compactness. Our  method combines  \cite{Ek90}'s $\S 4.1$ and \cite[Theorem 4.1]{EkHo87}. 
\begin{theorem}[The Generalized Mountain Pass Theorem]\label{t:critical-point-mp}
	Let $f\colon E\to \R$ be a continuous and G{\^a}teaux-differentiable function on a Banach space $E$ such that $f'\colon E\to E^*$
	is continuous from the norm topology of $E$ to the weak$^*$-topology of $E^*$.
	Take two distinct points $(u_0,u_1)$ in $E$, and define
	\begin{gather}
		\Gamma =\{c\in C([0,1];E)| c(0)=u_0,c(1)=u_1\} \nonumber \\
		d=\inf_{c\in \Gamma}\sup_{0\le t\le 1}f(c(t)).\label{e:mountain-pass}
	\end{gather}
	Assume $f$
	satisfies condition {\rm(C)} at the level $d$ (see Definition \ref{d:conditionC}). If
	\begin{equation}\label{e:d-bigger}
		d>\max\{f(u_0),f(u_1)\},
	\end{equation}
	then there is a $\bar u\in \Cr(f,d)$ such that
	\begin{enumerate}
		\item[(a)] $u_0\xrightarrow[f]{} \bar u$,
		\item[(b)] either  $\bar u$ is a mountain-pass point or $\bar u\in \cl(M)\setminus M$, where $M=M(f,d)$.
	\end{enumerate}
\end{theorem}
We recall the pioneering \cite[Definition 1.3]{EkHo87}.
\begin{definition}\label{d:mountain-pass-essential}
	The critical point $\bar u$ found in Theorem \ref{t:critical-point-mp} is called 
	\textit{mountain-pass essential}.
	\end{definition}
	
\subsection{Minimal periodic solution problem for autonomous Hamiltonian systems}
Then we apply our main  critical point theorem to study 
 brake solutions for partially convex reversible Hamiltonian systems. 
Classical tools to deal with  nonlinear Hamiltonian systems are critical points theory and  index theory.
For partially convex reversible systems (we will give the  precise definition soon), we use Theorem \ref{t:critical-point-mp} and new index theory combining the Morse index and Maslov-type indices.
For general systems, we follow the classical method to find brake solutions whose Maslov-type indices are controlled, then apply our new index iteration inequalities to get new results about brake subharmonic solutions.

Consider the Hamiltonian system
\begin{equation}\label{e:HS}
	\dot{x}=JH'(x),\ \ \ J=\left(
	\begin{array}{cc}
		0 & -I_n \\
		I_n & 0 \\
	\end{array}
	\right),
\end{equation}
where $x\in\R^{2n}$ and the Hamiltonian function $H\in C^1(\R^{2n},\R)$.
In the pioneering work \cite{rabinowitz1978} of 1977, P. Rabinowitz proved
the existence of periodic solutions for the superquadratic Hamiltonian systems.
Suppose $H\in C^1(\R^{2n},\R)$ satisfies
\begin{enumerate}
	\item [(H1)] $H\geq 0$,
	\item [(H2)] $H(x)=o(|x|^2)$ as $|x|\to 0$, and
	\item [(H3)] there are constants $\mu>2$ and $\bar r>0$ such that
	$0<\mu H(x)\le x\cdot H'(x)$ for all $|x|\ge \bar r$.
\end{enumerate}
Then, for any $T>0$, \eqref{e:HS} possesses a non-constant $T$-periodic solution.

Note that $T/k$-periodic solutions, $k\in\N$, are also $T$-periodic.
In the same paper \cite{rabinowitz1978},
P. Rabinowitz conjectured that, under the above same hypothesis,
there is a non-constant solution of \eqref{e:HS} having any prescribed minimal period.
The conjecture is still open.
The first result on this question was
given by Ambrosetti and Mancini in \cite{amb-V81} of 1981.
They proved minimality under the added assumptions that $H$ is strictly convex and
the Fenchel conjugate $H^*$ of $H$ satisfies another condition.
A more general breakthrough for strictly convex Hamiltonian systems was made
by Ekeland and Hofer in \cite{EkHo85} of 1985.
More precisely, assuming $H$ satisfies
\begin{enumerate}
	\item [(H0)$^+$:] $H\in C^2(\R^{2n},\R)$ and the Hessian $H''(x)$ is positive definite for any $x\in\R^{2n}\setminus\{0\}$,
\end{enumerate}
and (H1)--(H3); they prove that,
for any $T>0$, \eqref{e:HS} possesses a periodic solution with minimal
period $T$.
In the 1990s, Professor Long developed Maslov-type iteration theory and used the direct variational to reprove Ekeland-Hofer's result under a weaker condition than  (H0)$^+$ with his students, for which please see \cite[Chapters 10 and 13]{Lo}.

Let $N=\left(
\begin{array}{cc}
	-I_n & 0 \\
	0& I_n \end{array}
\right)$,
where $I_n$ is the $n\times n$ identity matrix.
In this article, we study reversible systems, i.e., 
\begin{enumerate}
	\item [(H4)] $H(Nx)=H(x)$ for any $x\in\R^{2n}$.
\end{enumerate}

Given $T>0$, we consider the following  
problem 
\begin{equation}\label{e:HS-T}
	\begin{cases}
		\dot x(t)=JH'(x(t)), \\
		x(-t)=Nx(t),\\
		x(t+T)=x(t),
	\end{cases}
\end{equation}
for all $t\in\R$.
A $T$-periodic solution $(x,T)$ of \eqref{e:HS-T} is called a \textit{brake solution}, which is just a brake orbit if Hamiltonians $H$ are
of the forms
\begin{equation*}
	H(p,q)=\frac{1}{2}|p|^2+V(q)\qquad
	V\in C^1(\R^n,\R).
\end{equation*}
At $t=0,\frac{T}{2}$, the momentum of $T$-periodic brake orbits is zero, the typical examples including 
a simple pendulum and oscillations of a spring.
Brake solutions were called symmetric periodic orbits in \cite{FravanKoe}, where they extend our setting the standard symplectic vector space $(\R^{2n},\omega_0)$ to any symplectic manifold $(M,\omega)$ with an antisymplectic involution of $M$. Note that  $N^2=I_{2n}$ and $N^TJN=-J$.
\subsubsection{Main results}
We will prove
\begin{theorem}\label{t:super-mini-brake}
	Suppose $H\in C^2(\R^{2n},\R)$ satisfies {\rm (H1)--(H4)} and 
	\begin{enumerate}
		\item [\rm (H0)$^{q+}$:] The $n\times n$ matrix $H_{qq}(x)=(H_{q_iq_j}(x))_{n\times n}$ is positive definite for any nonzero vector $x=\begin{pmatrix}
			p\\
			q
		\end{pmatrix}\in\R^{2n}$, where $p=(p_1,p_2,\cdots,p_n)^T\in \R^n,q=(q_1,q_2,\cdots,q_n)^T\in \R^n$.
	\end{enumerate}
	Then for any $T>0$, \eqref{e:HS-T} has a periodic brake solution $\bar x$ with minimal period $T$ provided
	this solution $\bar x$  further satisfies either 
	\begin{enumerate}  
		\item  [\rm (H0)$^{p\ge0}$:]The $n\times n$ matrix  $H_{pp}(\bar x(t))$ is semi-positive definite for $t\in \R$, 
		 where 
		  $H_{pp}(x)=(H_{p_ip_j}(x))_{n\times n}$  for
		   $x=\begin{pmatrix}
		 	p\\
		 	q
		 \end{pmatrix}\in\R^{2n}$,
		 \item[or] $n=1$.
	\end{enumerate}
\end{theorem}

As in the celebrated  \cite{EkHo85}, our proof also use dual 
action principle. Since at first we just assume partial strict convexity {\rm (H0)$^{q+}$} of $H$, we can not directly apply their proof to our case. The difficulty is that by adding a quadratic form to the Hamiltonian $H$, we can still prove the adjusted dual action functional satisfies the conditions of mountain-pass Theorem \ref{t:critical-point-mp}.
In \cite{liu2010}, C. Liu proved that
for any $T>0$, \eqref{e:HS-T} has a $T$-periodic brake solution whose
minimal period is $T$ or $\frac{T}{2}$ under the assumptions of  {\rm (H1)--(H4)} and {\rm (H0)$^+$}.
In \cite{zhang2015}, assuming both that $H''(x)$ is semi-positive definite and {\rm (H0)$^{q+}$} instead of {\rm (H0)$^+$}, D. Zhang got the same result.
In both  Liu and Zhang's proof, they used the well-known variational principle and the iteration inequalities of the Maslov-type indices. A Maslov-type index of a path of symplectic matrices  or a periodic solution with some boundary condition is an integer, 
the precise definition of which will be given in Section \ref{ss:Maslov-index-iteration}.
Both in  Liu-Zhang's and our proof, we always need  the Maslov-type indices with brake symmetric and periodic boundary conditions (cf. Definition  \ref{d:Maslov-type-index}) satisfy
\begin{equation}\label{e:index-big-0}
	i_{L_0}(\bar x)\ge 0, \quad \text{and} \quad i^{\Lo}_1(\gamma)+\nu_1(\gamma)\ge n+1,
\end{equation}
 where $L_0=\{0\}\times \R^n$ and $\gamma$ is the matrizant of the variational equation along the  trajectory $\bar x(t)$ of \eqref{e:HS-T}. 
In \cite{FZZZ22}, assuming {\rm (H1)--(H4)} and 
$H''(x)$ is semi-positive definite for $x\in\R^{2n}$, we 
got that for any $T>0$, the system \eqref{e:HS-T}  possesses a non-constant $T$-periodic  brake solution  with minimal period not less than  $\frac{T}{n+1}$.

However their methods can not eliminate $\frac{T}{2}$, half the period; please see the example about the Maslov-type indices of the special symplectic path in \cite{zhang2015}.

More naturally,  we give the reversible version of 
\cite[Threorem 2]{CZ-E-L90}
\begin{theorem}\label{t:potential-well1}
	Let $\Omega$ be an open subset of $\R^{2n}$ with the following symmetry property,
	that is, $N\Omega=\Omega$.
	Assume $H\in C^2(\R^{2n},\R)$ satisfies {\rm (H1)-(H2),(H4)}, {\rm (H0)$^{q+}$} and 
	\begin{enumerate}
		\item [\rm (H3)$'$] The smallest eigenvalue of $H''(x)$ goes to $+\infty$ as $|x|\to \infty$ or $x\to \partial \Omega$.
	\end{enumerate}
	Then, for any $T>0$, \eqref{e:HS-T} has a periodic brake solution $\bar x$ with minimal period $T$ provided
	this solution $\bar x$  further satisfies either {\rm (H0)$^{p\ge 0}$} or $n=1$.
\end{theorem}
\subsubsection{Methods and new observations}
Assuming $H_{qq}(p,q)\in \R^{n\times n}$ are positive definite, we define a modified dual functional $\psi$ and 
find a mountain-pass essential point $\bar u$ of it, which corresponds to a non-constant  $T$-periodic brake solution $\bar x$ of the Hamiltonian system.
Moreover the Morse index  and the Maslov-type index of $\bar x$ are the same, which is less than one.
Under the additional assumption of 
{\rm (H0)$^{p\ge0}$} or $n=1$, and by Maslov-type index
theory, we can prove that:
if the Morse index of $\bar u$ is zero, then $T$ is the minimal period of $\bar x$; if the Morse index of $\bar u$ is one, then  the minimal period of $T$-periodic solution $\bar x$ is $T$ or $\frac{T}{2}$.
For the second case, using the topological structure near $\bar u$ on the sublevel set of $\psi$, we prove half a period $\frac{T}{2}$ is impossible. 
Thus we get our result.

In 2021, the author finished the subharmonic part Theorem \ref{t:super-quadratic-1} without writing the remaining partially convex {\bf(SH7)} case, in which I found that time shift may change the Maslov-type indices with Lagrangian boundary conditions and used the Bott-type iteration formulae to deal with this problem. The author's that unpublished paper has also included the following  important iteration inequality \eqref{e:iteration-ineq} with its proof \eqref{e:iteration-inequality}, 
Corollary \ref{0-1}, and all the lines from Lemma \ref{l:shift-index} to Remark \ref{r:time-shift-change}.
In 2022, I found we can eliminate the half period in the minimal periodic problem on brake solutions of reversible convex Hamiltonian systems by using the method in \cite{EkHo87}, which used to tackling nonautomomous systems; and also constructed the convex $C^2$ truncated Hamiltonians $\hat H_K$ in \eqref{e:HS-T-K} and $\tilde H$ \eqref{e:tilde-H}.
At first, the author did not want to use Maslov-type index iteration theory; then found it is impossible to prove
the reversible version of \cite{EkHo85} just via the Morse index.
At the beginning of 2023, I wrote the complete proof for the case $m^-(\bar u)=1$ to get rid of half the period; in May of the same year, to solve the case $m^-(\bar u)=0$, I wrote the earlier version of Lemmas \ref{l:i-nu-L0-L1} and
\ref{l:i=0}, which was inspired by my Lemma \ref{l:shift-index} and  Corollary \ref{c:shift-two-times-index}.
Over the summer break of 2024, the author thought a new method to handle the partially convex case by the novel adjusted dual functional \eqref{e:reduced-functional}.

At the beginning of May 2025, I proposed the weaker condition {\rm (H0)$^{p\ge0}$} or $n=1$ and did my best to  to prove our main Theorem \ref{t:super-mini-brake}. 
For using iteration inequalities \eqref{e:iterarion-even}, \eqref{e:iteration-ineq} and \eqref{e:i-L0-L1}, we need
the following \eqref{e:i-L1+nu-L1-ge1} holds, i.e.,
\[i_{L_1}(\gamma)+\nu_{L_1}(\gamma)\ge 1,\]
which cannot be proved just by the additional new condition {\rm (H0)$^{p\ge0}$} or $n=1$.
Besides Maslov-type index and H\"{o}rmander index theory, 
the proof needs  more knowledge of the periodic solution $\bar x$.
For $n=1$, we will combine the index estimate \eqref{e:morse-index-less1-greater1} and Proposition \ref{p:C=0}.
For {\rm (H0)$^{p\ge0}$}, we will combine the relative Morse index, the sign of the critical value of the action functional and {\rm (H1)}. Note that the following Lemmas \ref{l:relative-morse-index} and \ref{l:ker-A-sB-constant}
about semi-positive definite $B(t)=H''(x(t))$  were first given by the author in the original submission of \cite{FZZZ22}. Using the two lemmas, the following $(L_0,L_1)$-concavity of $\gamma$ \eqref{e:i-nu-L0-L1} and 
the direct action functional 
\begin{equation}\label{e:phi-autonomous}
\Phi(x)=\int_{-\frac{T}{2}}^{\frac{T}{2}}\left[\frac{1}{2}(J\dot x,x)+H(x)\right]dt,
\end{equation}
we will prove the desired  \eqref{e:i-L1+nu-L1-ge1} in Lemma \ref{l:i+nu-ge-1}.
In fact, in our case, we do not assume 
$B_2(t)=H''(\bar x(t))$ is semi-positive definite.
However, we can decompose
\[B_2(t)=B_1(t)+[B_2(t)-B_1(t)],\] 
where $B_2-B_1$ is semi-positive definite and  both $B_1(t)$ and $B_2(t)-B_1(t)$ have the particular matrix forms which
can be handled.

\subsection{Multiple subharmonic solutions for nonautonomous Hamiltonian systems}
We also study 
  nonautonomous Hamiltonian systems
\begin{equation}\label{e:HS-nonauto}
	\dot{x}=JH'(t,x),
\end{equation}
where $x\in\R^{2n}$, $H$ is $T$-periodic in $t\in\R$ and $T>0$. By $H'(t,\cdot)$ we denote the partial gradient with respect to $x\in\R^{2n}$. Then it is nature to seek $T$-periodic solutions
of \eqref{e:HS-nonauto}. Since $H$ is $jT$ periodic for all $j\in\N$, one can also search for $jT$-periodic solutions, which are called \textit{subharmonics}:
\begin{equation*}
	\begin{split}
		&\dot x(t)=J H'(t,x(t)) \\
		&x(t+jT)=x(t)
	\end{split}
\end{equation*}
for all $t\in\R$.
Clearly, any $T$-periodic solution is $jT$-periodic for all $j\in\N$.
Thus an additional argument is required to show that these subharmonics  which we found are indeed distinct.

In \cite{Rabin80}, assuming the following conditions,
\begin{enumerate}
	\item[{\bf(SH1)}] $H(t,x)\ge 0$  for all $t\in\R$, $x\in \R^{2n}$;
	\item[{\bf(SH2)}]  $H(t,x)=o(|x|^2)$   as $|x|\to0$,
	\item [{\bf(SH3)}] there exist constants  $\mu>2$  and $\bar{r}>0$ such that
	\begin{equation*}
		0<\mu H(t,x)\le x\cdot  H'(t,x),\quad  \text{for all} \quad t\in\R, |x|\ge \bar{r}\/ ;
	\end{equation*}
	\item[{\bf(SH4)}] there is a $T>0$ such that $ H(t+T,x)= H(t,x)$ for all $t\in\R$, $x\in\R^{2n}$\/;
	\item[{\bf(SH5)}] there are constants $\theta,R_1>0$ such that $|H'(t,x)|\le \theta x\cdot  H'(t,x)$ for all
	$t\in\R$, $x\in \R^{2n}$, $|x|>R_1$\/;
\end{enumerate}
Rabinowitz shows the existence of infinitely many distinct subharmonics. 
About infinite subharmonics on closed symplectic aspherical  manifolds, there was a conjecture formulated by Charles Conley in 1984, for the details of which we refer the reader to the celebrated 
\cite{SaZeh92}, \cite{Hin09} and \cite{Ginz10}.
Given an integer $j\in\Z$ and a $kT$-periodic function $x_k$, denote by $j*x_k$ the
phase shift, defined by
\[(j*x_k)(t)=x_k(t+jT).\]
For a $kT$-periodic solution of \eqref{e:HS-nonauto}, which is denoted by $(x_k,kT)$,
if
\begin{equation*}
	j*x_k\neq x_k \text{ for all }\quad j\neq 0 \mod k,
\end{equation*}
then
we  say that $kT$ is a \textit{simple} period of $x_k$.
Given two  periodic solutions $(x_j,jT)$  and $(x_k,kT)$ of \eqref{e:HS}, we shall say they are \textit{geometrically distinct} if
\[l*x_j\neq h*x_k \text{ for all } l,h\in\Z.\]

 In order to utilize Morse theory or the Maslov-type indices, we assume
\begin{enumerate}
	\item[{\bf(SH0)}] $H\colon \R\times \R^{2n}\to \R^{2n}$, $H$ possesses a second partial derivative with respect to $x\in\R^{2n}$ such that
	$H$, $H_{x_i}$, $H_{x_ix_j}\in C(\R\times \R^{2n},\R)$ for $i,j\in\{1,2,\ldots,2n\}$.
	\end{enumerate}
Further assuming the strictly convexity of $H(t,x)$ about $x$ and  other growth conditions about the gradients of $H(t,\cdot )$ and its Fenchel conjugate $H^*(t;\cdot)$,
I. Ekeland and H. Hofer in \cite{EkHo87} proved that \eqref{e:HS-nonauto}
possesses subharmonic $jT$-periodic solution $x_j$ for each integer $j\ge1$ and all of them are pairwise \textit{geometrically distinct}; and for almost such $H$ (in the sense of Baire), $jT$ is a simple period of $x_j$ for $j\in\N$.
Without the convex assumption of $H$, in \cite{Liu00} by using Maslov-type index iteration theory, C. Liu proved that for each integer $j\ge 1$, there exists a non-constant $jT$-periodic solution $x_j$ of \eqref{e:HS-nonauto} such that for any integer $p>2n+1$, $x_j$ and $x_{pj}$
are geometrically distinct; while in \cite{Zhou25sub}, the author improved to $p>2n$. For $p>1$, if $x_{pj}$ is nondegenerate, then $x_j$ and $x_{pj}$ are geometrically distinct. Note that 
 a $jT$-periodic solution $(x_j,jT)$ is called nondegenerate if $\nu_{1}(x_j)=0$.

We assume $H$ satisfies the reversible condition:
\begin{enumerate}
	\item[{\bf(SH6)}]  $H(t,x)=H(-t,Nx)$ for all $t\in\R$, $x\in\R^{2n}$.
\end{enumerate}
In \cite{LiLiu10}, assuming the Hessian of $H$ is positive definite, some results about \textit{brake subharmonics} similar to \cite{Liu00}
were obtained.
Actually, using direct action functional, they added a semi-positive quadratic form 
$\frac{1}{2}\tilde B(t)x\cdot x$ to the superquadratic Hamiltonian $H$.

More precisely, for $\tilde B \in C(S_T, \mathcal{L}_s(\R^{2n}))$, we denote the $C^0$-norm by $\|\tilde B\|_{C^0}$, where
$S_T=\R/T\Z$ and $\mathcal{L}_s(\R^{2n})$ is the set of symmetric $2n\times 2n$ real matrices. 
We define
\begin{enumerate}
	\item[{\bf(B1)}] $\tilde B(t)$ is semi-positive definite for all $t\in\R$,
\end{enumerate} 
and  the reversible condition:
\begin{enumerate}
	\item[{\bf(B2)}] $N\tilde B(t)N=\tilde B(-t)$ for $t\in\R$.
\end{enumerate}

In the literature, to tell apart geometrical distinct brake subharmonics,  they always need
assume $H(t,x)$ was convex about $x$. 
However, because of our new iteration inequalities, we can study  brake subharmonics for general systems and get better results.
For $j\in\N$, consider  
\begin{equation}\label{e:HS-jT}
	{\rm (HS)_j}:\qquad
	\left\{
	\begin{array}{l}
		\dot x(t)=J  H'(t,x(t)),\\
		x(-t)= Nx(t),\\
		x(t+jT)=x(t),
	\end{array}
	\right.
\end{equation}
for all $t\in\R$.
Solutions  of {\rm (HS)$_j$}, for $j\ge 2$, are  called  
\textit{brake subharmonics}.
A periodic solution $(x_j,jT)$ of {\rm (HS)$_j$} is called nondegenerate if $\nu_{L_0}(x_j)=0$, to be defined in \eqref{e:Maslov-index-L_0}. 
We will prove
\begin{theorem}\label{t:super-quadratic-1}
	Suppose $ H(t,x)=\frac{1}{2}\tilde B(t)x\cdot x+ \tilde H(t,x)$ with $\tilde H$ satisfying {\bf (SH0)}--{\bf (SH6)}, $\tilde B\in C(S_T, L_s(\R^{2n}))$ satisfying {\bf (B1)}-{\bf (B2)}.
	Then for any $j\in\N$ and $1\le j< \frac{2\pi}{T\|\tilde B\|_{C^0}}$, the Hamiltonian system \eqref{e:HS-jT} possesses a non-constant $jT$-periodic solution $x_j$
 satisfies
	\begin{enumerate}
		\item for $j \in (2\N-1)\cup(4\N-2)$, $p\in \N$ and $pj< \frac{2\pi}{T\|\tilde B\|_{C^0}}$, $x_j$ and $x_{pj}$ are geometrically  distinct provided $p>n+1$.
		Furthermore, if  $x_j$ is nondegenerate,
		then for $p>1$, $x_j$ and $x_{pj}$ are geometrically
		distinct.
		\item For $j \in 4\N$, $p\in \N$ and $pj< \frac{2\pi}{T\|\tilde B\|_{C^0}}$, $x_j$ and $x_{pj}$ are geometrically  distinct provided $p>2n+1$.
		Furthermore, if $x_j$ is nondegenerate,
		then for $p>n+1$, $x_j$ and $x_{pj}$ are geometrically
		distinct.
	\end{enumerate}
 If we furthermore assume 
	\begin{enumerate}
		 \item[{\bf (SH7)}]  The $n\times n$ matrix $H_{qq}(t,x)$ is positive definite for any $t\in \R$ and  $x=\begin{pmatrix}
			p\\
			q
		\end{pmatrix}\in\R^{2n}\setminus\{0\}$,
		\end{enumerate}
	 then we have
	\begin{enumerate}
		\item[Case 1:] Assume $j \in (2\N-1)\cup(4\N-2)$, $p\in \N$ and $pj< \frac{2\pi}{T\|\tilde B\|_{C^0}}$.
		Then
		$x_j$ and $x_{pj}$  are  geometrically distinct provided
		$p>2$. 
		\item[Case 2:]
		Assume $j\in 4\N$, $p\in \N$ and $pj< \frac{2\pi}{T\|\tilde B\|_{C^0}}$. Then
		$x_j$ and $x_{pj}$  are  geometrically distinct provided $p>n+2$.
	\end{enumerate}
\end{theorem}
In Theorem \ref{t:super-quadratic-1}, $\tilde B(t)$ can be 
constant zero matrix. If $\tilde B(t)\equiv 0$,
then $\frac{2\pi}{T\|\tilde B\|_{C^0}} =+\infty$.
Moreover, we have
\begin{theorem}\label{t:super-quadratic-2}
	Suppose that $H(t,x)$ satisfies {\bf (SH0)}-{\bf (SH7)}.
	Then for any $j\in\N$, the Hamiltonian system \eqref{e:HS-jT} possesses a non-constant $jT$-periodic solution $x_j$
	which satisfies
	\begin{equation*}
		i_{L_0}(x_j)\leq 1\le  i_{L_0}(x_j)+\nu_{L_0}(x_j).
	\end{equation*}
		\begin{enumerate}
		\item[Case 1:]Assume $j \in (2\N-1)\cup(4\N-2)$ and  $p\in \N$. Then  $x_j$ and $x_{pj}$ are geometrically  distinct provided $p>1$.
		\item[Case 2:]
		Assume $j=2^l$ for some $l\in \N$. If $x_j$ is nondegenerate, then $jT$ is a simple period
		of $x_j$.
	\end{enumerate}

\end{theorem}

This paper consists of five sections and five appendixes.
Section \ref{s:introduction} presents the new critical point theory and our applications on reveisble Hamiltonian systems.
In Section \ref{s:critical-point}, 
first we prove our extended Mountain Pass Theorem \ref{t:critical-point-mp}. To use it prove Theorems \ref{t:super-mini-brake}, \ref{t:potential-well1} and \ref{t:super-quadratic-2}, we then do some preparations.
In Section \ref{ss:dual-action}, we establish the variational framework and prove that the new functional still satisfies the conditions 
of Theorem \ref{t:critical-point-mp}. In Section \ref{ss:morse-index-topology}, we define and estimate the Morse index of the critical point found in Section \ref{ss:critical-point}, the geometrical structure near which 
in also studied similar to \cite{EkHo87}.
In Section \ref{ss:Maslov-index-iteration}, we recall Maslov-type index
theory and give some new iteration inequalities and formulae including the H\"{o}rmander index.
As a direct application of our new iteration inequalities, in Section \ref{ss:subharmonic}, we prove 
Theorem \ref{t:super-quadratic-1} by using the well-known direct variation.
We will prove the equality of the Morse index and the Maslov-type index for the brake solution of partially 
convex reversible Hamiltonian systems in
Section  \ref{ss:Morse-Maslov-indices}. We also give a formula for the relative Morse index on the semi-positive definite case in Section \ref{ss:relative-morse-index}, which will be utilized to deal with minimal periodic problems under our novel weaker condition {\rm (H0)$^{p\ge0}$}. 
In Section \ref{s:partial-convex-symmetric}, we prove Theorems \ref{t:super-quadratic-2},\ref{t:super-mini-brake} and  \ref{t:potential-well1}.
Our approach here is a blend of the mountain-pass critical point theory, Morse theory and the Maslov-type indices. We also give an \textit{a priori} estimate of the solutions we found.
The first three appendixes are more novel.
We use
the classical Morse index theorem of the quadratic form in \ref{a:morse} to prove the equivalence of the Morse index and Maslov-type index of the critical point for our new adjusted functional. 
In \ref{a:hormander}, we calculate the
H\"{o}mander index explicitly. In \ref{app:continuity-eigemvalue}, we give a self-contained knowledge of the eigenvalues of the  self-adjoint Fredholm operators as the preparations for 
the spectral flow and the relative Morse index in Section \ref{ss:relative-morse-index}.

Carefully studying and developing the methods in \cite{EkHo85},
\cite{EkHo87} and \cite{GhoPre89}, we prove our theorems by using the dual action principle, the relative Morse index, the critical values of the critical points, convex analysis, the critical points of the  mountain-pass essential type, \textit{a priori} estimate of solutions,
the Morse index of the periodic solutions and iteration theorem of the Maslov-type indices.

\section{Study of the critical point}\label{s:critical-point}
\subsection{Existence of the mountain-pass essential point}\label{ss:critical-point}
Let $E$ be a real Banach space. For a slightly better result, we change the metric on $E$.
Just as in \cite[Chapter IV.1]{Ek90}, define the geodesic length $l(c)$ of a curve $c\in C^1([0,1];E)$ to be:
\[l(c):=\int^1_0\frac{||\dot c(t)||}{1+||c(t)||}dt.\]
Then define the geodesic distance $\delta$ between two points $x_1$ and $x_2$ in $E$ to be
\[\delta(x_1,x_2):=\inf \{l(c);c\in C^1([0,1];E),c(0)=x_1,c(1)=x_2 \}.\]
Clearly $\delta (x_1,x_2)\le \|x_1-x_2\|$. In fact, $(E,\delta)$ is a complete metric space.
We denote the dual space of $E$ by $E^*$ with the operator norm $\|\cdot\|_*$.
\begin{definition}\label{d:conditionC}
	We shall say that a G{\^a}teaux-differentiable function $f\colon E\to \R$ satisfies condition (C)
	at the level $d$
	if every sequence $x_n$ such that
	\[ f(x_n)\to d, \quad (1+\|x_n\|)\|f'(x_n)\|_*\to 0 \quad { as }\quad n\to +\infty,\]
	has a convergent subsequence, which converges to a critical point $\bar x$ of $f$ with $f(\bar x)=d$.
\end{definition}
\begin{definition}
	We shall say that a $C^1$ function $f\colon E\to \R$ satisfies condition (PS) if every sequence $\{x_n;n\in\N\}$
	such that $|f(x_n)|$ is uniformly bounded and $f'(x_n)\to 0$ in $E^*$ has a convergent sequence.
\end{definition}
By definition, if a function $f$ satisfies condition (PS), then it must satisfy condition (C).
We recall \cite[Theorem IV.1.12]{Ek90}, which will be used to prove our main critical point theorem.
\begin{proposition}\label{p:critical-point-mp}
	Let $f\colon E\to \R$ be a continuous and G{\^a}teaux-differentiable function on a Banach space $E$ such that $f'\colon E\to E^*$
	is continuous from the norm topology of $E$ to the weak$^*$-topology of $E^*$.
	Take two distinct points $(u_0,u_1)$ in $E$, and define
	\begin{gather*}
		\Gamma =\{c\in C([0,1];E)| c(0)=u_0,c(1)=u_1\},  \text{ and } \\
		d=\inf_{c\in \Gamma}\sup_{0\le t\le 1}f(c(t)).
	\end{gather*}
	Suppose $F$ is a closed subset of $E$ such that $F_d:=F\cap \{u\in E;f(u)\ge d\}$ separates $u_0$ and $u_1$, i.e., $u_0$ and $u_1$ belong to disjoint connected components in $E\setminus F_d$. Assume $f$
	satisfies condition {\rm(C)} at the level $d$. Then either $F\cap \cl(M(f,d))\neq\emptyset$
	or $F\cap \Cr(f,d)$ contains a mountain-pass point.
\end{proposition}
\begin{remark}
	By using the metric space $(E,\delta)$,
	Proposition \ref{p:critical-point-mp} improved
	\cite[Theorem (1.ter)(a)]{GhoPre89}, where they assume $f$ verifies the stronger condition (PS) in lieu of (C).
	Ekeland used a partition of unity and the metric space $(X,\delta)$ in the proof of \cite[Theorem IV.1.6 (Ghoussoub-Preiss)]{Ek90}.
	Both the proofs need Ekeland's variational principle.
\end{remark}

The classical theorem of Ambrosetti-Rabinowitz gives the existence of a critical value, thus a critical point.
In \cite{GhoPre89}, N. Ghoussoub and D. Preiss
formulate a more general principle which gives some information  about the critical point's location.
We reprove and extend \cite[Theorem 1.4]{EkHo87} to our Theorem \ref{t:critical-point-mp}.
\begin{proof}[Proof of Theorem \ref{t:critical-point-mp}]
	Since $E$ is locally path connected, the connected components and the path components are the same.
	Recall $\dot f^d=\{u\in E; f(u)<d\}$ and $u_0,u_1\in \dot f^d$.
	Denote by $\mathcal{L}$ the path component of $\dot f^d$ containing $u_0$.
	Note that $\dot f^d$ is an open subset of $E$, thus $\mathcal{L}$ is open in $E$.
	Set $\partial \mathcal{L}=\cl(\mathcal{L})\setminus \mathcal{L}$, i.e., the  boundary of $\mathcal{L}$. We claim that $\partial \mathcal{L}$ separate $u_0$ and $u_1$ in $E$, that is, $u_0$ and $u_1$ belong to disjoint connected components of $E\setminus \partial\mathcal{L}$.

	In fact, pick any $c\in\Gamma$ and define $t_0\in[0,1]$ by
	\[t_0=\sup\{t\in [0,1];c(t)\in \mathcal{L}\}.\]
	Since $u_0\in \mathcal{L}$ and $\mathcal{L}$ is open,
 we have $t_0>0$; and by \eqref{e:d-bigger}, we have $t_0<1$. We claim that $c(t_0)\in \partial \mathcal{L}$.
	By definition, $c(t_0)\in \cl(\mathcal{L})$; since $\mathcal{L}$ is open, $c(t_0)\notin\mathcal{L}$.
	Thus $c(t_0)\in \partial \mathcal{L}$, and we get that there is no path in $E\setminus \partial\mathcal{L}$ from $u_0$ to $u_1$.
	
	Then we prove that $\partial \mathcal{L}\subset\{u\in E;f(u)=d\}$.
	Since $\partial \mathcal{L}\subset \cl(\mathcal{L}) $, we have  $\partial \mathcal{L}\subset\{u\in E;f(u)\le d\}$.
	Argue by contradiction: if there were an element $w\in \partial \mathcal{L}$ such that $f(w)<d$, then there would exist a connected open neighborhood $V$ of $w$ such that  $V\subset \dot f^d$. Since $w\in \cl(\mathcal{L})$, $V\cap \mathcal{L}\neq \emptyset$. By the connectedness of $V$ and the definition of $\mathcal{L}$, we get that $w\in V\subset \mathcal{L}$, which yields a contradiction.
	
	Thus $\partial\mathcal{L}=\partial\mathcal{L}\cap \{u\in E;f(u)\ge d\}$ separates $u_0$ and $u_1$. Applying Proposition \ref{p:critical-point-mp}, we get that there exists a $\bar u\in \partial\mathcal{L}$ such that either  $\bar u\in  \cl(M(f,d))$
	or $\bar u$ is a mountain-pass point. By the definition of $\mathcal{L}$, we get $u_0 \xrightarrow[f]{} \bar u$, which implies $\bar u\notin M(f,d)$. The proof is complete.
\end{proof}

\subsection{Application  of Theorem \ref{t:critical-point-mp} to partially convex reversible Hamiltonian systems}\label{ss:dual-action}

Fix a positive number $T$ and denote the interval $I_T:=[-\frac{T}{2},\frac{T}{2}]$.
For any $p>1$, define \[\hat L^p(I_T;\R^{2n})=\{u\in L^p(I_T;\R^{2n}); u(-t)=Nu(t) \text{ for } t\in I_T\}\]
with norm  $\|u\|_{p}=\left(\int_{-\frac{T}{2}}^{\frac{T}{2}}|u(t)|^p dt\right)^{\frac{1}{p}}$, and 
\[ \hat W_T^{1,p}(I_T;\R^{2n})=\{x\in  W^{1,p}(I_T;\R^{2n});x(-t)=Nx(t) \text{ for } t\in I_T \text{ and } x(\frac{T}{2})=x(-\frac{T}{2})\}.\]
Note for $p>1$, if  $x\in  W^{1,p}(I_T;\R^{2n})$, then $x$ is absolutely continuous.
For $\lambda\ge 0$, we define an operator $A_{\lambda}\colon \hat L^{\mu}(I_T;\R^{2n})\to \hat L^{\mu'}(I_T;\R^{2n})$ by $A_{\lambda}x=-J\dot x+\Lambda x$ with domain $\dom(A_{\lambda})=\hat W_T^{1,\mu'}(I_T;\R^{2n})$, where $\Lambda=\begin{pmatrix}
	\lambda I_n&0\\
	0&0
\end{pmatrix}$, $\lambda\ge 0$ and  $\mu'$ satisfies $\frac{1}{\mu'}+\frac{1}{\mu}=1$.
Since we have assumed $\mu>2$ in {\rm (H3)} and {\bf(SH3)}, one gets $1<\mu'<2$.
Similar to \cite[Lemma II.4.4]{Ek90},
$A_{\lambda}$ is a closed self-adjoint operator from $\hat L^{\mu}(I_T;\R^{2n})$ to $\hat L^{\mu'}(I_T;\R^{2n})$.
Since $x(-t)=Nx(t)$, where $N=\begin{pmatrix} -I_n&0\\
	0&I_n
\end{pmatrix}
$, we have $\ker A_{\lambda}=\{\xi\in \R^{2n};\xi=\begin{pmatrix}0\\
	b\end{pmatrix}, b\in \R^n\}$.
	In this paper, we denote the inner product of $x,y\in\R^{2n}$  by $x\cdot y$ or $(x,y)$.
	We also use the notation $(x,y)(t):=x(t)\cdot y(t)$, $t\in \R$, for a pair of paths in $\R^{2n}$, that is, $x(t),y(t)\in \R^{2n}$.
For $x=\begin{pmatrix}
	p\\
	q
\end{pmatrix}\in\R^{2n}
$, $\Lambda x=\begin{pmatrix}
	\lambda p\\
	0
\end{pmatrix}\in\R^{2n}$ and $(\Lambda x,x)=\lambda |p|^2$,
which will be used in the adjusted dual functional later.

The direct action functional of  \eqref{e:HS-jT} for $j=1$ is defined by
\[\Phi(x)=\int_{-\frac{T}{2}}^{\frac{T}{2}}\left[\frac{1}{2}(J\dot x,x)+H(t,x)\right]dt\]
on $\hat W_T^{1,\mu'}(I_T;\R^{2n})$.
To get a  $C^1$ and convex functional, $H$ will be modified.
Choose a positive integer $m$ such that $2<2+\frac{1}{m}\le\mu$ and
denote by $\alpha =2+\frac{1}{m}$.
Let $K>\bar r$ and $\chi\in C^{\infty}(\R,\R)$ such that
$\chi(y)=1$ if $y\le K+1$, $\chi(y)=0$ if $y\ge K+2$, and $\chi'(y)<0$
if $y\in (K+1,K+2)$.
Set
\[H_K(t,x)=\chi(|x|)H(t,x)+(1-\chi(|x|))R|x|^{\alpha},\]
where $R\ge \displaystyle\max_{K+1\le|x|\le K+2,t\in I_T}\frac{H(t,x)}{|x|^{\alpha}}$.
Then $H_K$ satisfies {\bf (SH0)--(SH7)} with the same $\bar r$ and  $\mu$ replaced by $\alpha$
in {\bf(SH3)},  $\theta$ replaced by $\bar \theta= \max\{\theta,1\}$ and $R_1$ replaced by $\bar R_1=\max\{R_1,1\}$ in  {\bf(SH5)}.
Note that there exists a positive constant $M_K$ such that
$H''_K(t,x)y\cdot y\ge -M_K |y|^2$ for $K+1\le |x|\le K+2$.
Define a special convex $C^2$ function $l$ on $\R^{2n}$ by
\begin{equation*}
	l(x)=\left\{
	\begin{array}{ll}
		(|x|^2-2K|x|+K^2)^{\frac{\alpha}{2}}, & \hbox{for $|x|\ge K$;} \\
		0, & \hbox{for $|x|\le K$.}
	\end{array}
	\right.
\end{equation*}
For $|x|\ge K$, let $j(x)=(|x|-K)^2$.
Direct calculation shows that the Hessian \[j''(x)=2\frac{|x|-K}{|x|}I_{2n}+2\frac{K}{|x|^3}x\cdot x^T.\]
Then we obtain
\[l''(x)y\cdot y\ge \alpha j^{\frac{\alpha}{2}-1}(x)\frac{|x|-K}{|x|}|y|^2 \qquad\text{ for } |x|\ge K.\]
So for $|x|\ge K+1$, we have
$l''(x)y\cdot y\ge \alpha \frac{1}{K+1}|y|^2$.
Define 
\[\hat H_K(t,x)=H_K(t,x)+\frac{K+1}{\alpha}(M_K+1)l(x)\] for $x\in\R^{2n}$.
Then $\hat H_K(t,x)$ satisfies {\bf (SH0)--(SH7)} as $H_K$, $\hat H_K=H$
for $|x|\le K$ and $\hat H''_K(t,x)$ is positive definite for $|x|\ge K+1$.
Integrating the inequality in {\bf (SH3)}, we get
\begin{equation}\label{e:H-super-quadratic}
    	\hat H_K(t,x) \ge a_3|x|^{\alpha}-a_4,
\end{equation}
for all $t\in\R$, $x\in\R^{2n}$, where $a_3>0$ and $a_4$
are both constants \textit{independent} of $K$.
Actually, both $a_3$ and $a_4$ depend only on  $\bar r$, $\alpha$ and the maximal value of $H(t,x)$ on $\{(t,x)\in I_T\times\R^{2n} ;|x|\le \bar r\}$.

We will solve
\begin{equation}\label{e:HS-T-K}
	\begin{cases}
		\dot x(t)=J\hat H_K'(t,x(t)), \\
		x(-t)=Nx(t),\\
		x(t+T)=x(t), 
	\end{cases}
\end{equation}
and find a solution 
$\bar x$ of \eqref{e:HS-T-K}  with an \textit{a priori} estimate similar to \eqref{e:a-priori-estimate-solution}. Denote by $K=\max\{C,\bar r,\bar R_1\}$,
where $C$ is a positive number depending only on $\bar r,\mu,T,\bar R_1, \bar \theta$ and the maximal value of $H(t,x)$ and $|H'(t,x)|$ on $\{(t,x)\in I_T\times \R^{2n};|x|\le \max\{\bar r,\bar R_1\}\}$.
Then $|\bar x(t)|\le K$ for $t\in \R$.
Thus  $\hat H_K(t,\bar x)=H(t,\bar x)$ and $\bar x$ solves the original \eqref{e:HS-jT}.

Since  $\hat H''_K(t,x)$ is positive definite for $|x|\ge K+1$, we can choose $\lambda>0$ so large that
\begin{enumerate}
	\item[{\rm (F0)}]  $\hat H_K(t,x)+\frac{1}{2}(\Lambda x,x)$ is strictly convex about $x\in\R^{2n}$;
\end{enumerate}
see Lemma \ref{l:positive-definite-lambda} for a proof.
Denote this strictly convex function by \[F(t,x)=\hat H_K(t,x)+\frac{1}{2}(\Lambda x,x).\]

Similar to \cite[Proposition II.4.5]{Ek90}, we have the least action principle, that is, the critical points of $\Phi$ are exactly the solutions of \eqref{e:HS-T-K}.
In fact, set
\[\mathcal{F}(x)=\int_{-\frac{T}{2}}^{\frac{T}{2}}\left[\hat H_K(t,x(t)) + G(x(t))\right]dt \quad \text{ for }\quad  x\in \hat L^{\alpha}(I_T;\R^{2n}),\]
where $G(x)=\frac{1}{2}(\Lambda x,x)$.
Let $\beta$ be the conjugate exponent of $\alpha$, i.e., $\frac{1}{\beta}+\frac{1}{\alpha}=1$.
Similar to \cite[Corollary II.3.3]{Ek90},
by (F0), we have
\[\partial \mathcal{F}(x)=\{x^*\in \hat L^{\beta}(I_T;\R^{2n}); x^*(t)\in \partial [\hat H_K(t,x(t))+G(x(t))] \text{ a.e. } t\in I_T\}, \]
where $\partial$ stands for subgradients of $\mathcal{F}$ at $x$ and $\hat H_K+G$ at $x(t)$.
Denote 
\[\Phi_K(x)=\int_{-\frac{T}{2}}^{\frac{T}{2}}\left[\frac{1}{2}(J\dot x-\Lambda x,x)+\hat H_K (t,x)+\frac{1}{2}(\Lambda x,x)\right]dt\]
on $\hat W_T^{1,\beta}(I_T;\R^{2n})$.
By regularity, a critical point of $\Phi_K$  is
a classical solution of $\eqref{e:HS-T-K}$.

The Fenchel conjugate (or Legendre transform) of $F(t,x)$ about $x\in\R^{2n}$  is defined by
\[F^*(t;y)=\sup_{x\in \R^{2n}}\{x\cdot y-F(t,x)\}.\]
Because of {\rm (F0)}, the function $F^*$ is well-defined and satisfies {\bf (SH0)}, i.e., $C^2$.  The Legendre reciprocity
formula (cf.\cite[Proposition II.1.15]{Ek90}) about the convex function $F(t,\cdot)$ is:
\begin{proposition}\label{p:legendre-reciprocity}
	The three properties are equivalent:
	\begin{enumerate}
		\item [\rm (i)] $y=F'(t,x)$,
		\item [\rm (ii)] $x=\nabla F^*(t;y)$, where $\nabla F^*$ is the partial gradient of $F^*(t;y)$ with respect to $y\in\R^{2n}$, 
		\item [\rm (iii)] $F(t,x)+F^*(t;y)=x\cdot y$.
	\end{enumerate}
\end{proposition}
According to \cite[Proposition II.2.10]{Ek90}, 
\[(F^*)^{''}(t;F'(t,x))F''(t,x)=I_{2n},\]
where $(F^*)^{''}(t;y)\in \R^{2n\times 2n}$ is the Hessian of $F^*(t;y)$ with respect to $y\in \R^{2n}$.

The dual functional $\Psi_K \colon \hat W_T^{1,\beta}(I_T;\R^{2n})\to \R$  is defined by
\[\Psi_K(x)=\int_{-\frac{T}{2}}^{\frac{T}{2}}\left[ \frac{1}{2}(J\dot{x}-\Lambda x,x)+F^*(t;-J\dot{x}+\Lambda x)\right]dt.\]
Since $F(t,x)$ is convex in $x\in\R^{2n}$, $\mathcal{F}$ is a convex lower semicontinuous functional
on  $\hat L^{\alpha}(I_T;\R^{2n})$ and $\Psi_K$ satisfies the formula (14) in
\cite[Theorem II.4.2]{Ek90}. In fact, the proof is almost the same as that of
\cite[Proposition II.4.6]{Ek90}, since we have
\[\mathcal{F}^*(x^*)=\int_{-\frac{T}{2}}^{\frac{T}{2}}F^*(t;x^*(t))dt \quad \text{ for any } x^*\in \hat L^{\beta}(I_T;\R^{2n})\]
and $\dom(\mathcal{F})=\hat L^{\alpha}(I_T;\R^{2n})$ by \cite[Theorem II.3.2]{Ek90}.
Then
we have the dual action principle  for our reversible
case:
\begin{proposition}\label{p:dual-action-principle}
	$x$ is a critical point of $\Psi_K$ if and only if there is a constant $\xi\in \{0\}\times \R^{n}$
	such that $x+\xi$ solves \eqref{e:HS-T-K}.
	We have
	\begin{equation}\label{e:func-dual-func}
		\Phi_K(x+\xi)+\Psi_K(x)=0.
		\end{equation}
\end{proposition}

We take advantage of the fact $\Psi_K$ is invariant by space translations:
\[\Psi_K(x+\xi)=\Psi_K(x) \qquad \forall \xi\in \{0\}\times \R^{n},\]
to perform the change of variables $-J\dot{x}+\Lambda x=u$.
For $r>1$, let \[\hat L_{\circ}^r(I_T;\R^{2n})=\{u\in \hat L^r(I_T;\R^{2n})|\int^{\frac{T}{2}}_{-\frac{T}{2}}u(t)dt=0\}.\]
We get a reduced functional 
\begin{equation}\label{e:reduced-functional}
\psi(u)=\int_{-\frac{T}{2}}^{\frac{T}{2}}\left[ \frac{1}{2}(-J\Pi u+J\Lambda J\Pi^2u,u)+F^*(t;u)\right]dt
\end{equation}
on 
$\hat L^{\beta}_{\circ}(I_T;\R^{2n})$,
where $\Pi u$ denotes the primitive of $u$ whose mean over $I_T$ is zero:
\[\frac{d}{dt}(\Pi u)=u \quad \text{and} \quad \int^{\frac{T}{2}}_{-\frac{T}{2}}(\Pi u)(t)dt=0.\]
Let $e_1,...,e_{2n}$ denote the usual bases in $\R^{2n}$, for $T=2\pi$, $j\in\Z\setminus \{0\}$ and $1\le k\le n$,
\begin{equation}\label{e:Pi-example}
	\begin{split}
	\Pi [(\sin jt)e_k+(\cos jt)e_{k+n}]&= \frac{1}{j}[-(\cos jt)e_k+(\sin jt)e_{k+n}],\\
	J\Pi [(\sin jt)e_k+(\cos jt)e_{k+n}]&= -\frac{1}{j}[(\sin jt)e_k+(\cos jt)e_{k+n}].
	\end{split}
\end{equation}
Then we have
$\psi(-J\dot x+\Lambda x)=\Psi_K(x)$ for $x\in \hat W^{1,\beta}_T(I_T;\R^{2n})$.
According to Proposition \ref{p:dual-action-principle}, we obtain 
\begin{proposition}\label{p:dual-critical-solution}
If we find a critical point $\bar u\in \hat L_{\circ}^{\beta}(I_T;\R^{2n})$ of $\psi$, then there must be
some $\bar x\in \hat W_T^{1,\beta}(I_T;\R^{2n})$ such that $-J\frac{d\bar x}{dt}+\Lambda \bar x=\bar u$ and $\bar x$ is the critical point of $\Psi_K$, thus a $T$-periodic brake solution of \eqref{e:HS-T-K}.
\end{proposition}
Since $F(t,x)$ is superquadratic in $q\in\R^n$ for $x=\begin{pmatrix}
	p\\q
\end{pmatrix}$ and contains the quadratic term $\frac{1}{2}\lambda |p|^2$,  by duality we have
\begin{lemma}\label{l:H*-big-power2}
	For any $\varepsilon>0$,
	there exits an $\eta>0$ such that
	\[ F^*(t;y)\ge \frac{1}{2\varepsilon}|q|^2+\frac{1}{2(\lambda +\varepsilon)}|p|^2\quad \text{ for }\quad  y=\begin{pmatrix}
		p\\q\end{pmatrix},
	|y|\le  \eta.\]
\end{lemma}
\begin{proof}
By {\bf (SH2)}, for any $\varepsilon>0$, there exists a $\delta>0$ such that
$\hat H_K(t,x)\le \frac{\varepsilon}{2}|x|^2$ if $|x|<\delta$ and $t\in I_T$.
Set $\eta_1=\min \{|F'(t,x)|;|x|\ge \delta, t\in I_T\}$.
Then $\eta >0$, since $F'(t,0)=0$ and the Hessian of $F$ about $x\neq0$ is positive definite.
Recall $F(t,x)=\hat H_K(t,x)+\frac{1}{2}(\Lambda x,x)$. By definition,
	\[F^*(t;y)=\sup_{x\in\R^{2n}}\left[(x,y)-\hat H_K(t,x)-\frac{1}{2}(\Lambda x,x)\right].\]
Notice the condition {\rm (F0)}, which states $F(t,x)$ is  convex about $x\in\R^{2n}$.
By the definition of $\eta_1$, if $|y|< \eta$, then  $y=F'(t,x)$ for some $x$ with $|x|\le \delta$. 
Set $\eta=\min \{\eta_1,\varepsilon\delta\}$.
Thus for $y<\eta $ we have
	\begin{equation*}
		\begin{split}
			F^*(t;y)&\ge \sup _{|x|\le \delta}\left[(x,y)-\frac{\varepsilon}{2}|x|^2-\frac{1}{2}(\Lambda x,x)\right]\\
			&= \frac{1}{2\varepsilon}|q|^2+\frac{1}{2(\lambda +\varepsilon)}|p|^2.
		\end{split}
	\end{equation*}
\end{proof}

We are now in a position to apply our main critical point theorem to the functional $\psi$.
We first have to check the assumptions of Theorem \ref{t:critical-point-mp}, where we take $u_0=0$ and $u_1$ will be defined presently. 
Note that, because of conditions {\bf (SH1)} and {\bf (SH2)}, $\psi(0)=0$.
\begin{theorem}\label{t:mountain-condition}
	\begin{enumerate}
		\item [\rm (i)]
 $\psi$ has a local minimum at the origin. In fact,
		there exist two positive numbers $\rho,a$ such that
		$\psi(u)\ge a$ for any $u\in \hat L_{\circ}^{\beta}(I_T;\R^{2n})$ with $\|u\|_{\beta}=\rho$.
		\item [\rm (ii)] There is an $e\in\hat L_{\circ}^{\beta}(I_T;\R^{2n})$ with $\|e\|_{\beta}> \rho$ such that $\psi(e)<0$.
		\item [\rm (iii)] The functional $\psi$ satisfies  condition {\rm (PS)}.
	\end{enumerate}
\end{theorem}
\begin{proof}
	First we give some estimates of $F$, $F^*$ and their gradients, which are inspired by \cite[Lemmas IV.2.3-2.6]{Ek90}.
	Note that the modified $\hat H_K$ satisfies 
	\begin{enumerate}
	\item [{\bf (SH8)}] $\displaystyle \limsup_{|x|\to \infty}\max_{t\in I_T}\frac{\hat H_K(t,x)}{|x|^{\alpha}}<\infty$.
	\end{enumerate}
	Recall  $F(t,x)=\hat H_K(t,x)+\frac{1}{2}(\Lambda x,x)$ and $\alpha>2$. Thus $F$ still satisfies {\bf (SH8)}.
	
	Since $\hat H_K$ satisfies {\bf (SH3)} with $\mu$ replaced by $\alpha$, integrating this inequality then yields 
	\begin{equation}\label{e:super-quadratic1}
		\hat H_K(t,x)\ge \frac{m^{\alpha}}{\alpha}\frac{|x|^{\alpha}}{\bar r^{\alpha}} \qquad \text{for} \quad |x|\ge \bar r, t\in I_T,
	\end{equation}
	where $\frac{m^{\alpha}}{\alpha}:=\min\{H(t,x);|x|=\bar r,t\in I_T\}$ is \textit{independent} of $K$.
	Since $F\ge \hat H_K$, $F$ still satisfies \eqref{e:super-quadratic1}.
	Together with {\bf(SH8)}, there is a constant $c>0$ such that
	\begin{equation}\label{e:H-big-small-alpha}
		\frac{m^{\alpha}}{\alpha \bar r^{\alpha}}|x|^{\alpha}\le F(t,x)\le \frac{c^{\alpha}}{\alpha}|x|^{\alpha} \qquad\text{for}\quad  |x|\ge \bar r, t\in I_T.
	\end{equation}
	Thus using the convexity of $F(t,\cdot)$ (cf. \cite[Lemma IV.2.6]{Ek90}), there is  a constant $c_1$ such that 	\begin{equation}\label{e:bound-H-gradient}
		|F'(t,x)|\le c_1(1+|x|^{\alpha-1})\qquad \text{for}\quad (t,x)\in I_T\times \R^{2n}.
	\end{equation}

Recall $F^*(t;y)=\sup\limits_{x\in \R^{2n}}\{x\cdot y-F(t,x)\}$ and $\frac{1}{\beta}+\frac{1}{\alpha}=1$.	
Dual to \eqref{e:H-big-small-alpha}, $F^*(t;\cdot)$ satisfies 
\begin{equation}\label{e:low-up-bound-H*}
		\frac{1}{\beta k^{\beta}}|y|^{\beta}\le F^*(t;y) \le \frac{\bar r^{\beta}}{\beta m^{\beta}}|y|^{\beta} \qquad\text{for}\quad |y|\ge R,t\in I_T, 
\end{equation}
where $R:=\max\{|F'(t,x)|;|x|\le \bar r,t\in I_T\}+1$.
The existence of $k>0$ (which can be taken arbitrarily large) follows from that $F$ satisfies {\bf (SH8)}
and $F^*(t;y)>0$ for $y\neq 0$.
Similarly, there exists a constant $c_2$ such that
\begin{equation}\label{e:bound-H*-gradient}
	|\nabla F^*(t;y)| \le (1+c_2)|y|^{\beta-1}\qquad
\text{for}\quad (t,y)\in I_T\times \R^{2n}. 
\end{equation}

By  Lemma \ref{l:H*-big-power2} and \eqref{e:low-up-bound-H*},   we can choose $k$ so large that for $y=\begin{pmatrix}
		p\\q
	\end{pmatrix}\in \R^{2n}$,
	\begin{equation}\label{e:F*-ge-eta}
		\begin{cases}
			F^*(t;y)\ge \frac{|q|^2}{2\varepsilon}+\frac{|p|^2}{2(\varepsilon+\lambda)} & |y|\le \eta, \\
			F^*(t;y)\ge \frac{1}{\beta k^{\beta}}|y|^{\beta} & |y|\ge \eta,
		\end{cases}
	\end{equation}
	where $\eta$ is defined in Lemma \ref{l:H*-big-power2}.
	
With the $R$ defined in \eqref{e:low-up-bound-H*}, we have\begin{equation}\label{e:H*-beta}
		(y,\nabla F^*(t;y))\le  \beta F^*(t;y)-(\frac{\beta}{2}-1)(\Lambda \nabla F^*(t;y),\nabla F^*(t;y)) \text{ for } |y|\ge R, t\in I_T.
	\end{equation}
	Now we prove \eqref{e:H*-beta}.
	For $|x|\ge \bar r$, we have
	\begin{equation}\label{e:super-quadratic}
		(x,F'(t,x))=(x,\hat H_K'(t,x)+\Lambda x)\ge \alpha \hat H_K(t,x)+(\Lambda x,x).
	\end{equation}
	For $|y|\ge R$, let $x=\nabla F^*(t;y)$.
	By the Legendre reciprocity formula, we get $y=F'(t,x)$.
	By the definition of $R$ in \eqref{e:low-up-bound-H*}, we get $|x|\ge \bar r$.
 \eqref{e:super-quadratic} becomes
	\begin{equation*}
		\begin{split}
			(\nabla F^*(t;y),y)&\ge \alpha (\hat H_K(t,x)+\frac{1}{2}(\Lambda x,x))+(1-\frac{\alpha}{2})(\Lambda x,x)\\
			&= \alpha F(t,x)+ (1-\frac{\alpha}{2})(\Lambda x,x)\\
			&= \alpha (-F^*(t;y)+(\nabla F^*(t;y),y))+
			(1-\frac{\alpha}{2})(\Lambda \nabla F^*(t;y)\nabla F^*(t;y)).	\end{split}
	\end{equation*}
	Since $\frac{1}{\alpha}+\frac{1}{\beta}=1$,
	we readily get \eqref{e:H*-beta}.
	

	(1) 	We recall for $u\in \hat L_{\circ}^{\beta}(I_T;\R^{2n})$,
	\[\psi(u)=\int_{-\frac{T}{2}}^{\frac{T}{2}}\left[ \frac{1}{2}(-J\Pi u+J\Lambda J\Pi^2u,u)+F^*(t;u)\right]dt.\]
	According to the estimate of \eqref{e:bound-H*-gradient},
	$\psi$ is $C^1$ on $\hat L_{\circ}^{\beta}(I_T;\R^{2n})$; see
	\cite[Corollary II3.5]{Ek90}.
	By the definition of $\Pi$ and Ascoli's theorem,
	$\Pi\colon \hat L_{\circ}^{\beta}(I_T;\R^{2n}) \to \hat L_{\circ}^{\alpha}(I_T;\R^{2n})$
	is a compact operator.
	
	Now we prove that $\psi$ satisfies condition (PS).
	Consider a (PS) sequence $u_n\in \hat L_{\circ}^{\beta}(I_T;\R^{2n})$, i.e.,
	$\psi(u_n)$ is bounded and $\psi'(u_n)\to 0 $ in $(\hat L_{\circ}^{\beta}(I_T;\R^{2n}))^*$.
	We need to show that it has a convergent subsequence in $\hat L_{\circ}^{\beta}(I_T;\R^{2n})$.
	
	Since $(\hat L_{\circ}^{\beta}(I_T;\R^{2n}))^*\hookrightarrow  \hat L^{\alpha}(I_T;\R^{2n})$,
	\[\psi'(u_n)=-J\Pi u_n+ J\Lambda J\Pi^2u_n+ \nabla F^*(t;u_n)+\xi_n,\]
	where $\xi_n\in \{0\}\times \R^{n}$ is a constant vector.
	Denote by
	\begin{equation}\label{e:explicit-gradient}
		\epsilon_n=-J\Pi u_n+ J\Lambda J\Pi^2u_n+ \nabla F^*(t;u_n)+\xi_n.
	\end{equation}
	Then by assumption, $\epsilon_n\to 0$ in $\hat L^{\alpha}(I_T;\R^{2n})$.
	
	We will prove $u_n\in \hat L_{\circ}^{\beta}(I_T;\R^{2n})$ is bounded.
	Assuming that, then $u_n$ has a weakly convergent subsequence in $\hat L_{\circ}^{\beta}(I_T;\R^{2n})$.
	Applying  \eqref{e:bound-H*-gradient} and the compactness of  $\Pi$ to \eqref{e:explicit-gradient}, $\xi_n$ is a bounded sequence of $\R^{2n}$ and has a subsequence which converges to
	some $\xi\in\{0\}\times \R^n$.
	In the sense of subsequences, which are still denoted by $u_n,\xi_n$ and $\epsilon_n$,
	we have $u_n$ is weakly convergent to some $u\in \hat L_{\circ}^{\beta}(I_T;\R^{2n})$  and
	\[J\Pi u_n- J\Lambda J\Pi^2u_n-\xi_n+\epsilon_n \text{ converges to } J\Pi u- J\Lambda J\Pi^2u-\xi \in \hat L_{\circ}^{\beta}(I_T;\R^{2n}).\]
	
	We rewrite \eqref{e:explicit-gradient} as
	$\nabla F^*(t;u_n)=J\Pi u_n- J\Lambda J\Pi^2u_n-\xi_n+\epsilon_n$.
	Then by the Legendre reciprocity formula (cf. Proposition \ref{p:legendre-reciprocity}),
	we have
	\[u_n=F'(t,J\Pi u_n- J\Lambda J\Pi^2u_n-\xi_n+\epsilon_n).\]
	Since $F'\colon \hat L_{\circ}^{\alpha}(I_T;\R^{2n}) \to \hat  L^{\beta}(I_T;\R^{2n})$ is continuous,
	\[\lim_{n\to \infty}u_n=F'(t,J\Pi u- J\Lambda J\Pi^2u-\xi) \text{ in } \hat L^{\beta}(I_T;\R^{2n}).\]
	We have already known that $u_n$ is weakly convergent to $u$ in $\hat L_{\circ}^{\beta}(I_T;\R^{2n})$.
	Finally, we have
	$u=F'(t,J\Pi u- J\Lambda J\Pi^2u-\xi)$;
	then $J\Pi u- J\Lambda J\Pi^2u-\xi=\nabla F^*(t;u)$, that is
	$\psi'(u)=0$.
	
	We write $u_n=v_n+w_n$ to prove the boundedness of $u_n$,
	where \begin{align*}
		v_n(t)=
		\left\{
		\begin{array}{lr}
			0   &\text{ if } |u_n(t)|>\eta,\\
			u_n(t) &\text{ if } |u_n(t)|\le \eta; \\
		\end{array}
		\right.
		\qquad&
		w_n(t)=
		\left\{
		\begin{array}{lr}
			u_n(t) &\text{ if } |u_n(t)|>\eta,\\
			0 &\text{ if } |u_n(t)|\le \eta, \\
		\end{array}
		\right.
	\end{align*}
	and $\eta$ is the same one in \eqref{e:F*-ge-eta}.
	
	Then
	by the definition of $v_n$,  we have $\|v_n\|_{\infty}<\eta$. Now we only need to prove $\|w_n\|_{\beta}$ is bounded.
	Writing \eqref{e:explicit-gradient} into $\psi(u_n)$, we have
	\[\psi(u_n)=\int_{-\frac{T}{2}}^{\frac{T}{2}}\left[ \frac{1}{2}(-\nabla F^*(t;u_n)+\epsilon_n, u_n)+F^*(t;u_n)\right]dt.\]
	
	Split the above integral into two parts by the following partition
	\[I_T=\{t\in I_T;|u_n(t)|\le \eta\}\cup \{t\in I_T;|u_n(t)|> \eta\}.\]
	Together with the assumption that $\psi(u_n)$ is uniformly bounded,
	we get
	\begin{equation*}
		\begin{split}
			\int_{\{t;|u_n(t)|>\eta\}}&\left[ \frac{1}{2}(-\nabla F^*(t;u_n)+\epsilon_n, u_n)+F^*(t;u_n)\right]dt \\
			&=\int_{-\frac{T}{2}}^{\frac{T}{2}}\left[ \frac{1}{2}(-\nabla F^*(t;w_n)+\epsilon_n, w_n)+F^*(t;w_n)\right]dt
		\end{split}
	\end{equation*}
	is uniformly bounded.
	
	Note first that condition \eqref{e:H*-beta} implies that there exists a constant $c_3\ge 0$ such that
	\[(y,\nabla F^*(t;y))\le  \beta F^*(t;y)-(\frac{\beta}{2}-1)(\Lambda \nabla F^*(t;y),\nabla F^*(t;y))+c_3 \text{ for any } y\in\R^{2n}, t\in I_T.\]
Together with \eqref{e:bound-H*-gradient} and \eqref{e:F*-ge-eta},
we get
	\begin{equation*}
		\begin{split}
			\int_{-\frac{T}{2}}^{\frac{T}{2}}&\left[ \frac{1}{2}(-\nabla F^*(t;w_n)+\epsilon_n, w_n)+F^*(t;w_n)\right]dt \\
			&\ge (1-\frac{\beta}{2})\int_{-\frac{T}{2}} ^{\frac{T}{2}}F^*(t;w_n)dt
			+\frac{1}{2}(\frac{\beta}{2}-1)\lambda(1+c_2)\|w_n\|^{2\beta-2}_{2\beta-2}
			-Tc_3-\frac{1}{2}\|\epsilon_n\|_{\alpha}\|w_n\|_{\beta}\\
			& \ge (1-\frac{\beta}{2})\frac{1}{\beta k^{\beta}}\|w_n\|^{\beta}_{\beta}
			+\frac{\beta-2}{4}\lambda(1+c_2)\|w_n\|^{2\beta-2}_{2\beta-2}-Tc_3-\frac{1}{2}\|\epsilon_n\|_{\alpha}\|w_n\|_{\beta}.
		\end{split}
	\end{equation*}
	
	Since $1<\beta<2$,
	$0<2\beta-2<\beta$, we obtain
	\[\|w_n\|^{2\beta-2}_{2\beta-2}\le T^{\frac{2}{\beta}-1}\|w_n\|^{2\beta-2}_{\beta}.\]
	Since $1<\beta<2$ and $\|\epsilon_n\|_{\alpha}\to 0$ as $n\to \infty$,
	$\|w_n\|_{\beta}$ is uniformly bounded.
	Then we get the sequence $u_n=v_n+w_n$ is bounded in $\hat L_{\circ}^{\beta}(I_T;\R^{2n})$. Thus we get {\rm (iii)}: $\psi$ satisfies condition (PS).
	
	(2) 
	 Now we will prove {\rm (i)} and {\rm (ii)}. Since $H$ satisfies
	{\bf (SH1)} and {\bf (SH2)}, we have $H(t,0)=0$
	and  $F^*(t;0)=0$. Then we get $\psi(0)=0$.
	We claim that there exists a positive number $\rho$ such that
	\[\inf \{\psi(u);\|u\|_{\beta}=\rho\}>0.\]
	Define $L^{\beta}_{\circ}(I_T;\R^n)=\{z\in L^{\beta}(I_T;\R^n);\int_{-T/2}^{T/2}z(t)dt =0\}$, and 
	\begin{equation*}
		\begin{split}
			\Pi \colon L^{\beta}_{\circ}(I_T;\R^n) &\rightarrow  L^{\beta}_{\circ}(I_T;\R^n)\\
			z&\mapsto \int^t_{-T/2}z(s)ds-\frac{1}{T}\int^{T/2}_{-T/2}\int^t_{-T/2}z(s)dsdt.
		\end{split}
	\end{equation*}
	We solve $-J\dot x+\Lambda x=u$:
	for $u=\begin{pmatrix}p\\q
	\end{pmatrix}\in \hat L^{\beta}_{\circ}(I_T;\R^{2n})$, 
	\begin{equation}\label{e:u-inverse-x}
	x=J\Pi u-J\Lambda J\Pi^2u=\begin{pmatrix}
		-\Pi q\\
		\Pi p +\lambda \Pi^2q
	\end{pmatrix}\in \hat W_T^{1,\beta}(I_T;\R^{2n}) \mod {0}\times \R^n.\end{equation}
	Note that since $q\in L^{\beta}_{\circ}(I_T;\R^n)$ and $q(-t)=q(t)$, $\int_{-T/2}^t q(s) ds$ is an odd function and thus \[(\Pi q)(t)=\int_{-T/2}^t q(s) ds.\]
	\begin{equation}\label{e:psi-J-big}
		\begin{split}
			\frac{1}{2}\int_{-\frac{T}{2}}^{\frac{T}{2}}&(J\dot x-\Lambda x,x) dt=-\frac{1}{2}\int_{-\frac{T}{2}}^{\frac{T}{2}}(u,x) dt
			=\frac{1}{2}\int^{\frac{T}{2}}_{-\frac{T}{2}}[(\Pi q,p)-(\Pi p,q)-\lambda(\Pi^2q,q)]dt\\
			&=\int^{\frac{T}{2}}_{-\frac{T}{2}} \left[(\Pi q,p)+\frac{\lambda}{2}(\Pi q,\Pi q)\right] dt
			\ge \int^{\frac{T}{2}}_{-\frac{T}{2}} (\Pi q,p) dt.
		\end{split}
	\end{equation}
For any $z\in L^{\beta}_{\circ}(I_T;\R^n)$, 
	set 
	\begin{align*}
		z_1(t)=
		\left\{
		\begin{array}{lr}
			0   &\text{ if } |z(t)|>\eta,\\
			z(t) &\text{ if } |z(t)|\le \eta; \\
		\end{array}
		\right.
		\qquad&
		z_2(t)=
		\left\{
		\begin{array}{lr}
			z(t) &\text{ if } |z(t)|>\eta,\\
			0 &\text{ if } |z(t)|\le \eta, \\
		\end{array}
		\right.
	\end{align*}
	where $\eta$ is from Lemma \ref{l:H*-big-power2}. Thus $z=z_1+z_2$.
Here we use notation
\[p(t)=p_1(t)+p_2(t) \quad q(t)=q_1(t)+q_2(t) \quad u(t)=u_1(t)+u_2(t), \]
where $u_i=\begin{pmatrix}p_i\\q_i\end{pmatrix}, i=1,2$.
The mean value of $q_i,i=1,2$ may not be zero.
But we still have 
\[(\Pi q)(t)=\int_{-T/2}^t[q_1(s)+q_2(s)]ds = \int_{-T/2}^tq_1(s) ds+ \int_{-T/2}^tq_2(s)ds.\]
We still use the notation $\Pi$ to represent the primitive, that is, $(\Pi q_i)(t)= \int_{-T/2}^tq_i(s)ds$, $i=1,2$.
Note that $q_1\in L^{\infty}(I_T;\R^n)$ and $q_2\in L^{\beta}(I_T;\R^n)$. By H\"{o}lder's inequality, we have
\begin{equation}\label{e:pi-2-less}
\|\Pi q_1\|_2\le T\|q_1\|_2,\quad \|\Pi q_2\|_2\le T^{\frac{\alpha+2}{2\alpha}}\|q_2\|_{\beta},
\quad \|\Pi q_1\|_{\alpha}\le T^{\frac{\alpha+2}{2\alpha}}\|q_1\|_2,\quad
\|\Pi q_2\|_{\alpha}\le T^{\frac{2}{\alpha}}\|q_2\|_{\beta}.
\end{equation}
By virtue of  the H\"{o}lder inequality, the triangle inequality  and \eqref{e:pi-2-less}, we obtain
\begin{equation*}
	\begin{split}
		-\int^{\frac{T}{2}}_{-\frac{T}{2}} (\Pi q,p) dt
		&=	\int^{\frac{T}{2}}_{-\frac{T}{2}} [-(\Pi q,p_1)-(\Pi q,p_2) ]dt\\
		&\le \|\Pi q\|_2\|p_1\|_2+\|\Pi q\|_{\alpha}\|p_2\|_{\beta}	\\
		&	\le  \|\Pi q_1\|_2\|p_1\|_2+\|\Pi q_2\|_2\|p_1\|_2+
		\|\Pi q_1\|_{\alpha}\|p_2\|_{\beta}+\|\Pi q_2\|_{\alpha}\|p_2\|_{\beta}\\
		&\le (T\|q_1\|_2+T^{\frac{\alpha+2}{2\alpha}}\|q_2\|_{\beta})\|p_1\|_2+(T^{\frac{\alpha+2}{2\alpha}}\|q_1\|_2+T^{\frac{2}{\alpha}}\|q_2\|_{\beta})\|p_2\|_{\beta}.
	\end{split}
\end{equation*}
We transform rectangular terms to the squares by the inequality $ab\le \frac{1}{2}(\frac{a^2}{\varepsilon^2_1}+\varepsilon^2_1b^2)$, valid for any $\varepsilon_1$, and go on.
We get
\begin{equation}\label{e:psi-H-big}
	\begin{split}
		-\int^{\frac{T}{2}}_{-\frac{T}{2}} (\Pi q,p) dt
		&\le
		\frac{T}{2}(\varepsilon^2_1\|p_1\|_2^2+\frac{\|q_1\|_2^2}{\varepsilon^2_1})+\frac{T^{\frac{\alpha+2}{2\alpha}}}{2}(\frac{\|q_2\|^2_{\beta}}{\varepsilon^2_1}+\varepsilon^2_1\|p_1\|_2^2)\\
		&\quad\ +\frac{T^{\frac{\alpha+2}{2\alpha}}}{2}\|q_1\|^2+\frac{T^{\frac{2}{\alpha}}}{2}\|q_2\|^2_{\beta}+\frac{1}{2}(T^{\frac{\alpha+2}{2\alpha}}+T^{\frac{2}{\alpha}})\|p_2\|^2_{\beta}\\
		&= (\frac{T}{2}+\frac{T^{\frac{\alpha+2}{2\alpha}}}{2})\varepsilon^2_1\|p_1\|_2^2+ (\frac{T}{2\varepsilon^2_1}  +\frac{T^{\frac{\alpha+2}{2\alpha}}}{2}) \|q_1\|_2^2\\
		&\quad \ +
		\frac{1}{2}(T^{\frac{\alpha+2}{2\alpha}}+T^{\frac{2}{\alpha}})\|p_2\|^2_{\beta}+(\frac{T^{\frac{\alpha+2}{\alpha}}}{2\varepsilon^2_1}+\frac{T^{\frac{2}{\alpha}}}{2})\|q_2\|^2_{\beta}.
	\end{split}
\end{equation}	

Then by the definition of $\psi$ (\eqref{e:reduced-functional}),  \eqref{e:u-inverse-x},
\eqref{e:psi-J-big}, \eqref{e:psi-H-big} and \eqref{e:F*-ge-eta}, we have  
\begin{equation}\label{e:psi-big>0}
	\begin{split}
		\psi(u)&\ge
		-(\frac{T}{2}+\frac{T^{\frac{\alpha+2}{2\alpha}}}{2})\varepsilon^2_1\|p_1\|_2^2- (\frac{T}{2\varepsilon^2_1}  +\frac{T^{\frac{\alpha+2}{2\alpha}}}{2}) \|q_1\|_2^2 +\frac{\|q_1\|_2^2}{2\varepsilon}+\frac{\|p_1\|_2^2}{2(\varepsilon+\lambda)}\\
		&\quad\  -\frac{1}{2}(T^{\frac{\alpha+2}{2\alpha}}+T^{\frac{2}{\alpha}})\|p_2\|^2_{\beta}-(\frac{T^{\frac{\alpha+2}{\alpha}}}{2\varepsilon^2_1}+\frac{T^{\frac{2}{\alpha}}}{2})\|q_2\|^2_{\beta} +\frac{1}{\beta k^{\beta}} \|u_2\|_{\beta}^{\beta}.
	\end{split}
\end{equation}
For any $T>0$, choose $\varepsilon^2_1=2T\varepsilon$, where $\varepsilon>0$. To prove (i),
we want both the 
coefficients of terms $\|p_1\|_2^2$ and $\|q_1\|^2_2$ to be positive, that is,
\begin{gather}
	\frac{T}{2} +\frac{T^{\frac{\alpha+2}{2\alpha}}}{2}< \frac{1}{4(\varepsilon+\lambda)T\varepsilon}, \label{e:estimate_big0_1}\\
	\frac{1}{4\varepsilon}+\frac{T^{\frac{\alpha+2}{2\alpha}}}{2}<\frac{1}{2\varepsilon}.
	\label{e:estimate_big0_2}
\end{gather}
	For any $T>0$, we can choose $\varepsilon$ so small that both \eqref{e:estimate_big0_1} and \eqref{e:estimate_big0_2} hold.
	Then we get corresponding $\delta$ and $\eta$ in Lemma \ref{l:H*-big-power2}.
	The second line of \eqref{e:psi-big>0} may be negative.
	Since $\beta\in (1,2)$ and $\|u_2\|_{\beta}=\|p_2\|_{\beta}+\|q_2\|_{\beta}$, its behavior near $0$ is determined by the term $\|u_2\|^{\beta}_{\beta}$.
	If $\|u_2\|_{\beta}>0$ small enough, we get
	$\psi (u)>0$.
	Obviously, we have
	 $\|u\|^{\beta}_{\beta}=\|u_1\|^{\beta}_{\beta}+\|u_2\|^{\beta}_{\beta}$.
	  It follows that, if $\|u\|_{\beta}=\rho>0$ is small enough, $\psi (u)>0$ and (i) holds.

	{\rm (ii)}
	Setting $x(t)=\begin{pmatrix}p(t)\\q(t)
	\end{pmatrix}=\begin{pmatrix}
		-\sin \frac{2\pi t}{T} y_0\\
		\cos \frac{2\pi t}{T} y_0
	\end{pmatrix}$, $t\in I_T$, where $y_0\in \R^n$,
	we get $x\in \hat W^{1,\beta}_T(I_T;\R^{2n})$,
	\begin{equation}\label{e:psi-small-0}
		u=J\dot x-\Lambda x=\frac{2\pi}{T}\begin{pmatrix}
		 \sin \frac{2\pi t}{T} y_0\\
		 -\cos \frac{2\pi t}{T} y_0
		\end{pmatrix}+\lambda \begin{pmatrix}
		\sin \frac{2\pi t}{T} y_0\\
		0
		\end{pmatrix},
\end{equation}
and $u\in \hat L_{\circ}^{\beta}(I_T;\R^{2n})$.
Then we obtain
\[\int^{\frac{T}{2}}_{-\frac{T}{2}} (J\dot x-\Lambda x,x) dt=(-2\pi-\frac{\lambda T}{2})|y_0|^2.\]
Since $F(t,x)\ge a_3|x|^{\alpha}-a_4+\frac{1}{2}(\Lambda x,x)$ (cf. \eqref{e:H-super-quadratic}), writing the two convex continuous functions $F_1(x)=a_3|x|^{\alpha}$ and $F_2(x)=\frac{1}{2}(\Lambda x,x)$,
we have \[F(t,x)\ge F_1(x)+F_2(x)-a_4,\]
for all $t\in\R$ and $x\in\R^{2n}$.
Using the concept of \textit{inf-convolute} of $F_1^*$ and $F^*_2$, 
we have
\begin{equation}\label{e:inf-convolute}
		\begin{split}
		F^*(t;-J\dot x+\Lambda x)&\le (F_1+F_2)^*(t;-J\dot x+\Lambda x)+a_4\\
		&\le F_1^*(t;-J\dot x)+F_2^*(\Lambda x)+ a_4\\
		&\le c_3|\dot x|^{\beta} +\frac{\lambda}{2}|p|^2+a_4,	\end{split}
	\end{equation}where  
	$c_3$ and $a_4$ are constants independent of $\lambda$ and $K$.
	In fact, the second inequality in \eqref{e:inf-convolute}  follows immediately from formulae (3) and (5) in \cite[Theorem II.2.3]{Ek90}.
	Thus we have
	\begin{equation*}
		\begin{split}
			\psi(u)=\Psi_K(x)&\le (-\pi -\frac{\lambda T}{4})|y_0|^2
		+\sqrt{2}c_3\frac{(2\pi)^{^{\beta}}}{T^{\beta-1}}|y_0|^{\beta}
		+\frac{\lambda T}{4}|y_0|^2+a_4T\\
		&=-\pi |y_0|^2
		+\sqrt{2}c_3\frac{(2\pi)^{^{\beta}}}{T^{\beta-1}}|y_0|^{\beta}+
		a_4T.
		\end{split}
	\end{equation*}
	Let $h>0$. Then 
	\begin{equation}\label{e:psi-upper-bound}
		\psi(hu)\le -\pi h^2|y_0|^2
		+\sqrt{2}c_3\frac{(2\pi)^{^{\beta}}}{T^{\beta-1}}h^{\beta}|y_0|^{\beta}+
		a_4T,
	\end{equation}
	and the right-hand side goes to $-\infty$ when 
	$h\to \infty$, since $1<\beta<2$.
	Note that $u\neq 0$, and we can choose $h$ sufficient large such that $\psi(hu)<0$ and $\|hu\|_{\beta}>\rho$ in {\rm (i)}. Then we choose $e=hu$.

\end{proof}

\subsection{Morse index and the geometrical structure near the critical point}\label{ss:morse-index-topology}
In this subsection,
we define the Morse index of the brake solution of \eqref{e:HS-T-K} and prove the Morse index
of the critical point we found in Subsection \ref{ss:critical-point} is less than one.
At the same time, we describe the geometrical structure near the critical point of mountain-pass type.

For $1\le p\le \infty$, recall  $I_T=[-T/2,T/2]$ and 
\[\hat L_{\circ}^p(I_T;\R^{2n})=\{u\in L^p(I_T;\R^{2n})| u(-t)=Nu(t), t\in I_T  \text{ and }\int_{-\frac{T}{2}}^{\frac{T}{2}}udt=0\}.\]
Denote a quadratic form on $\hat L_{\circ}^2(I_T;\R^{2n})$
by
\begin{equation}\label{e:quadratic-form-T}
	\begin{split}
		q_{T/2}(u,u)&=\int_{-\frac{T}{2}}^{\frac{T}{2}}\frac{1}{2} \left[(-J\Pi u+J\Lambda J\Pi^2u, u)+((F^*)^{''}(t;
		\bar u)u,u)\right]dt\\
		&=\int_{-\frac{T}{2}}^{\frac{T}{2}}\frac{1}{2} \left[(-J\Pi u+J\Lambda J\Pi^2u, u)+((F^{''}(t,\bar x))^{-1}u,u)\right]
		dt,
	\end{split}
\end{equation}
where $\bar x$ is a non-constant solution of \eqref{e:HS-T-K},
$\bar u=-J\frac{d}{dt}\bar x+\Lambda \bar x$ and $\Pi\colon \hat L_{\circ}^2(I_T;\R^{2n}) \to \hat L_{\circ}^2(I_T;\R^{2n})$ denotes the primitive.
By \cite[Proposition II.2.10]{Ek90}, remembering that 
\[-J\frac{d \bar x}{dt}+\Lambda \bar x=\hat H_K'(t,\bar x)+\Lambda \bar x=F'(t,\bar x),\]
we have
\[(F^*)''(t;-J\frac{d \bar x}{dt}(t)+\Lambda x(t))=F''(t,\bar x(t))^{-1}.\]

Note that $F''(t,\bar x)=\hat H_K''(t,\bar x)+\Lambda$.
The index $m^-(\bar x)=m^-(\bar u)$ is defined to be the maximal dimension of the  negative definite subspace associated with the quadratic form $q_{T/2}$ on $\hat L_{\circ}^2(I_T;\R^{2n})$. Let $m^0(\bar x)$ denote the nullity of the quadratic form $q_{T/2}$.
Define the restricted function $\psi_{\infty}\colon \hat L_{\circ}^{\infty}(I_T;\R^{2n})\to \R$
by $\psi_{\infty}(u)=\psi(u)$ for any $u\in \hat L_{\circ}^{\infty}(I_T;\R^{2n})$.
Then $\psi_{\infty}\in C^2(\hat L_{\circ}^{\infty};\R)$.

If $\bar u$ is a critical point of $\psi$ on $\hat L_{\circ}^{\beta}(I_T;\R^{2n})$, then
$\bar u=-J\frac{d \bar x}{d t}+\Lambda \bar x$ and $\bar x$ solves \eqref{e:HS-T-K}. We get
$\bar u$ is continuous, thus $\bar u\in \hat L_{\circ}^{\infty}(I_T;\R^{2n})$.
The restriction of $q_{T/2}$ to $\hat L_{\circ}^{\infty}(I_T;\R^{2n})$ is just the Hessian of $\psi_{\infty}$ at $\bar u$:
\[(\psi''_{\infty}(\bar u)v,v)=q_{T/2}(v,v)\ \ \ \forall v\in \hat L_{\circ}^{\infty}(I_T;\R^{2n}).\]

Similar to \cite[Lemma IV.3.1]{Ek90}, to make use of Morse theory, we 
want to do our analysis down from $\hat L_{\circ}^{\beta}(I_T;\R^{2n})$ to $\hat L_{\circ}^{\infty}(I_T;\R^{2n})$.
\begin{lemma}\label{l:L-beta-infinity}
	Take $\bar u\in \Cr(\psi,d)$  and
	set \begin{gather*}
		\dot \psi^d:=\{u\in \hat L_{\circ}^{\beta}(I_T;\R^{2n}); \psi(u)<d \}, \\
		B_{\beta}(\bar u,\delta)=\{u\in \hat L_{\circ}^{\beta}(I_T;\R^{2n});\|u-\bar u\|_{\beta}\le \delta\},\\
		B_{\infty}(\bar u,\varepsilon)=\{u\in \hat L_{\circ}^{\infty}(I_T;\R^{2n});\|u-\bar u\|_{\infty}\le \varepsilon\}.
	\end{gather*}
	For any $\varepsilon>0$, there exists a $\delta >0$ such that, for any
	path component $P$ of $\dot \psi^d$, if
	\begin{equation}\label{e:component-cap-nonempty}
		P\cap B_{\beta}(\bar u,\delta)\neq\emptyset,
	\end{equation} then
	\[P\cap B_{\beta}(\bar u,\delta)\cap B_{\infty}(\bar u,\varepsilon)\neq\emptyset.\]
\end{lemma}
\begin{remark}
	For the critical point found by Theorem \ref{t:critical-point-mp}, the condition (\ref{e:component-cap-nonempty}) can be satisfied.
\end{remark}
\begin{proof}
	We take three steps to prove this lemma.
	
	Step 1: Given $\delta >0$ and let $P$ be
	a path component of $\dot \psi^d$ such that $P\cap B_{\beta}(\bar u,\delta)\neq\emptyset$.
	Let $d(\delta,P)=\inf \{\psi(u);u\in P\cap B_{\beta}(\bar u,\delta)\}$.
	Since $P\subset \dot \psi^d$, we have
	\begin{equation}\label{e:critical-value-small}
		d(\delta,P)<d.
	\end{equation}
	
	We claim that the infimum is achieved at a point in $P\cap B_{\beta}(\bar u,\delta)$.
	Take a minimizing sequence $u_n\in P\cap B_{\beta}(\bar u,\delta)$.
	Since $\hat L_{\circ}^{\beta}(I_T;\R^{2n})$ is reflexive, we can get a subsequence of $u_n$ to be
	weakly convergent to some $\tilde{u}\in \hat L_{\circ}^{\beta}(I_T;\R^{2n})$.
	We want to show that $\tilde u\in P\cap B_{\beta}(\bar u,\delta)$.
	Since $B_{\beta}(\bar u,\delta)$ is closed  and convex, $\tilde u\in B_{\beta}(\bar u,\delta)$.
	In fact, according to Mazur's theorem (\cite[Corollary 3.8]{Bre2011}), there exists a
	sequence  made up of convex combinations of the $u_n$'s that converges strongly
	to $\bar u$.
	Since  $\psi$ is weakly lower semicontinuous, we have
	$\psi(\tilde u)\le \displaystyle\liminf_{n\to \infty}\psi(u_n)= d(\delta,P)$.
	
	We will prove that there is some $n_0$ such that for any $h\in [0,1]$,
	\begin{equation*}
		\psi((1-h)\tilde u+hu_{n_0})<d .
	\end{equation*}
	This means that $\tilde u\in P$.
	
	Arguing by contradiction, we assume that for any $k\in\N$, there exists an $h_k\in[0,1]$
	such that
	$\psi(v_k)\ge d$ where $v_k=(1-h_k)\tilde u+h_ku_k$.
	Since $F^*(t;\cdot)$ is convex, we have
	\begin{equation}\label{e:connect-u-in-P}
		\begin{split}
		d&\le \int_{-\frac{T}{2}}^{\frac{T}{2}}\frac{1}{2} ((-J\Pi+J\Lambda J\Pi^2)v_k,v_k)dt+
		(1-h_k)\int_{-\frac{T}{2}}^{\frac{T}{2}}F^*(t;\tilde u)dt+h_k\int_{-\frac{T}{2}}^{\frac{T}{2}}F^*(t;u_k)dt\\&
		=h_k\psi(u_k)-h_k\int_{-\frac{T}{2}}^{\frac{T}{2}}\frac{1}{2}((-J\Pi+J\Lambda J\Pi^2)u_k,u_k)dt\\
		&\quad +\int_{-\frac{T}{2}}^{\frac{T}{2}}\frac{1}{2} ((-J\Pi+J\Lambda J\Pi^2)v_k,v_k)dt+(1-h_k)\int_{-\frac{T}{2}}^{\frac{T}{2}}F^*(t;\tilde u)dt.
		\end{split}
	\end{equation}
	We may assume, possibly taking a subsequence, that $h_k\to h_{\infty}\in [0,1]$ and $v_k$ is weakly convergent to  $\tilde{u}$ in $\hat   L_{\circ}^{\beta}(I_T;\R^{2n})$.
	Taking the limit in \eqref{e:connect-u-in-P},
	we get
	\[d\le h_{\infty}d(\delta,P)+(1-h_{\infty})\psi(\tilde u)\le d(\delta,P),\]
	which contradicts \eqref{e:critical-value-small}.
	This  proves that $\tilde u\in P$.
	Finally we have $\tilde u\in P\cap B_{\beta}(\bar u,\delta)$.
	
	Note that $P$ is open.
	If $\|\tilde u-\bar u\|_{\beta}<\delta$,
	then $\psi'(\tilde u)=0$.
	
	If $\|\tilde u-\bar u\|_{\beta}=\delta$,
	then there is a Lagrange multiplier $\hat \lambda\in \R$ such that
	\begin{equation}\label{e:lagrange-multiplier}
		(\hat \lambda(\tilde u-\bar u)^{\beta -1}+\psi'(\tilde u),w)=0\ \ \ \forall w\in \hat L_{\circ}^{\beta}(I_T;\R^{2n}),
	\end{equation}
	where we denote by $y^{\beta-1}=y|y|^{\beta-2}$ for a vector $y\in\R^{2n}$.
	Then for any minimum $\tilde u=\tilde u_{P,\delta}\in P\cap B_{\beta}(\bar u,\delta)$,
	there is a real number  $\hat \lambda=\hat \lambda_{\delta,P}$ such that
	\eqref{e:lagrange-multiplier} holds.
	
	Since $(\psi'(\tilde u),\bar u-\tilde u)=\frac{d}{dh}|_{h=0}\psi((1-h)\tilde u+h\bar u)\ge 0$,
	taking $w=\tilde u-\bar u$ in \eqref{e:lagrange-multiplier}, we get
	$\hat \lambda \ge 0$.
	Using the explicit expression of $\psi'(\tilde u)$, we get that there exists a
	constant vector $\tilde \xi=\tilde \xi_{\delta,P}\in \{0\}\times \R^n$
	such that
	\begin{equation}\label{e:lagrange-multiplier-explicit}
		\hat \lambda(\tilde u-\bar u)^{\beta-1}-J\Pi \tilde u+J\Lambda J\Pi^2\tilde u+\nabla F^*(t;\tilde u)=\tilde \xi.
	\end{equation}
	The fact $\psi'(\bar u)=0$ means that there exists $\bar \xi \in \{0\}\times \R^n$
	such that
	\begin{equation}\label{e:critical-point-explicit}
		-J\Pi \bar u+J\Lambda J\Pi^2\bar u+\nabla F^*(t;\bar u)=\bar \xi.
	\end{equation}
	Adding \eqref{e:lagrange-multiplier-explicit} and \eqref{e:critical-point-explicit},
	we get
	\begin{equation}\label{e:necessary-condition-u}
		\hat \lambda(\tilde u-\bar u)^{\beta-1}+\nabla F^*(t;\tilde u)-\nabla F^*(t;\bar u)
		-J\Pi (\tilde u-\bar u)+J\Lambda J\Pi^2(\tilde u-\bar u)=\tilde \xi-\bar \xi.
	\end{equation}

	Step 2: We are going to prove
	
	{\bf Claim}:
	Given a sequence of pair $(\delta_n,P_n)$ such that $\delta_n\to0$ and
	$B_{\beta}(\bar u,\delta_n)\cap P_n\neq \emptyset$,
	the corresponding sequence $\tilde u_{\delta_n,P_n}$ consisting of the minimizers
	is uniformly bounded and converges to $\bar u$ in $\hat L_{\circ}^{\infty}(I_T;\R^{2n})$ as $n\to \infty$.
	
	Let $\tilde u_n=\tilde u_{\delta_n,P_n}$, $\lambda_n=\hat \lambda_{\delta_n,P_n}$
	and $\tilde \xi_n=\tilde \xi_{\delta_n,P_n}$ where
	$\hat \lambda_{\delta_n,P_n}$ and $\tilde \xi_{\delta_n,P_n}$ were defined in Step 1.
	Writing these into
	\eqref{e:necessary-condition-u}, we get
	\begin{equation}\label{e:necessary-condition-u-n}
		\lambda_n(\tilde u_n-\bar u)^{\beta-1}+\nabla F^*(t;\tilde u_n)-\nabla F^*(t;\bar u)
		-J\Pi (\tilde u_n-\bar u)+J\Lambda J\Pi^2(\tilde u_n-\bar u)=\tilde \xi_n-\bar \xi.
	\end{equation}
	Taking the $L^2$-inner product with $\frac{\tilde u_n-\bar u}{\|\tilde u_n-\bar u\|_{\beta}}$
	on both sides of \eqref{e:necessary-condition-u-n}, we get
	\begin{equation*}
		\begin{split}
			\lambda_n\|\tilde u_n-\bar u\|^{\beta-1}_{\beta}&+
			(\nabla F^*(t;\tilde u_n)-\nabla F^*(t;\bar u),\tilde u_n-\bar u)\|\tilde u_n-\bar u\|^{-1}_{\beta}\\
			& +(-J\Pi(\tilde u_n-\bar u)+J\Lambda J\Pi^2(\tilde u_n-\bar u),\tilde u_n-\bar u)\|\tilde u_n-\bar u\|^{-1}_{\beta}=0.
		\end{split}
	\end{equation*}
	Since $\nabla F^*(t;\cdot)\colon \hat L^{\beta}_{\circ}(I_T;\R^{2n})\to \hat L^{\alpha}(I_T;\R^{2n})$ is continuous
	and $\Pi\colon \hat L^{\beta}_{\circ}(I_T;\R^{2n})\to \hat L_{\circ}^{\alpha}(I_T;\R^{2n})$ is compact, we get
	\[\lambda_n\|\tilde u_n-\bar u\|^{\beta-1}_{\beta}\to 0 \text{ when } n\to \infty.\]
	Writing this back into \eqref{e:necessary-condition-u-n}, we have
	$\tilde \xi_n\to \bar \xi$ when $n\to \infty$.
	Rewrite \eqref{e:necessary-condition-u-n} as the following form:
	\begin{gather}
		\nabla F^*(t;\tilde u_n) + \lambda_n(\tilde u_n-\bar u)^{\beta -1}=f_n \label{e:tildeu-condition}\\
		\text{where } f_n= \nabla F^*(t;\bar u)+J\Pi(\tilde u_n-\bar u)-J\Lambda J\Pi^2(\tilde u_n-\bar u)+ \tilde \xi_n-\bar \xi.\nonumber
	\end{gather}
	Then $f_n$ converges uniformly to $\nabla F^*(t;\bar u(t))$.
	Note that the left-hand side of \eqref{e:tildeu-condition}
	is the gradient at $y=\tilde u_n(t)$ of the following convex function about $y$
	\[F^*(t;y)+\frac{\lambda_n}{\beta}|y-\bar u(t)|^{\beta} \text{  for any } t\in \R.\]
	Hence, we write that the function lies above its tangent hyperplane at $\tilde  u_n(t)$:
	\begin{equation}\label{e:H^*-convex-tildeu}
		\begin{split}
			F^*(t;\bar u(t))\ge & F^*(t,\tilde  u_n(t))+\frac{\lambda_n}{\beta}|y-\bar u(t)|^{\beta}\\
			&+(f_n(t),\bar u(t)-\tilde  u_n(t)).
		\end{split}
	\end{equation}
	For $t\in I_T$ either $|\tilde  u_n(t)|\le R$
	or $|\tilde  u_n(t)|> R$, where $R$ was defined in \eqref{e:low-up-bound-H*}.
	For the second case, we use the lower bound of $F^*$ \eqref{e:low-up-bound-H*} to get
	\[F^*(t;\bar u(t))\ge \frac{1}{\beta k^{\beta}}|\tilde  u_n(t)|^{\beta}+(f_n(t),\bar u(t)-\tilde  u_n(t)).\]
	Together with the uniform boundedness of
	the continuous function sequence $f_n$,
	we get $\tilde u_n$ is uniformly bounded in $L^{\infty}(I_T;\R^{2n})$.
	
	We claim that $\|\tilde  u_n-\bar u\|_{\infty}\to0$ as $n\to \infty$.
	If the claim is not true, then we can find a subsequence of
	$\tilde  u_n$ which is still denoted by itself, and a sequence $t_n\in I_T$ such that
	\[\tilde  u_n(t_n)-\bar u(t_n)\to \bar y \neq 0\quad  \text{and} \quad t_n\to \bar t \quad \text{ as }\quad n\to \infty.\]
	Since $\lambda_n\ge 0$,  from inequality \eqref{e:H^*-convex-tildeu}, we get
	the limit satisfies
	\[F^*(t;\bar u(\bar t ))\ge F^*(t;\bar u(\bar t )+\bar y)-(\nabla F^*(t;\bar u(\bar t)),\bar y),\]
	which we write as
	\[F^*(t;\bar u(\bar t )+\bar y)\le F^*(t;\bar u(\bar t ))+(\nabla F^*(t;\bar u(\bar t)),\bar y) .\]
	Since the Hessian  of $F^*(t;\cdot)$ is positive definite, we get $\bar y=0$, which is a contradiction.
	
	Step 3:
	Assume our lemma is false. Then there is an $\varepsilon>0$, a sequence  $\delta_n\to0$
	and a sequence $P_n$ of path components of $\dot \psi^d$, such that
	$P_n\cap B_{\beta}(\bar u,\delta_n)\neq \emptyset$
	and
	\begin{equation}\label{e:beta-infinity-empty}
		P_n\cap B_{\beta}(\bar u,\delta_n)\cap B_{\infty}(\bar u,\varepsilon)= \emptyset.
	\end{equation}
	Since we have found a minimizer $\tilde u_n\in P_n\cap B_{\beta}(\bar u,\delta_n)$
	of $\psi$ such that $\|\tilde u_n-\tilde u\|_{\infty}\to 0$ as $n\to \infty$, this contradicts
	\eqref{e:beta-infinity-empty}. The proof is complete.
\end{proof}

Lemma \ref{l:L-beta-infinity} has some useful consequences.
Using the similar method in \cite[IV.3]{Ek90}, we get the following two conclusions
by analogy with Corollaries IV.3.2 and IV.3.4 in loc. cit.
\begin{proposition}\label{p:minimal-infinity-beta}
	$\bar u\in \hat L_{\circ}^{\infty}(I_T;\R^{2n})$ is a local minimum for $\psi_{\infty}$ if and only if it is a local minimum of $\psi$.
\end{proposition}
\begin{proposition}\label{p:mountainpass-infinity-beta}
	If $\bar u\in \hat L_{\circ}^{\infty}(I_T;\R^{2n})$ is a mountain-pass point for $\psi$, then it is a mountain-pass point for $\psi_{\infty}$.
\end{proposition}

Now we study the $C^2$ function $\psi_{\infty}$.
Since $F''(t,\bar x(t))$ is positive definite for any $t\in I_T$,
by \cite[Proposition I.4.2]{Ek90}, $q_{T/2}$ induces a $q_{T/2}$-orthogonal splitting of
$\hat L^2_{\circ}(I_T;\R^{2n})\colon$
\[\hat L^2_{\circ}(I_T;\R^{2n})=E_-\oplus E_0\oplus E_+\]
such that $E_0$ is the null-subspace of $q_{T/2}$,
$q_{T/2}$ is positive definite and negative definite on $E_+$ and $E_-$ respectively.
By the spectral theory of a compact selfadjoint operator on Hilbert space, $E_-$ and $E_0$ are finite-dimensional. 

By regularity,
we have a topological splitting of $\hat L^{\infty}_{\circ}(I_T;\R^{2n})\colon$
\begin{equation}\label{e:spitting-L-infinity}
	\hat L^{\infty}_{\circ}(I_T;\R^{2n})=E_-\oplus E_0\oplus E^{\infty}_+
\end{equation}
where $E^{\infty}_+=E_+\cap \hat L^{\infty}_{\circ}(I_T;\R^{2n})$.
In fact, for $u\in \hat L^2_{\circ}(I_T;\R^{2n})$, if 
\[-J\Pi u +J\Lambda J\Pi^2u+  (F^{''}(t,\bar x))^{-1}u=\tilde \lambda u+\xi,\] for some $\tilde \lambda\le 0$ and $\xi\in\{0\}\times \R^n$, then let $x=J\Pi u-J\Lambda J\Pi^2u+\xi$ and we have
\[-J\dot x+\tilde \lambda F''(t,\bar x(t))J\dot x= [F''(t,\bar x(t))-\Lambda+\tilde \lambda F''(t,\bar x(t))\Lambda]x.\]
Since $-J+\tilde \lambda F''(t,\bar x(t))J$ is invertible for $\tilde\lambda\le 0$,
we obtain $x\in C^1(I_T;\R^{2n})$ and finally get $u=-J\dot x+\Lambda x\in \hat L^{\infty}_{\circ}(I_T;\R^{2n})$.

We denote by $P_-,P_0$ and $P_+$ the corresponding projections about the splitting \eqref{e:spitting-L-infinity}.
Denote by
\[u_+=P_+u\in E_+, u_0=P_0u\in E_0 \text{ and } u_-=P_-u\in E_- \text{ for } u\in \hat L^{\infty}_{\circ}(I_T;\R^{2n}).\]
We have
\begin{lemma}\label{l:E-0-to-+}
	There exists a neighborhood $\mathcal{V}$ of $\bar u_-+\bar u_0$ in $E_-\oplus E_0$ and a
	$C^1$ map $\sigma \colon \mathcal{V}\to E^{\infty}_+$ satisfying
	$\langle \psi'_{\infty}(v+\sigma(v)),w\rangle=0$, for $v\in \mathcal{V}$ and $w\in E^{\infty}_+$;
	$\sigma(\bar u_-+\bar u_0)=\bar u_+$,  $\sigma'(\bar u_-+\bar u_0)=0$.
	The new functional $\hat \psi_{\infty}(v):=\psi_{\infty}(v+\sigma(v))$ is $C^2$.
\end{lemma}
\begin{proof}
	Define a map $S\colon \hat L^{\infty}_{\circ}(I_T;\R^{2n}) \to \hat L^{\infty}_{\circ}(I_T;\R^{2n})$ by
	\[S(u):=-J\Pi u+J\Lambda J\Pi^2u+\nabla F^*(t;u)-\frac{1}{T}\int_{-\frac{T}{2}}^{\frac{T}{2}} \nabla F^*(t;u(t))dt .\]
Then we have
	\[\langle \psi'_{\infty}(u),v\rangle=\int_{-\frac{T}{2}}^{\frac{T}{2}}(S(u),v)dt.\]
	Clearly $S$ is a $C^1$ map near $\bar u$ and
	\[S'(\bar u)v=-J\Pi v +J\Lambda J\Pi^2v+  (F^*)^{''}(t;\bar u)v+\frac{1}{T}\int^{\frac{T}{2}}_{-\frac{T}{2}}\left[ J\Pi v -J\Lambda J\Pi^2v-(F^*)^{''}(t;\bar u)v\right]dt.\]
	Similar to \cite[Page 19]{EkHo87}, we have
	\[S'(\bar u) E^{\infty}_+ \subset  E^{\infty}_+.\]
	Moreover, $S'(\bar u)\colon E^{\infty}_+ \to  E^{\infty}_+$ is an isomorphism.
	One step is the regularity, that is,
	for $S'(\bar u)v\in \hat L^{\infty}_{\circ}(I_T;\R^{2n})$ where $v\in \hat L^2_{\circ}(I_T;\R^{2n})$,
	then by the explicit expression of $S'(\bar u)$ we get
	$v\in \hat L^{\infty}_{\circ}(I_T;\R^{2n})$.
	
	We use the implicit function theorem to find a neighborhood of $\mathcal{V}$
	of $\bar u_-+\bar u_0$ in $E_-\oplus E_0$ and a unique $C^1$ map
	$\sigma\colon \mathcal{V}\to E^{\infty}_+$ such that
	\[P_+S(v+\sigma(v))=0 \text{ for $v\in \mathcal{V}$ and } \sigma(\bar u_-+\bar u_0)=\bar u_+.\]
	
	Following the lines in the proof of \cite[Lemma IV.3.5]{Ek90}, we get this lemma.
	Note that $\hat \psi'_{\infty}(v)=(P_-+P_0)S(v+\sigma(v))$ is a $C^1$ map.
\end{proof}
We recall the following normal form theorem.
\begin{theorem}[cf. {\cite[Chapter IV.3]{Ek90}}]\label{t:normal-form}
	Let $\mathcal{U}$ be an open neighborhood of $0$ in a Hilbert space $V$.
	Assume that $\varphi\in C^2(\mathcal{U},\R)$ and $0$ is a critical point of $\varphi$. Suppose that $L=\varphi''(0)$ is a Fredholm operator, so that $V$ splits orthogonally into positive, negative and null subspaces relative to $\varphi''(0)\colon$
	\[V=E_+\oplus E_-\oplus E_0\]
	with $E_0=\ker L$ and $E_+\oplus E_- =\im L$. Then there exists an open neighborhood $\mathcal{V}$
	of $0$ in $\im L$,
	and an open neighborhood $\mathcal{W}$ of $0$ in $\ker L$, a local homeomorphism $h$
	from $\mathcal{V}\times \mathcal{W} $ into $\mathcal{U}$ with $h(0,0)=0$, and a function
	$f\in C^2(\mathcal{W},\R)$ with $f'(0)=0$, $f''(0)=0$ such that
	\[\varphi(h(v,w))=\frac{1}{2}(Lv,v)+f(w)\ \ \ \text{for}\ \ (v,w)\in \mathcal{V}\times \mathcal{W}, \]
	\begin{equation}\label{e:restriction-of-h}
		\text{and the restriction }\ \ h(\cdot,0) \colon \mathcal{V} \to \im L \quad \text{is a local  $C^1$-diffeomorphism}.
	\end{equation}
\end{theorem}
\begin{remark}
	According to Remark 1) after \cite[Theorem 8.3]{MaWi89}, we get \eqref{e:restriction-of-h}
	which also will be used in the following lemma.
\end{remark}
Applying the normal form theorem to the functional
$\hat \psi_{\infty}(\cdot+\bar u_-+\bar u_0)\colon E_-\oplus E_0\to \R$,
we get
the corresponding maps $h$ and $f$. In this case, we have $\im L=E_-$, $\ker L=E_0$ and $f(0)=d$.

Denote by $u\sim u'$ if $u$ and $u'$ belong to the same path component of the sublevel set $\dot \psi^d_{\infty}$.
This also defines an equivalent relation in $\dot\psi^d_{\infty}$.

By the same proofs in \cite[Lemma IV.3.7 and Corollary IV.3.8]{Ek90},
we have the following lemma in our reversible and adjusted variational framework.
\begin{lemma}\label{l:dim-negative-morethan-one}
	Assume $\dim E_-\ge 1$.
	\begin{enumerate}
		\item[(a)] Set $\tilde{u}=\bar u_-+\bar u_0\in E_-\oplus E_0$. There exists a neighborhood $\mathcal{U}$ of $\bar u$ in $\hat L^{\infty}_{\circ}(I_T;\R^{2n})$ such that for $u\in \mathcal{U}\cap \dot \psi^d_{\infty}$,  there exists some
		$v\in E_-\setminus \{0\}$ such that $u \sim \tilde{u}+h(v)+\sigma(\tilde{u}+h(v))$.
		\item[(b)] There is a neighborhood $\mathcal{U}$ of $\bar u$ in $\hat L^{\infty}_{\circ}(I_T;\R^{2n})$
		with the following property: for every $u\in \mathcal{U}\cap \dot\psi^d_{\infty}$
		we can find $v\in E_-\setminus\{0\}$ and some $\rho>0$ such that
		$\bar u+tv\sim u$, where $t$ is any positive number smaller than $\rho$.
	\end{enumerate}
\end{lemma}
\begin{proof}
	(a) is almost same as in the proof of \cite[Lemma IV.3.7]{Ek90} and we just stretch the key points.
	By the definition of $\sigma$, there exists $\delta>0$ such that
	\begin{equation}\label{e:min-convex}
		||\sigma(\tilde u+v)-\bar u_+||_{\infty}<\delta \implies
		\hat \psi_{\infty}(\tilde u+v)=\min\{\psi_{\infty}(\tilde u+v+\sigma(\tilde u+v)+w);||w||_{\infty}<\delta\}.
	\end{equation}
	We may choose such a small $\eta$ above (68) of Page 166 in loc. cit.  to fall within the neighborhood of the origin in $E_-\oplus E_0$ specified by Theorem \ref{t:normal-form}, and to make  $\|\sigma(v+\tilde u)-\bar u_+\|_{\infty}<\delta$ hold when $v=h(v')$ and $\|v'\|_{\infty}<\eta$.

	We then define an open neighborhood $\mathcal{U}$
	of $\bar u$ in $\hat L^{\infty}_{\circ}(I_T;\R^{2n})$ by:
	\[\mathcal{U}=\{\tilde{u}+h(v)+\sigma(\tilde u+h(v))+w;v\in E_-\oplus E_{\circ},\|v\|_{\infty}<\eta; w\in E_+,\|w\|_{\infty}<\delta\}.\]
	Take any $u\in \mathcal{U}\cap \dot \psi^d$, so that
	\[u=\tilde{u}+h(v)+\sigma(\tilde u+h(v))+w.\]
	Recall $u\sim u'$ if $u$ and $u'$ belong to the same path component of $\dot\psi^d$. By \eqref{e:min-convex}, we have
	\[u\sim u':=\tilde{u}+h(v)+\sigma(\tilde u+h(v)).\]
In fact, by \eqref{e:min-convex} and the monotropy of convex function, we have
	\[\hat \psi_{\infty}(\tilde u+h(v))\le \psi(\tilde u+h(v)+\sigma(\tilde u+h(v))+tw)\le \psi(u) \quad \text{for} \quad t\in[0,1].\]
	
	By the definition of $\hat\psi_{\infty}$, we have
	\begin{equation*}
		\psi_{\infty}(u')=\hat\psi_{\infty}(\tilde u+h(v)).
	\end{equation*}
	
	Define $\tilde v_-\in E_-$ by:
	\begin{equation*}
		\tilde v_-=\begin{cases}
			v_- &\text{if $\|v_-\|_{\infty}\ge \frac{\eta}{2}$},\\
			\frac{\eta}{2\|v_-\|_{\infty}}v_- &\text{if $0<\|v_-\|_{\infty}< \frac{\eta}{2}$},\\
			 \frac{\eta}{2}e_-&\text{if $v_-=0$},\\
		\end{cases}
	\end{equation*}
	where we choose $e_-\in E_-$ with $\|e_-\|_{\infty}=1$, and set
	\begin{equation*}
		u'':=\tilde u+h(\tilde v_-+v_0)+\sigma(\tilde u+h(\tilde v_-+v_0)).
	\end{equation*}
	Finally, we  can claim that
	\begin{equation*}
		u'\sim u'' \sim u''' :=\tilde u +h(\tilde v_-)+\sigma(\tilde u+h(\tilde v_-)).
	\end{equation*}
We just follow the proof of \cite[Lemma IV.3.7]{Ek90},
where the reader can also find how to prove (b). Note that there is a tiny typo in the last line of Page 168 in loc. cit.:
	\[\psi_{\infty}\left[\bar u+t_n\left((1-\lambda_n)c'(0)+\lambda_n\frac{c(t_n)-\bar u}{t_n}\right)\right]\ge \psi_{\infty}(\bar u).\]
\end{proof}
\begin{corollary}\label{c:dim-greater-two}
	If $\dim E_-\ge 2$, then there exists a neighborhood $\mathcal{U}$ of $\bar u$
	in $\hat L^{\infty}_{\circ}(I_T;\R^{2n})$ such that $\dot \psi^d_{\infty}\cap \mathcal{U}$
	is path-connected.
\end{corollary}
\begin{proof}
	By Lemma \ref{l:dim-negative-morethan-one}(a) and noticing that $\dim E_-\ge 2$,
	we get this corollary.
\end{proof}

We recall \cite[Theorem 3.1]{EkHo87} for our later use.
\begin{theorem}\label{t:two-path-components}
	Let $\bar u\in \Cr(\psi,d)$ be a mountain-pass essential point, and let $\bar x$ be the corresponding $T$-periodic brake solution of problem \eqref{e:HS-T-K}.
	Then \begin{equation}\label{e:morse-index-less1-greater1}
	m^-(\bar x)\le 1\le m^-(\bar x)+m^0(\bar x).
	\end{equation}
	If $m^-(\bar x)= 1$, the set
	\[\{\mathscr{P}\ ; \mathscr{P}\text { is a path component  of } \dot \psi^d \text{\ \ and\  \ } \bar u \in \cl(\mathscr{P}) \}\]
	has exactly two different elements $\mathscr{P}^+$ and $\mathscr{P}^-$. Moreover,
	there is an eigenvector $e\neq 0$ associated with the negative eigenvalues of $\psi_{\infty}''(\bar u)$ and some $\eta >0$
	such that for $0<s<\eta$, $\bar u+se\in \mathscr{P}^+$ and $\bar u-se\in \mathscr{P}^-$.
\end{theorem}
\begin{proof}
	Note that $\bar u\in \Cr(\psi,d)$ is a mountain-pass essential point.
	If $\bar u$ is a mountain-pass point, then by
	Proposition \ref{p:mountainpass-infinity-beta} and Corollary \ref{c:dim-greater-two},
	we have $m^-(\bar x)\le 1$.
	If $\bar u$ is a limit of a sequence of local minima $u_n$ at the level $d$, we have $m^-(\bar x)=0$, because
	the Morse index of local minimum is zero, i.e., $m^-(u_n)=0$ and the sequence of quadratic forms corresponding to $u_n$ converges the quadratic form to corresponding to  $\bar u$.
	By the Morse lemma, if $m^-(\bar u)=0$ and $m^0(\bar u)=0$,
	$\bar u$ would be a local minimum. So $m^-(\bar x)+ m^0(\bar x)\ge 1$.
	Thus if $m^-(\bar x)= 1$, $\bar u$ is a mountain-pass point.
	By Lemma \ref{l:dim-negative-morethan-one}, we obtain our theorem.
	The proof can be found in \cite[Theorem 3.1]{EkHo87} and the detail is omitted. For the convenience of readers, we give a remark here.
	As in \cite[page 22]{EkHo87}, define
	\begin{equation}\label{e:open-neighborhood}
		\tilde{W}=B_{\tau}(\bar u)\cup (\cup\{\text{path components of $\dot \psi^d$ intersecting $B_{\tau}(\bar u)$}\}),
	\end{equation}
	where $B_{\tau}(\bar u)=\{u\in \hat L_{\circ}^{\beta}(I_T;\R^{2n});||u-\bar u||_{\beta}\le \tau\}$. Then $\tilde{W}$ is an open neighborhood of $\bar u$ in
	$\hat L_{\circ}^{\beta}(I_T;\R^{2n})$. Since $\bar u$ is a mountain-pass point, by Lemma \ref{l:dim-negative-morethan-one}(b), we have
	\[\tilde{W}\cap \dot \psi^d=\mathscr{P}^+\cup \mathscr{P}^-,\] where $\mathscr{P}^+$ and $\mathscr{P}^-$ two different path components of
	$\tilde{W}\cap \dot \psi^d$ such that there is an eigenvector $e\neq0$ associated with the negative eigenvalue 
	of $\psi_{\infty}''(\bar u)$ and
	for all $t'>0$ sufficient small, $\bar u+t'e\in \mathscr{P}^+$ and $\bar u-t'e\in \mathscr{P}^-$. Now we claim that $\mathscr{P}^+$ and $\mathscr{P}^-$ are also two different path components of $\dot \psi^d$. In fact, since $\mathscr{P}^+\subset \dot \psi^d$ is path-connected, there exists a path component of $\dot \psi^d$, which is denoted by $\mathscr{P}^*$, such that $\mathscr{P}^+\subset \mathscr{P}^*$. By the definition of $\tilde{W}$ \eqref{e:open-neighborhood}, we have that $\mathscr{P}^*\cap B_{\tau}(\bar u)$ is nonempty, thus $\mathscr{P}^*\subset \tilde{W}\cap \dot \psi^d$; finally we get $\mathscr{P}^*=\mathscr{P}^+$.
	
\end{proof}

\begin{proposition}\label{p:morse-index--1}
	For $j\in\N$,
	let $\psi\colon \hat L_{\circ}^{\beta}(I_{2jT};\R^{2n})\to\R$ be the functional as in \eqref{e:reduced-functional}
	\[\psi(u)=\int_{-jT}^{jT}\left[ \frac{1}{2}(-J\Pi u+J\Lambda J\Pi^2u,u)+F^*(t,u)\right]dt.\]
	Assume $\bar u$ is a mountain-pass essential point of $\psi$. Let $q_{jT}$ be the quadratic form  defined  on 
	$\hat L_{\circ}^2(I_{2jT};\R^{2n})$ in \eqref{e:quadratic-form-T} corresponding $\bar u$ and $\psi$. If 
	\begin{enumerate}
		\item[\rm (i)] $\bar u(t)=\bar u(t-jT)$ for $t\in[0,jT]$; and
		\item[\rm (ii)] there is an $e\in \hat L_{\circ}^2(I_{2jT};\R^{2n}) $ such that $q_{jT}(e,e)<0$ and $e(t)=-e(t-jT)$ for $t\in [0,jT]$,
	\end{enumerate}
	then $m^-(\bar u)\neq1$.
\end{proposition}
\begin{proof}
We will show that if the conclusion were not true,  a contradiction would occur.
Now we assume $m^-(\bar u)=1$.
	Since $0 \xrightarrow[\psi]{}\bar u$, by Theorem
	\ref{t:two-path-components}, we have $0\in \mathscr{P}^+$ or $0\in \mathscr{P}^-$,
	where $\mathscr{P}^{\pm}$ is the path component of $\dot \psi^d$ whose closure contains $\bar u$ and $\mathscr{P}^+\cap\mathscr{P}^- =\emptyset$. Furthermore, there is a positive number $\eta$
	such that for $0<t'\le \eta $, $\bar u+t'e\in \mathscr{P}^+$ and $\bar u-t'e\in \mathscr{P}^-$.
	Without loss of generality we can assume $0\in \mathscr{P}^+$.
	Let
	$c_0$ be a path in $\mathscr{P}^+$ connecting $0$ to $\bar u+\eta e$:
	\[c_0(0)=0 \ \ \ \text{and}\ \ \ c_0(1)=\bar u+\eta e.\]
	Then $\psi(c_0(s))<d$ for $0\le s\le 1$.
	 For $s\in[0,1]$, $c_0(s)\in \hat L^{\beta}_{\circ}(I_{2jT};\R^{2n})$.
	Now we define  the $jT$-time shift path of $c_0(s)$ for $s\in[0,1]$,
	by
	\begin{equation}\label{e:shift-of-path}
		c_1(s)(t)=
		\begin{cases}
			c_0(s)(t-jT)&  t\in[0,jT], \\
			c_0(s)(t+jT)& t\in[-jT,0].
		\end{cases}
	\end{equation}
	By definition and assumptions {\rm (i)(ii)}, we get
	\[c_1(0)=0 \ \ \ \text{and}\ \ \ c_1(1)=\bar u-\eta e.\]
	Note that for any $s_1,s_2\in [0,1]$, $\|c_0(s_1)-c_0(s_2)\|_{\beta}=\|c_1(s_1)-c_1(s_2)\|_{\beta}$.
	For any $s\in[0,1]$, since $c_0(s)\in \hat L_{\circ}^{\beta}(I_{2jT};\R^{2n})$, we have
	$c_1(s)\in \hat L_{\circ}^{\beta}(I_{2jT};\R^{2n})$ and $\psi(c_0(s))=\psi(c_1(s))<d$.
	So $\bar u-\eta e$ lies in the path component of the origin in $\dot\psi^d$, that is,
	$\bar u-\eta e\in \mathscr{P}^+$.
	We get a contradiction. 
	Thus $m^-(\bar u)\neq1$.
	
	In fact, $-J\frac{d}{dt}+\Lambda$ induces an isomorphism from $\hat W_{2jT}^{1,\beta}(I_{2jT};\R^{2n})$ onto $\hat L_{\circ}^{\beta}(I_{2jT};\R^{2n})$. 
	For $s\in[0,1]$, let $\tilde c_0(s)\in \hat W_{2jT}^{1,\beta}(I_{2jT};\R^{2n})$ satisfies
	$-J\frac{d}{dt}\tilde c_0(s)+\Lambda \tilde c_0(s)=c_0(s)$.
	Then the $jT$-time shift path of $\tilde c_0(s)$ is  absolutely continuous on $[-jT,jT]$, which is 
	similarly defined by \eqref{e:shift-of-path} and denoted by $\tilde c_1(s)$.
  Then  $-J\frac{d}{dt}\tilde c_1(s)+\Lambda \tilde c_1(s)=c_1(s)$ and $\tilde c_1(s)\in\hat W_{2jT}^{1,\beta}(I_{2jT};\R^{2n})$.
   By simple computation,  one can get
	\[\psi(c_0(s))=\Psi_K(\tilde c_0(s))=\Psi_K(\tilde c_1(s))=\psi(c_1(s)).\] 
	\end{proof}.

\section{Iteration theory of Maslov-type indices: old and new}\label{s:Maslov-index}
\subsection{
Iteration formulae and  inequalities of the Maslov-type index theory}\label{ss:Maslov-index-iteration}
In order to understand iterations of periodic solutions for nonlinear Hamiltonian systems, one effective
way is to study  Maslov-type indices of iterations of the matrizant for corresponding linearized systems.
In the following proof of our main results, we will need the assistance of the iteration theory of Maslov-type indices.
Given a periodic solution $\bar x$ of the Hamiltonian system
\eqref{e:HS-jT}, we consider the corresponding linear Hamiltonian system $\dot y=JH''(t,\bar x(t))y$ and define the Maslov-type indices of $\bar x$.

In this section, we recall the Maslov-type indices for symplectic paths with Lagrangian boundary conditions, prove some  useful iteration inequalities and calculate the H\"{o}rmander index explicitly.
Let $\mathcal{L}_s(\R^{2n})$ be the space of symmetric $2n\times 2n$  real matrices and
let $\Sp(2n)=\Sp(2n,\R)=\{M\in \GL(2n,\R); M^TJM=J\}$ be the group of symplectic matrices in $\R^{2n}$, where $\GL(2n,\R)$ denotes the group of invertible
$2n\times 2n$ matrices with real entries and the expression $M^T$ denotes the transpose of the matrix $M$. 

We denote the set of paths in the symplectic group starting form the identity matrix by
\[\mathcal P_{\tau}(2n)=\mathcal P_{\tau}(2n,\R)=\{\gamma\in C([0,\tau],\Sp(2n));\gamma(0)=I_{2n}\},\]
where $\tau$ is a positive number.

Let $\gamma_1\in \mathcal P_{\tau_1}(2n)$ and
$\gamma_2\in \mathcal P_{\tau_2}(2n)$ be two symplectic paths
starting from $I_{2n}$, where $\tau_1,\tau_2>0$.
The \textit{iteration} of $\gamma_1$ and $\gamma_2$ is a symplectic path $\tilde \gamma \in P_{\tau_1+\tau_2}(2n)$
defined by 
\begin{equation}\label{e:two-iterations}
	\tilde \gamma(t)=
	\begin{cases}
	\gamma_1(t), & 0\le t \le \tau_1,\\
	\gamma_2(t-\tau_1)\gamma_1(\tau_1),& \tau_1\le t\le \tau_1+\tau_2.
	\end{cases}
\end{equation}

Set $S_{\tau}=\R/\tau\Z$.
Given a path $B\in C(S_{\tau},\mathcal{L}_s(\R^{2n}))$,
by $\gamma(t)$
we denote the particular fundamental matrix solution of the linear
system of differential equations
\begin{equation}\label{e:linear-path}
	\dot x =J B(t)x, \quad x\in\R^{2n},\\
\end{equation}
with $\gamma(0)=I_{2n}$.
We shall refer to the unique matrix function $\gamma(t)$ above as the \textit{matrizant} of the system \eqref{e:linear-path}.
Since the matrix $B(t)$ is symmetric for any $t\in S_{\tau}$, $\gamma\in \mathcal P_{\tau}(2n)$.

Let $(V,\omega)$ be a finite-dimensional complex symplectic linear space, where $\omega$ is the symplectic form, i.e., a nondegenerate  skew-Hermitian $2$-form. 
A linear subspace $\lambda$ of $(V,\omega)$
is called Lagrangian if $\lambda^{\omega}= \lambda$,
where $ \lambda^{\omega}$ is defined by
\[ \lambda^{\omega}=\{u\in V; \omega(u,v)=0 \text{ for all } v\in  \lambda\}.\]
We denote the set of Lagrangian subspaces of symplectic linear space $(V,\omega)$ by $\Lag(V)$.
We denote the set of symplectic linear transformations from $V$
to itself by
\[\Sp(V,\omega):=\{M\in \GL(V);M^*\omega=\omega\},\]
where $\GL(V)$ denote the set of invertible linear transformations from $V$ to itself and $M^*\omega$ is the pullback $2$-form given by $M^*\omega(u,v)=\omega(Mu,Mv)$ for $u,v\in V$.
Let $(\C^{2n},\omega_0)$ be the standard complex symplectic vector space, where  the symplectic form is defined by
\begin{equation*}
	\omega_0(u,v)= J\bar u\cdot v,\quad \text{ for all } u,v\in \C^{2n}.
\end{equation*}

Note $\Sp(2n,\R)$ can be seen as a subset of $\Sp(\C^{2n},\omega_0)$.
In the following, we always consider the symplectic space $V=\C^{2n}\oplus \C^{2n}$ and $\omega=(-\omega_0)\oplus\omega_0$. 
Then for any $M\in$ {\rm Sp}$(2n)$, its graph
$\Graph(M)=\{(x,Mx)\subset V;x\in \C^{2n}\}\in \Lag(V).$
Denote by $\alpha_0=\{0\} \times \C^n\subset \C^{2n}$
and $\alpha_1=\C^n \times\{0\}\subset \C^{2n}$.
Then  $\alpha_i\times \alpha_j\in \Lag(V)$, where $i,j\in\{0,1\}$.
As in \cite{zhou-zhu-wu}, we write $\tilde \alpha_j=\alpha_j\times \alpha_j$ for $j=0,1$.
We also 
denote $L_0=\{0\} \times \R^n\subset \R^{2n}$
and $L_1=\R^n \times\{0\}\subset \R^{2n}$.
Note that $\alpha_i=L_i\otimes \C$, $i=0,1$. 
By $\dim(\cdot)$ and $\dim_{\C}(\cdot)$, we denote the dimension of a real linear space and the complex dimension of a complex linear space respectively.

We follow Cappell, Lee and Miller's pioneer paper \cite{C-Lee-M} (cf. \cite{duistermaat,Robin}) to define the Maslov index.
\begin{definition}\label{d:Maslov-type-index}
	For $\tau>0$,
	let $(\lambda(t),\mu(t))$, $t\in[0,\tau]$, be a path of pairs of Lagrangian subspaces of $(V,\omega)$.
	We denote the \textit{Maslov index} of the path $(\lambda(t),\mu(t))$, $t\in[0,\tau]$, by $\Mas\{\lambda,\mu\}$ as in \cite{zhou-zhu-wu}, where we work in the complex symplectic space.
	 Let $\gamma\in\mathcal{P}_{\tau}(2n)$ be a symplectic path.
	  For $W\in\Lag(V)$, we define the Maslov-type index
	\[i_{W}(\gamma)=\Mas\{\Graph(\gamma),W\},\]
	where $\Graph(\gamma(t))=\{(x,\gamma(t)x)\subset V;x\in \C^{2n}\}$, $t\in[0,\tau]$.
We define the Maslov-type indices with  periodic boundary condition  by
	\[i_1^{\Lo}(\gamma)=i_{\Graph(I_{2n})}(\gamma)-n,
	\quad
	\nu_1(\gamma)=\nu_1(\gamma(\tau))=\dim \ker (\gamma(\tau)-I_{2n});\]
 and denote the Maslov-type indices with $e^{\sqrt{-1}\theta}$-boundary condition, where $e^{\sqrt{-1}\theta}\in U(1)\setminus 1$  by
	\[i_{e^{\sqrt{-1}\theta}}(\gamma)=i_{\Graph(e^{\sqrt{-1}\theta}\I_{2n})}(\gamma),\quad 
	\nu_{e^{\sqrt{-1}\theta}}(\gamma)=\dim_{\C} \ker_{\C} (\gamma(\tau)-e^{\sqrt{-1}\theta}I_{2n}).\]
We define the Maslov-type indices about the boundary conditions $L_i\times L_j$ ($i,j\in \{0,1\}$) by
	\begin{equation}\label{e:Maslov-index-L_0}
			i_{L_i}(\gamma)= i_{\tilde\alpha_i}(\gamma)-n,\quad
			\nu_{L_i}(\gamma)=\nu_{L_i}(\gamma(\tau))=\dim_{\C}(\Graph(\gamma(\tau))\cap \widetilde \alpha_i)=\dim(\gamma(\tau)L_i\cap L_i);
	\end{equation}
	for $i\neq j$, we denote $i_{L_i\times L_j}(\gamma)=i_{\alpha_i\times \alpha_j}(\gamma)$,
	\begin{equation}\label{e:Maslov-index-L_0-L_1}
			\nu_{L_i\times L_j}(\gamma)=\nu_{L_i\times L_j}(\gamma(\tau))=\dim_{\C}(\Graph(\gamma(\tau))\cap (\alpha_i\times \alpha_j))=\dim (\gamma(\tau)L_i\cap L_j).\\
		\end{equation}
\end{definition}
\begin{remark}
	$i_1^{\Lo}(\gamma)$ is exactly the  Maslov-type index defined by 
	Professor Y. Long (cf. \cite{Lo}), and 	$\nu_1(\gamma)$ is the nullity of $\gamma(\tau)-I_{2n}$, i.e., the dimension of the null-space of which.
	For $M=\begin{pmatrix}
		A & B \\
		C& D \\
	\end{pmatrix}\in \Sp(2n)$, where $A,B,C,D\in \R^{n\times n}$, $\nu_{L_0}(M)=\dim \ker B$ and $\nu_{L_1}(M)=\dim \ker C$.
\end{remark}
Given a $\tau$-periodic solution  $(x,\tau)$ of \eqref{e:HS-jT}, where $H(t,x)$ is $\tau$-periodic about $t\in\R$ including the autonomous case, 
set $B(t)=H''(t,x(t))$ for all $t\in \R$. Then $B\in  C(S_{\tau},\mathcal L_s{(\mathbb {R}}^{2n}))$.
By $\gamma_x$ we denote the matrizant of the linear Hamiltonian system
(\ref{e:linear-path}) with $B(t)=H''(t,x(t))$ for $t\in\R$.
We call $\gamma_x$ the associated symplectic path of $(x,\tau)$.
Define a symplectic path $\gamma\in\mathcal P_{\frac{\tau}{2}}(2n)$ by
$\gamma(t)=\gamma_x(t)$ for $t\in[0,\tau/2]$.
We define the Maslov-type indices of $(x,\tau)$ with the Lagrangian boundary condition $L_i$, $i=0$ or $1$, via that of the restriction of the symplectic path  $\gamma_x$ on the interval $[0,\tau/2]$:
\begin{equation*}
	(i_{L_i}(x),\nu_{L_i}(x))=  \left(i_{L_i}(\gamma),\nu_{L_i}\left(\gamma
	(\tau/2)\right)\right).
\end{equation*}
We still denote the restriction of $\gamma_x$ on $[0,\tau]$
by $\gamma_x\in\mathcal P_{\tau}(2n)$.
We define the Maslov-type indices of $(x,\tau)$ with periodic boundary condition  and $e^{\sqrt{-1}\theta}$-boundary condition by
\begin{equation*}
	(i^{\Lo}_1(x),\nu_1(x))= \left(i^{\Lo}_1(\gamma_x),\nu_1\left(\gamma_x
	(\tau)\right)\right) \text{ and }
\end{equation*}
\[(i_{e^{\sqrt{-1}\theta}}(x),\nu_{e^{\sqrt{-1}\theta}}(x))=\left (i_{e^{\sqrt{-1}\theta}}(\gamma_x),\nu_{e^{\sqrt{-1}\theta}}(\gamma_x(\tau))\right),\]
where $e^{\sqrt{-1}\theta}\in U(1)\setminus 1$ respectively.

According to  \cite[Lemma 3.1]{Lo-Z-Zhu}, we have
\begin{lemma}\label{l:partial-positive-L0-L1}
	Let $B\in C(S_{\tau},\mathcal{L}_s(\R^{2n}))$ and let
	$\gamma_{B}$ be the \textit{matrizant} of the system \eqref{e:linear-path}. 
	Define a symplectic path $\gamma\in\mathcal P_{\frac{\tau}{2}}(2n)$ by
	$\gamma(t)=\gamma_B(t)$ for $t\in[0,\tau/2]$.
 Write $B(t)=\left(\begin{smallmatrix}
		B_{11}(t)&B_{12}(t)\\
		B_{21}(t)&B_{22}(t)
	\end{smallmatrix}\right)$, where $B_{ij}(t)\in \R^{n\times n}$ for $i,j=1,2$ and $t\in S_{\tau}$.
	Set $S=\frac{\tau}{2}$.
\begin{enumerate}
	\item [\rm (a)]
	If $B_{22}(t)$ is positive definite for all $t\in[0,S]$, then $i_{L_0}(\gamma)\ge 0$.
	\item [\rm (b)]
	If $B_{11}(t)$ is positive definite for all $t\in[0,S]$, then $i_{L_1}(\gamma)\ge 0$.
\end{enumerate}
	\end{lemma}

Since the $\tau$-periodic brake solution has the brake symmetric property, we obtain
\[x(t)=Nx(\tau-t)\text{ for } t\in [\tau/2,\tau].\]
Hence the restriction of $\tau$-periodic brake solution $(x,\tau)$ on $[0,\tau/2]$ determines the whole
solution $(x,\tau)$. Thus the associated symplectic path $\gamma_x$ of $(x,\tau)$ inherits the similar property. Specifically, for $t\in [\tau/2,\tau]$, we have
\begin{equation*}
	\gamma_x(t)=N\gamma_x(\tau-t)\gamma_x(\tau/2)^{-1}N\gamma_x(\tau/2).
\end{equation*}
For other $t\in\R$, we have similar formulae.
This motivates us to define the iteration of a symplectic path  about the brake symmetry.
\begin{definition}[cf. {\cite[Definition 2.9]{L-Z2}}]\label{d:iteration-brake}
	Given a positive number $S$, let $\gamma\in \mathcal{P}_S(2n)$. Denote by $\gamma(2S)=N\gamma(S)^{-1}N\gamma(S)$.
	We define the $k$-th iteration $\gamma^k$ of $\gamma$ for the brake symmetry by the restriction of  $\tilde\gamma(t)$ on $t\in[0,kS]$, where
	\begin{equation*}
		\tilde\gamma(t)=\left\{
		\begin{array}{ll}
			\gamma(t-2jS)(\gamma(2S))^j   &2jS\le t\le(2j+1)S,\   j\in{\mathbb N}\cup\{0\},\\
			\ & \ \\
			N \gamma((2j+2)S-t)N(\gamma(2S))^{j+1}   &(2j+1)S\le t \le (2j+2)S,\  j\in{\mathbb N}\cup\{0\}.
		\end{array}
		\right.
	\end{equation*}
\end{definition}

The relation among $i^{\Lo}_{1}$, $i_{L_0}$ and  $i_{L_1}$ is as follows.
\begin{proposition}[cf. {\cite[Proposition 4.2]{L-Z1}}]
	\begin{equation}\label{e:Maslov-index-identity-L_0}
		\begin{split}
		i^{\Lo}_1(\gamma^2)&=i_{L_0}(\gamma)+i_{L_1}(\gamma)+n.\\
		\nu_1(\gamma^2)&=\nu_{L_0}(\gamma)+\nu_{L_1}(\gamma).
		\end{split}
	\end{equation}
\end{proposition}
 The first equality is just \cite[Proposition C]{Lo-Z-Zhu}.
Now we recall the following Bott-type iteration formulae for the above brake iteration in \cite[Theorems 4.1 and 4.2]{L-Z1}.
\begin{proposition}\label{p:iteration-brake}
	For $k\in 2\N-1$ and $a\in\{0,1\}$ , we have
	\begin{equation*}
		\begin{split}
			i_{L_a}(\gamma^k)&=i_{L_a}(\gamma)+\sum_{j=1}^{\frac{k-1}{2}}i_{\beta^{2j}}(\gamma^2),\\
			\nu_{L_a}(\gamma^k)&=\nu_{L_a}(\gamma)+\sum_{j=1}^{\frac{k-1}{2}}\nu_{\beta^{2j}}(\gamma^2),
		\end{split}
	\end{equation*}
	where $\beta=e^{\frac{\sqrt{-1}\pi}{k}}$;
	while for $k\in 2\N$, we have
	\begin{equation*}
		\begin{split}
			i_{L_0}(\gamma^k)&=i_{L_0}(\gamma)+i_{L_0\times L_1}(\gamma)+\sum_{j=1}^{\frac{k}{2}-1}i_{\beta^{2j}}(\gamma^2),\\
			\nu_{L_0}(\gamma^k)&=\nu_{L_0}(\gamma)+\nu_{L_0\times L_1}(\gamma)+\sum_{j=1}^{\frac{k}{2}-1}\nu_{\beta^{2j}}(\gamma^2),
		\end{split}
	\end{equation*}
	where $\beta=e^{\frac{\sqrt{-1}\pi}{k}}$.
\end{proposition}
In \cite[Corollary 4.10]{zhou-zhu-wu}, by using the
 H\"{o}rmander index, the authors gave a new proof of the following iteration inequality (cf. \cite[Theorem 2.4.1$^{\circ}$]{liu2010}).
 \begin{proposition}
 	For $l\in\Z$,
 	\begin{equation}\label{e:iterarion-even}
 	i_{L_0}(\gamma^{2l})\ge i_{L_0}(\gamma^2)+(l-1)(i^{\Lo}_1(\gamma^2)+\nu_1(\gamma^2)-n).
 	\end{equation}
\end{proposition}

Now we introduce the H\"{o}rmander index and the triple index.
The H\"{o}rmander index originates in H\"{o}rmander's pioneering \cite{Hor71}, a main topic of which is the Cauthy problem of hyperbolic differential equations. Then Duistermaat \cite{duistermaat}
used the H\"{o}rmander index in studying the Maslov index and Morse theory.
In this paper we deal with real symplectic matrices,
while in 
\cite{zhou-zhu-wu} we used unitary matrices to define the Maslov index and studied indices in complex symplectic vector space.
Note that  $\R^{2n}\subset \C^{2n}$ and the real symplectic path $\gamma_x(t)$ is also a complex symplectic path.


\begin{definition}[{\cite[Definition 3.9]{zhou-zhu-wu}}]\label{d:hormander}
	Let $\lambda_1,\lambda_2,\mu_1,\mu_2$ be four Lagrangian subspaces of a fixed symplectic linear space $(V,\omega)$.
	Let $\lambda\colon [0,1]\to \Lag(V)$ be a path with $\lambda(0)=\lambda_1$ and $\lambda(1)=\lambda_1$.
Then the H\"{o}rmander index is defined by  
	\[s(\lambda_1,\lambda_2;\mu_1,\mu_2)=\Mas\{\lambda,\mu_2\}-\Mas\{\lambda,\mu_1\}.\]
\end{definition}
\begin{remark}
	Since $\Lag(V)$ is path-connected and the Maslov-type index is homotopy-invariant (cf. \cite[Page 182]{duistermaat} and \cite[Section 3.4]{BoZh14}),
	the H\"{o}rmander index is well-defined.
\end{remark}
To calculate the  H\"{o}rmander index we defined the triple index $i(\alpha,\beta,\delta)$ in \cite[Corollary 3.12]{zhou-zhu-wu} for any three Lagrangian subspaces $\alpha,\beta,\delta$ of $(V,\omega)$.
For an Hermitian form $Q$ (or a symmetric matrix $D$), we denote by 
$m^*(Q)$ (or  $m^*(D)$), $*=+,0,-$, the Morse positive index, the nullity and the Morse negative index of $Q$ (or $D$) respectively.
Here we just recall the calculation formula:
\begin{proposition}[cf. {\cite[Lemma 3.13]{zhou-zhu-wu}}]\label{p:i-index-calculation}
Let $(V,\omega)$ be a complex symplectic linear space of dimension $2m$,
 and $\alpha,\beta,\delta \in\Lag(V)$. Then we have
\begin{align}
 		i(\alpha,\beta,\delta)&=m^+(Q(\alpha,\beta,\delta))+\dim_{\C} (\alpha\cap \delta) -\dim_{\C} (\alpha\cap\beta\cap\delta) \label{e:i-index-calculation-a}\\
 		&\le m-\dim_{\C} (\alpha\cap \beta)-\dim_{\C}(\beta\cap \delta)+\dim_{\C}(\alpha\cap\beta\cap \delta),
 		\label{e:i-index-calculation-b}
 	\end{align}
where $Q(\alpha,\beta;\delta)$ is a Hermitian  form defined on 
$\alpha\cap(\beta+\delta)$ (cf. Definition \ref{d:hermitian-form}).
\end{proposition}
Now we recall 
\begin{proposition}[cf. {\cite[Theorem 1.1]{zhou-zhu-wu}}]
\label{p:hormander-calculate}
Let $\lambda_1,\lambda_2,\mu_1,\mu_2$ be four Lagrangian subspaces of $(V,\omega)$. Then we have
\begin{equation*}
	\begin{split}
		s(\lambda_1,\lambda_2;\mu_1,\mu_2)=i(\lambda_1,\mu_1,\mu_2)-i(\lambda_2,\mu_1,\mu_2).
	\end{split}
\end{equation*}
\end{proposition}
According to \cite[Proposition 4.7]{zhou-zhu-wu}, we have
\begin{lemma}\label{l:hormander-index-ge-nulity}
	Let $\beta$
	be a Lagrangian subspace of $(\C^{2n},\omega_0)$, and let $\lambda$ be a Lagrangian subspace of $(V,\omega)=(\C^{2n}\oplus \C^{2n},(-\omega_0)\oplus\omega_0)$. Then we have
	\begin{align}
		s(\Graph(I_{2n}),\lambda;\widetilde\alpha_0,\beta\times \alpha_0)
		&\ge \dim_{\C} (\beta \cap \alpha_0)-n, \label{e:Hormander-index-beta-L0}\\
		s(\Graph(I_{2n}),\lambda;\widetilde\alpha_0, \alpha_0\times \beta)
		&\ge \dim_{\C} (\beta \cap \alpha_0)-n. \label{e:Hormander-index-L0-beta}
	\end{align}
\end{lemma}
Note that the proof of \eqref{e:Hormander-index-L0-beta} is almost the same as \eqref{e:Hormander-index-beta-L0}.
We have the following iteration inequalities which will be used later.
First we state some properties of the Maslov-type indices.
\begin{lemma}\label{l:mas-index-pro}
	Let $\gamma\in\mathcal{P}_{\tau}(2n)$. Let $\beta_1,\beta_2$ be two Lagrangian subspaces of $(\C^{2n},\omega_0)$.
	{\rm (i)} For $P\in\Sp(2n)$, we have
	\begin{equation}\label{e:mas-index-multi}
		i_{\beta_1\times \beta_2}(\gamma P)
	=i_{P\beta_1\times \beta_2}(\gamma).
	\end{equation}
	{\rm(ii)} Set $S=\tau/2$. We define $R(\gamma)\in C([0,S],\Sp(2n))$  by 
	\begin{equation}\label{e:define-R-gamma}
	R(\gamma)(t)=\gamma(S-t) \text{ for } t\in[0,S].
	\end{equation} Then we have
	\begin{equation}\label{e:mas-index-define-R-gamma}
		i_{\beta_1\times \beta_2}(NR(\gamma)\gamma(S)^{-1}N)=
		i_{N\beta_2\times N\beta_1}(\gamma).
	\end{equation}
\end{lemma}
\begin{proof}
{\rm(i)} follows from \cite[Lemma 2.9]{zhou-zhu-wu}.
{\rm(ii)} is a special case of  \cite[Corollary 4.5]{zhou-zhu-wu}.
\end{proof}
\begin{proposition}\label{p:iteration-brake-inequality}
	For positive integer $k$ and $a\in\{0,1\}$, we have
	\begin{equation}\label{e:iteration-ineq}
		i_{L_a}(\gamma^k)\ge ki_{L_a}(\gamma)+\sum_{j=1}^{k-1}\nu_{L_a}(\gamma^j).
	\end{equation}
\end{proposition}
\begin{proof}As was hidden
	in the proof of \cite[Corollary 4.12]{zhou-zhu-wu}, we have actually proved \eqref{e:iteration-ineq} for 
	the case $L_1$ briefly.
	The method of proving the other case $L_0$ is similar, a detailed proof of which is as follows. 

	By Definition \ref{d:iteration-brake} and
using the H\"{o}rmander index (cf. \cite[Corollary 4.5 and Theorem 4.11]{zhou-zhu-wu}),
we have for $j\in\N$ and $j\ge 2$
\begin{equation}\label{e:iteration-inequality}
	i_{L_0}(\gamma^j)-i_{L_0}(\gamma^{j-1})-i_{L_0}(\gamma)\ge \nu_{L_0}(\gamma^{j-1}).
\end{equation}
Add up \eqref{e:iteration-inequality} for $j=2,3,...,k$. Then we obtain  \eqref{e:iteration-ineq}.

Now we prove \eqref{e:iteration-inequality}. By Definition \ref{d:Maslov-type-index} (cf. \eqref{e:Maslov-index-L_0}), we have
\begin{equation}\label{e:ietration-brake-j}
	\begin{split}
i_{L_0}(\gamma^j)-i_{L_0}(\gamma^{j-1})-i_{L_0}(\gamma)
&=i_{\widetilde\alpha_0}(\gamma^j)
-n-(i_{\widetilde\alpha_0}(\gamma^{j-1})-n)-
(i_{\widetilde\alpha_0}(\gamma)-n)\\
&=i_{\widetilde\alpha_0}(\gamma^j)-
i_{\widetilde\alpha_0}(\gamma^{j-1})-
i_{\widetilde\alpha_0}(\gamma)+n.
\end{split}
\end{equation}

When $j\in 2\N$, by definition $\gamma^j$ is the iteration of $\gamma^{j-1}$ and $NR(\gamma)\gamma(S)^{-1}N$, where 
$R(\gamma)\in C([0,S],\Sp(2n))$ is defined in \eqref{e:define-R-gamma}.
Then continuing from \eqref{e:ietration-brake-j}, we get
\begin{equation*}
	\begin{split}
	 i_{L_0}(\gamma^j)-i_{L_0}(\gamma^{j-1})-i_{L_0}(\gamma)
	&=i_{\widetilde \alpha_0}(NR(\gamma)\gamma^{-1}(S)N\gamma^{j-1}((j-1)S))-
	i_{\widetilde \alpha_0}(\gamma)+n\\
	&=i_{\gamma^{j-1}((j-1)S)\alpha_0\times \alpha_0} (NR(\gamma)\gamma^{-1}(S)N)-i_{\widetilde \alpha_0}	(\gamma)+n\ \text{\quad  by \eqref{e:mas-index-multi}}\\
	&=i_{\alpha_0\times N\gamma^{j-1}((j-1)S)\alpha_0}(\gamma) -i_{\widetilde \alpha_0}(\gamma)+n\ \text{\quad  by \eqref{e:mas-index-define-R-gamma}}\\
	&=s(\Graph(I_{2n}),\Graph(\gamma(S)),\widetilde \alpha_0, \alpha_0\times N\gamma^{j-1}((j-1)S)\alpha_0)+n\\
	&\ge \nu_{L_0}(\gamma^{j-1}),
	\end{split}	
\end{equation*}
where in the last line we used \eqref{e:Hormander-index-L0-beta}.

When $j\in 2\N+1$, by definition $\gamma^j$ is the iteration of $\gamma^{j-1}$ and $\gamma$.
Then also  continuing from 	\eqref{e:ietration-brake-j}, we obtain
\begin{equation}\label{e:odd-iterations-L0}
	\begin{split}
	i_{L_0}(\gamma^j)-i_{L_0}(\gamma^{j-1})-i_{L_0}(\gamma)
		&=i_{\widetilde \alpha_0}(\gamma\gamma^{j-1}((j-1)S))
		-i_{\widetilde \alpha_0}(\gamma)+n\\
		&=i_{\gamma^{j-1}((j-1)S)\alpha_0\times \alpha_0}(\gamma)-	i_{\widetilde \alpha_0}(\gamma)+n \text{ \quad by \eqref{e:mas-index-multi}}\\
		&=s(\Graph(I_{2n}),\Graph(\gamma(S)),\widetilde \alpha_0,\gamma^{j-1}((j-1)S)\alpha_0\times \alpha_0)+n\\
		&\ge \nu_{L_0}(\gamma^{j-1}),
	\end{split}	
	\end{equation}
	where in the last line we used \eqref{e:Hormander-index-beta-L0}.
	The proof of this lemma is complete.
\end{proof}

As a  direct corollary of the above lemma, we have 
the following iteration inequality, which is different from  {\cite[Lemma 2.4]{FZZZ22}}.
\begin{corollary}[Compare  loc. cit.]\label{0-1}
	For each  $k\in\N$, we have
	\begin{equation}\label{e:i-L0-k-i+nu}
		i_{L_0}(\gamma^k)\ge k\left( i_{L_0}(\gamma)+\nu_{L_0}(\gamma)\right)-n.
	\end{equation}
\end{corollary}
\begin{proof}
	Since
	$n\ge\nu_{L_0}(\gamma^j)\ge \nu_{L_0}(\gamma)$ for $j\in\N$, 
 \eqref{e:i-L0-k-i+nu} follows from \eqref{e:iteration-ineq}.
\end{proof}
From Proposition \ref{p:iteration-brake-inequality}, we readily obtain 
\begin{lemma}\label{l:iteration-ineq-L0}
	Let $k\in \N$. Assume $i_{L_0}(\gamma^k)\le 1$ and $i_{L_0}(\gamma)\ge 0$.
	Then 
	\begin{enumerate}
		\item[1)] $i_{L_0}(\gamma)=\nu_{L_0}(\gamma)=0$; or
		\item[2)] $k=2$, $i_{L_0}(\gamma^2)=1$ and $i_{L_0}(\gamma)=0$; or
		\item[2)] $k=1$.
	\end{enumerate}
	\end{lemma}
\begin{proof}
	Since $\nu_{L_0}(\gamma^j)\ge \nu_{L_0}(\gamma)$ for $j\in\N$,
	by \eqref{e:iteration-ineq} we have
	\begin{equation}\label{e:iteration-ineq-L0}
		i_{L_0}(\gamma^k)\ge ki_{L_0}(\gamma)+(k-1)\nu_{L_0}(\gamma).
	\end{equation}
Our lemma follows from \eqref{e:iteration-ineq-L0}.
\end{proof}

\begin{proposition}\label{p:control-iteration-times}
	Let $k\in\N$.
	Suppose $i_{L_0}(\gamma^k)\le 1$, $i_{L_0}(\gamma)\ge 0$ and 
	\begin{equation}\label{e:i+nu-ge-n+1}
	i^{\Lo}_1(\gamma^2)+\nu_1(\gamma^2)\ge n+1.
	\end{equation}
	When $k\in 2\N$, we furthermore assume $i_{L_0}(\gamma^2)\ge 0$.
	Then $k\le 4$ and $k\neq 3$.
\end{proposition}
\begin{proof}
	
	Since $i_{L_0}(\gamma)\ge 0$, according to $i_{L_0}(\gamma)$  we divide our proof into two cases.
	
	{\bf Case 1:} $i_{L_0}(\gamma)\ge 1$. By Proposition \ref{p:iteration-brake-inequality}, we get $k=1$.
	
	{\bf Case 2:}  $i_{L_0}(\gamma)= 0$.
	When $\nu_{L_0}(\gamma)\ge 1$,
	Proposition \ref{p:iteration-brake-inequality} tells us
	$k\le 2$.
	Since $\nu_{L_0}(\gamma)\ge 0$, the remained cases are
	\[i_{L_0}(\gamma)=0 \quad \text{ and }\quad \nu_{L_0}(\gamma)=0.\]
	By \eqref{e:Maslov-index-identity-L_0} and \eqref{e:i+nu-ge-n+1},
	we have 
	\begin{equation}\label{e:i+nu-ge1}
	i_{L_1}(\gamma)+\nu_{L_1}(\gamma)\ge 1.
	\end{equation}
	Now according to the partition of the positive integer $k$, we continue our study in two cases.
	
	When $k$ is odd, according to Bott's iteration formulae Proposition \ref{p:iteration-brake}, we have
	\begin{equation}\label{e:index-Bott-L0-L1}
		i_{L_1}(\gamma^k)- i_{L_1}(\gamma)=i_{L_0}(\gamma^k)-i_{L_0}(\gamma)=i_{L_0}(\gamma^k).
		\end{equation}
	Thus we obtain
	\begin{equation}\label{e:index-L0-L1-1}
		1\ge i_{L_0}(\gamma^k) = i_{L_1}(\gamma^k)-i_{L_1}(\gamma).
	\end{equation}
	By \eqref{e:index-L0-L1-1} and Proposition \ref{p:iteration-brake-inequality}, we get
	\[1+ i_{L_1}(\gamma)\ge ki_{L_1}(\gamma)+(k-1)\nu_{L_1}(\gamma).\]
	This means 
	\begin{equation}\label{e:iter-1-ge}
	1\ge (k-1)(i_{L_1}(\gamma)+\nu_{L_1}(\gamma)).
	\end{equation}
Together with \eqref{e:i+nu-ge1}, we get $k\le2$.
Since we assume $k$ is odd, we get $k=1$.
	
	When $k$ is even, by 
	\eqref{e:iterarion-even}, which means
	\begin{equation}\label{e:iter-k-even}
	i_{L_0}(\gamma^k)\ge i_{L_0}(\gamma^2)+(\frac{k}{2}-1)(i^{\Lo}_1(\gamma^2)+\nu_1(\gamma^2)-n),
	\end{equation}
	we get $k\le4$.
\end{proof}

\begin{lemma}\label{l:even-times-L0-L1}
Let $k\in\N$. Assume $i_{L_0}(\gamma^{2k})=1$ and $i_{L_0}(\gamma^k)\ge 0$.
	Then $i_{L_0}(\gamma^k)= 0$ and $i_{L_0\times L_1}(\gamma^k)=1$.
\end{lemma}
\begin{proof}
	By definition, $\gamma^{2k}$ is the second iteration of 
	$\gamma^k$ for the brake symmetry.
	Lemma \ref{p:iteration-brake-inequality} tells us
	\[i_{L_0}(\gamma^{2k})\ge 2 i_{L_0}(\gamma^k)+\nu_{L_0}(\gamma^k).\]
	Thus $i_{L_0}(\gamma^k)$ must be zero.
	By Proposition \ref{p:iteration-brake}, 
\begin{equation}\label{e:index-L0-L1-even}
i_{L_0}(\gamma^{2k})=i_{L_0}(\gamma^k)+i_{L_0\times L_1}(\gamma^k).
\end{equation}
Thus $i_{L_0\times L_1}(\gamma^k)=1$.
\end{proof}

\begin{lemma}\label{l:i-nu-L0-L1}
Let $\gamma\in \mathcal{P}_S(2n)$, where $S>0$. Rewrite $\gamma(S)=\begin{pmatrix}
	A&B\\
	C&D
\end{pmatrix}\in\Sp(2n,\R)$, where $A,B,C,D\in\R^{n\times n}$. Then we have
\begin{equation}\label{e:i-nu-L0-L1}
	i_{L_1}(\gamma)+\nu_{L_1}(\gamma)=i_{L_0}(\gamma)+\nu_{L_0}(\gamma)+m^+\begin{pmatrix}
		C^TA&C^TB\\
		B^TC&D^TB
	\end{pmatrix}-n+\dim\ker C.
\end{equation}
\end{lemma}
\begin{proof}
By Definitions  \ref{d:Maslov-type-index} and \ref{d:hormander}, Proposition \ref{p:hormander-calculate} and \eqref{e:i-index-calculation-a}, we have
\begin{equation}\label{e:L0-L1-mas}
	\begin{split}
		i_{L_0}(\gamma)-i_{L_1}(\gamma)
		&=
		\Mas\{\Graph(\gamma),\widetilde\alpha_0\}-\Mas\{\Graph(\gamma),\widetilde\alpha_1\}\\
		&=s(\Graph(I_{2n}),\Graph(M);\widetilde\alpha_1,\widetilde\alpha_0)\\
		&=i(\Graph(I_{2n}),\widetilde\alpha_1,\widetilde\alpha_0)-i(\Graph(M),\widetilde\alpha_1,\widetilde\alpha_0)\\
		&=n-(m^+(Q(\Graph(M),\widetilde\alpha_1;\widetilde\alpha_0))+\nu_{L_0}(M)).
	\end{split}
\end{equation}
Note that by Corollary \ref{c:i-index-cal-I}.(II), 
$i(\Graph(I_{2n}),\widetilde\alpha_1,\widetilde\alpha_0)=n$.
By calculation (Lemma \ref{l:quadratic-forms-calculation}.(ii)), $m^+(Q(\Graph(M), \widetilde\alpha_1;\widetilde\alpha_0 ))=m^+\begin{pmatrix}
	C^TA&C^TB\\
	B^TC&D^TB
\end{pmatrix}$.
Then together with \eqref{e:L0-L1-mas}, we have 
\begin{equation}\label{e:i-L0-nu-L0-i-L1}
	i_{L_0}(\gamma)+\nu_{L_0}(\gamma)=i_{L_1}(\gamma)-m^+
	\begin{pmatrix}
		C^TA & C^TB \\
		B^TC & D^TB\\
	\end{pmatrix}+n.
\end{equation}
Together with the definition that $\nu_{L_1}(\gamma)=\dim \ker C$, we obtain 
\eqref{e:i-nu-L0-L1}.
\end{proof}

\begin{lemma}\label{l:i=0}
	Let $\gamma\in \mathcal{P}_S(2n)$, where $S>0$. Then we have
	\begin{equation}\label{e:i-L0-L1}
		i_{L_0\times L_1}(\gamma)\ge i_{L_1}(\gamma)+\nu_{L_1}(\gamma).
	\end{equation}
\end{lemma}
\begin{proof}
	Set $M=\gamma(S)$. By Definitions  \ref{d:Maslov-type-index} and \ref{d:hormander}, Proposition \ref{p:hormander-calculate}, we have
	\begin{equation}\label{e:index0-hor-mas}
		\begin{split}
			i_{L_0}(\gamma)-i_{L_0\times L_1}(\gamma)
			&=\Mas\{\Graph(\gamma),\widetilde \alpha_0\}-n-\Mas\{\Graph(\gamma),\alpha_0\times \alpha_1\}\\
			&=s(\Graph(I_{2n}),\Graph(M);\alpha_0\times \alpha_1,\widetilde\alpha_0)-n\\
			&=i(\Graph(I_{2n}),\alpha_0\times \alpha_1,\widetilde\alpha_0)-i(\Graph(M),\alpha_0\times \alpha_1,\widetilde\alpha_0)-n \\
			&=n-(m^+(Q(\Graph(M),\alpha_0\times \alpha_1;\tilde\alpha_0))+\nu_{L_0}(M))-n\ \text{ by
				\eqref{e:i-index-calculation-a}}.
		\end{split}
	\end{equation}
	We used  Corollary \ref{c:i-index-cal-I}.(I), 
	$i(\Graph(I_{2n}),\alpha_0\times \alpha_1,\widetilde\alpha_0)=n$.
	Rewrite $M=
	\begin{pmatrix}
		A & B \\
		C & D\\
	\end{pmatrix}
	\in \Sp(2n,\R)$, where $A,B,C,D\in\R^{n\times n}$.
	Since $M^TJM=J$, we have
	\[C^TA=A^TC,\qquad D^TB=B^TD \qquad D^TA-B^TC=I_n.\]
	By the definition of $Q(\cdot,\cdot;\cdot)$ (cf. Lemma \ref{l:quadratic-forms-calculation},(i)),
	\begin{equation}\label{e:positive-Morse-index}
		m^+(Q(\Graph(M),\alpha_0\times \alpha_1;\widetilde\alpha_0))=m^+(D^TB).
	\end{equation}
	Together with \eqref{e:index0-hor-mas}, we have 
	\begin{equation}\label{e:i-L0-nu-L0}
		i_{L_0}(\gamma)+\nu_{L_0}(\gamma)=i_{L_0\times L_1}(\gamma)-m^+(D^TB).
	\end{equation}
	Combining \eqref{e:i-L0-nu-L0} and \eqref{e:i-L0-nu-L0-i-L1}, we obtain 
	\begin{equation}\label{e:i-L1+nu-L1}
		i_{L_0\times L_1}(\gamma)=i_{L_1}(\gamma)+\nu_{L_1}(\gamma)+m^+(D^TB)-m^+
		\begin{pmatrix}
			C^TA & C^TB \\
			B^TC & D^TB\\
		\end{pmatrix}-\dim \ker C+n.
	\end{equation}
	
	Denote $r=m^-(D^TB)$ and $s=m^0(D^TB)$.
	Thus there exists a $P\in\GL(n,\R)$ such that $P^TB^TDP=\begin{pmatrix}
		-I_r & 0&0\\
		0 & 0_{s}&0\\
		0&0&I_{n-s-r}
	\end{pmatrix}$, where $0\le r+s\le n$.
	Then we have
	\[
	\begin{pmatrix}
		I_n & 0\\
		0 & P^T\\
	\end{pmatrix}
	\begin{pmatrix}
		C^TA & C^TB \\
		B^TC & D^TB\\
	\end{pmatrix}
	\begin{pmatrix}
		I_n & 0\\
		0 & P\\
	\end{pmatrix}=
	\begin{pmatrix}
		C^TA & C^TBP \\
		P^TB^TC & P^TD^TBP\\
	\end{pmatrix}.\]
	Let $\lambda \le 0$.
	Set an $n\times n$ matrix $E(\lambda)=
	\begin{pmatrix}
		-I_r & 0&0\\
		0 & \lambda I_{s}&0 \\
		0&0&I_{n-r-s}
	\end{pmatrix}$.
Consider  the $2n\times 2n$ matrix \[Q(\lambda)=
	\begin{pmatrix}
		C^TA & C^TBP \\
		P^TB^TC & E(\lambda)\\
	\end{pmatrix}.\] Then 
\begin{equation}\label{e:m+-Q0}
	m^+(Q(0))=m^+\begin{pmatrix}
		C^TA&C^TB\\
		B^TC&D^TB
	\end{pmatrix}.
\end{equation}
For $\lambda<0$ and letting $P_1= -C^TBPE(\lambda)^{-1}$, we have
	\begin{equation*}
		\begin{split}
	\begin{pmatrix}
		I_n &P_1\\
		0 & I_n\\
	\end{pmatrix}Q(\lambda)\begin{pmatrix}
	I_n & 0\\
	P_1^T &I_n\\
\end{pmatrix}&=	\begin{pmatrix}
		I_n &P_1\\
		0 & I_n\\
	\end{pmatrix}\begin{pmatrix}
		C^TA & C^TBP \\
		P^TB^TC & E(\lambda)\\
	\end{pmatrix}\begin{pmatrix}
		I_n & 0\\
		P_1^T &I_n\\
	\end{pmatrix}\\
&=\begin{pmatrix}C^TA-C^TBPE(\lambda)^{-1}P^TB^TC&0\\
		0&E(\lambda )\end{pmatrix}.
	\end{split}
\end{equation*}
By continuity, there exists a $\delta >0$ such that for $-\delta <\lambda <0$, we have
	\begin{equation}\label{e:m+-quadratic}
		m^+(Q(0))\le m^+(Q(\lambda))=m^+\left(C^TA-C^TBP\left(
		\begin{array}{ccc}
			-I_r & 0 &0\\
			0 & \frac{1}{\lambda} I_s&0\\
			0&0&I_{n-s-r}
		\end{array}\right)P^TB^TC\right)+n-r-s.
	\end{equation}
	Note that 
	\begin{equation}\label{e:m+-kerC}
		m^+(C^TA-C^TBPE(\lambda)^{-1}P^TB^TC
	) +\dim\ker C \le n.
	\end{equation}
	By \eqref{e:m+-Q0}--\eqref{e:m+-kerC} and the fact that $m^+(D^TB)+r+s=n$, we obtain
	\begin{equation}\label{e:m+-m+-kerC+n}
	m^+(D^TB)-m^+
	\begin{pmatrix}
		C^TA & C^TB \\
		B^TC & D^TB\\
	\end{pmatrix}-\dim \ker C+n\ge 0.
	\end{equation}
	
Combining \eqref{e:i-L1+nu-L1} and \eqref{e:m+-m+-kerC+n}, we get
	\eqref{e:i-L0-L1}.
\end{proof}

\begin{proposition}\label{p:index=0}
	Assume $i_{L_0}(\gamma^k)=0$, $i_{L_0}(\gamma)\ge 0$ and $i_{L_1}(\gamma)+\nu_{L_1}(\gamma)\ge1$.
	When $k\in 2\N$, we furthermore assume $i_{L_0}(\gamma^2)\ge0$. Then we get $k=1$.
\end{proposition}
\begin{proof}
	According to the proof of Proposition \ref{p:control-iteration-times}, we get $k\le 2$.
	In fact, if $k$ is odd, by Lemma \ref{l:iteration-ineq-L0}, the remaining case we need to deal with is $i_{L_0}(\gamma)=0$. 
	 By \eqref{e:index-Bott-L0-L1} and \eqref{e:iteration-ineq},
	we have
	\[i_{L_1}(\gamma)=i_{L_1}(\gamma^k)\ge ki_{L_1}(\gamma)+(k-1)\nu_{L_1}(\gamma).\]
	Thus we obtain \[0\ge (k-1)(i_{L_1}(\gamma)+\nu_{L_1}(\gamma)),\] from which we get $k=1$.
	
	Note that by \eqref{e:Maslov-index-identity-L_0}, we have
	\[i_1^{\Lo}(\gamma^2)+\nu_1(\gamma^2)=i_{L_1}(\gamma)+\nu_1(\gamma)+i_{L_0}(\gamma)+\nu_{L_0}(\gamma)+n\ge n+1.\]
	If $k$ is even,  from \eqref{e:iter-k-even}, we get $k\le 2$. .
	
	If $k=2$, then by Proposition \ref{p:iteration-brake-inequality}, we get $i_{L_0}(\gamma)=0$ and $\nu_{L_0}(\gamma)=0$; thus $i_{L_0\times L_1}(\gamma)=0$ by \eqref{e:index-L0-L1-even}.
	According to Lemma \ref{l:i=0}, we get $i_{L_1}(\gamma)+\nu_{L_1}(\gamma)\le 0$. This contradiction shows $k=1$.
\end{proof}
\begin{corollary}\label{c:index=1}
	Assume $i_{L_0}(\gamma^k)=1$, $i_{L_0}(\gamma)\ge 0$ and
	\begin{equation}\label{e:i-L1+nu-L1-ge1}
		i_{L_1}(\gamma)+\nu_{L_1}(\gamma)\ge 1.
	\end{equation}
	When $k\in 2\N$, we furthermore assume $i_{L_0}(\gamma^2)\ge0$. Then we get $k\le 2$.
\end{corollary}
\begin{proof}
	According to Proposition \ref{p:control-iteration-times},
	we get $k=4$ or $k\le 2$.
	Arguing by contradiction, we assume
	$k=4$.
	Then $i_{L_0}(\gamma^4)=1$.
	By Proposition \ref{p:iteration-brake-inequality}, we have
	\[i_{L_0}(\gamma^4)\ge 2i_{L_0}(\gamma^2)+\nu_{L_0}(\gamma^2).\]
	Thus $i_{L_0}(\gamma^2)=0$.
	By \eqref{e:index-L0-L1-even}, we obtain $i_{L_0\times L_1}(\gamma)=0$.
	Then by Lemma \ref{l:i=0},
	we get
	\[i_{L_1}(\gamma)+\nu_{L_1}(\gamma)\le 0,\]
	which contradicts \eqref{e:i-L1+nu-L1-ge1}.
	Thus $k\neq 4$. Consequently, $k\le 2$.
\end{proof}

According to Proposition \ref{p:index=0} and Corollary 
\ref{c:index=1}, we notice that the positive lower bound \eqref{e:i-L1+nu-L1-ge1} is very important.
Now we give some sufficient conditions to ensure that
\eqref{e:i-L1+nu-L1-ge1} holds.
\begin{proposition}\label{p:positive-L1}
	Let $B=\begin{pmatrix}
		B_{11}&B_{12}\\
		B_{21}&B_{22}
	\end{pmatrix}\in C(S_{\tau},\mathcal{L}_s(\R^{2n}))$
	and $\gamma\in \mathcal P_{\frac{\tau}{2}}(2n)$ be given in Lemma \ref{l:partial-positive-L0-L1}.
If $B_{11}(t)$ is positive definite for all $t\in[0,\tau/2]$ and $\nu_{L_1}(\gamma)\ge 1$, then
\[i_{L_1}(\gamma)+\nu_{L_1}(\gamma)\ge1.\]
\end{proposition}
\begin{proof}
	This is a direct corollary of Lemma \ref{l:partial-positive-L0-L1}(a).
\end{proof}

\begin{proposition}\label{p:C=0}
	Let $\gamma\in\mathcal{P}_S(2n)$ with $S>0$.
	Assume $i_{L_0}(\gamma^k)\le 1\le i_{L_0}(\gamma^k)+\nu_{L_0}(\gamma^k)$,   $i_{L_0}(\gamma)=\nu_{L_0}(\gamma)=0$ and $i_{L_0}(\gamma^j)\le i_{L_0}(\gamma^k)$ for $j<k$, $j,k\in\N$.
	Set $\gamma(S)=\begin{pmatrix}
		A&B\\
		C&D
	\end{pmatrix}\in\Sp(2n,\R)$, where $A,B,C,D\in\R^{n\times n}$ and assume $C=0$.
	Then we have
	\[i_{L_1}(\gamma)+\nu_{L_1}(\gamma)= 1.\]
\end{proposition}
\begin{proof}
	By definition $C=0$ means $\nu_{L_1}(\gamma)=n$.
	Now we divide the proof into four cases.
	
	{\bf Case 1}: $k=2l-1$ is odd and $i_{L_0}(\gamma^k)\le0$.
	Then $l>1$ and $\nu_{L_0}(\gamma^{2l-1})\ge1$. By Proposition \ref{p:iteration-brake}, we have
	\[\nu_{L_0}(\gamma^{2l-1})=\nu_{L_0}(\gamma^{2l-1})-\nu_{L_0}(\gamma)=\nu_{L_1}(\gamma^{2l-1})-\nu_{L_1}(\gamma).\]
	Thus $\nu_{L_1}(\gamma)=\nu_{L_1}(\gamma^{2l-1})-\nu_{L_0}(\gamma^{2l-1})\le n-1$, which contradicts $C=0$.
	Consequently this case cannot happen.
	
	{\bf Case 2}: $k=2l+1$ is odd and $i_{L_0}(\gamma^k)=1$.
	Then $i_{L_0}(\gamma^3)\le 1$.
	According to Case 1, $i_{L_0}(\gamma^3)=1$.
	
	Note that $\gamma^2(2S)=NM^{-1}NM=\begin{pmatrix}
		D^TA+B^TC&2B^TD\\
		2A^TC& C^TB+A^TD
	\end{pmatrix}$, where $M=\gamma(S)$.
	Then 
	\[\gamma^2(2S)=\begin{pmatrix}
		I_n&2B^TD\\
			0& I_n
			\end{pmatrix}, \text{ and }
	M\gamma^{2}(2S)=\begin{pmatrix}
				 A&3B\\
				0& D
			\end{pmatrix} .\]
Since $\nu_{L_0}(\gamma)=\dim \ker B=0$, $C=0$ and $\gamma(S)$ is invertible, we have
\begin{equation}\label{e:nu-L0-2l-2l+1}
\nu_{L_0}(\gamma^{2})=\dim \ker B^TD=0 \text{ and } \nu_{L_0}(\gamma^{3})=\dim \ker B=0.\end{equation}
According to \eqref{e:odd-iterations-L0}, Proposition \ref{p:hormander-calculate}, \eqref{e:i-index-calculation-a} and \eqref{e:nu-L0-2l-2l+1}, 
	we have 
	\begin{equation}\label{e:even-iterations}
		\begin{split}
			i_{L_0}(\gamma^3)&-i_{L_0}(\gamma^2)-i_{L_0}(\gamma)
			=s(\Graph(I_{2n}),\Graph(\gamma(S)),\widetilde \alpha_0,\gamma^2(2S)\alpha_0\times \alpha_0)+n\\
			&=i(\Graph(I_{2n}),\widetilde \alpha_0,\gamma^2(2S)\alpha_0\times \alpha_0)-i(\Graph(M),\widetilde \alpha_0,\gamma^2(2S)\alpha_0\times \alpha_0)+n\\
			&=m^+(Q(\Graph(I_{2n}),\widetilde \alpha_0,\gamma^2(2S)\alpha_0\times \alpha_0))-
			m^+(Q(\Graph(M),\widetilde \alpha_0,\gamma^2(2S)\alpha_0\times \alpha_0))+n.
		\end{split}	
	\end{equation}
By Lemma \ref{l:quadratic-forms-calculation}.(iv) and (v), we have $m^+(Q(\Graph(I_{2n}),\widetilde \alpha_0,\gamma^2(2S)\alpha_0\times \alpha_0))=0$ and \[m^+(Q(\Graph(M),\widetilde \alpha_0,\gamma^2(2S)\alpha_0\times \alpha_0))=m^+(-(D^TB)^{-1}).\]

	By \eqref{e:index-L0-L1-even} and \eqref{e:i-L0-nu-L0},
	$i_{L_0}(\gamma^2)=i_{L_0\times L_1}(\gamma)=m^+(D^TB)$.
	If $i_{L_0}(\gamma^2)=0$, then together 
	with \eqref{e:even-iterations} we have
	$m^+(-D^TB)=n-1$ and $m^+(D^TB)=0$. Thus $D^TB$ is not invertible, which yields a contradiction. 
	If $i_{L_0}(\gamma^2)=1$, then from \eqref{e:index-L0-L1-even} we get $i_{L_0\times L_1}(\gamma)=1$ .
	By \eqref{e:i-L1+nu-L1}, we have
	\[i_{L_1}(\gamma)+\nu_{L_1}(\gamma)= 1.\]
	
{\bf Case 3:} $k=2l$ is even and $i_{L_0}(\gamma^k)=0$. 
	
	Set $\gamma^l(lS)=M_1=\begin{pmatrix}
		A_1&B_1\\
		C_1&D_1
	\end{pmatrix}$.
	Since $\nu_{L_1}(\gamma^l)\ge \nu_{L_1}(\gamma)$,
	$\dim \ker C_1 \ge \dim \ker C$.
	Thus $C_1=0$.
	Note that \[\gamma^{2l}(2lS)=NM_1^{-1}NM_1=\begin{pmatrix}
		D_1^TA_1+B_1^TC_1&2B_1^TD_1\\
		2A_1^TC_1& C_1^TB_1+A_1^TD_1
	\end{pmatrix}=\begin{pmatrix}
	I_n&2B_1^TD_1\\
	0& I_n
	\end{pmatrix}.\]
	 Thus $\nu_{L_0}(\gamma^{2l})=\dim \ker B_1^TD_1\ge1$.
	Since $C_1=0$, $B^T_1D_1$ is invertible. This contradiction excludes this case.
	
{\bf Case 4:} $k=2l$ is even and $i_{L_0}(\gamma^k)=1$.
Then according to Case 3, we have
$i_{L_0}(\gamma^2)=1$; thus by \eqref{e:index-L0-L1-even}, we obtain $i_{L_0\times L_1}(\gamma)=1$.
By \eqref{e:i-L1+nu-L1}, we have
\[i_{L_1}(\gamma)+\nu_{L_1}(\gamma)= 1.\]
	\end{proof}
\begin{remark}
	For $n=1$, that is, $M=\begin{pmatrix}
		a&b\\
		c&d
	\end{pmatrix}\in \Sp(2)$,
if $\nu_{L_1}(M)\ge 1$, then $c=0$.
\end{remark}

We will give the following lemma and corollary to illustrate the time-translate may change the Maslov-type indices for the brake symmetric
boundary condition.
\begin{lemma}\label{l:shift-index}
	Let $\gamma\in \mathcal{P}_S(2n)$.
	Then we have
	\begin{equation}\label{e:index-L0-L1}
		|i_{L_1\times L_0}(\gamma)-i_{L_0\times L_1}(\gamma)|\le n.
	\end{equation}
	Set $\gamma(S)=\begin{pmatrix}
		A & B \\
		C& D \\
	\end{pmatrix}\in \Sp(2n)$, where $A,B,C,D\in \R^{n\times n}$. Then
	\begin{gather*}
		i_{L_1\times L_0}(\gamma)-i_{L_0\times L_1}(\gamma)=n-m^+\begin{pmatrix}
			C^TA & A^TD \\
			D^TA& D^TB \\
		\end{pmatrix}-\dim \ker A.\\
		i_{L_1\times L_0}(\gamma)+\nu_{L_1\times L_0}(\gamma)-i_{L_0\times L_1}(\gamma)-\nu_{L_0\times L_1}(\gamma)=n-m^+\begin{pmatrix}
			C^TA& A^TD \\
			D^TA& D^TB \\
		\end{pmatrix}-\dim \ker D.
	\end{gather*}
\end{lemma}
\begin{proof}
  By Definitions  \ref{d:Maslov-type-index} and \ref{d:hormander}, Proposition \ref{p:hormander-calculate},
	we obtain
	\begin{equation*}
		\begin{split}
			i_{L_1\times L_0}(\gamma)&-i_{L_0\times L_1}(\gamma)= \Mas\{\Graph(\gamma),\alpha_1\times \alpha_0\}-\Mas\{\Graph(\gamma),\alpha_0\times \alpha_1\} \\
			&= s(\Graph(I_{2n}),\Graph(M);\alpha_0\times \alpha_1,\alpha_1\times \alpha_0)\\
			&= i(\Graph(I_{2n}),\alpha_0\times \alpha_1,\alpha_1\times \alpha_0)-i(\Graph(M),\alpha_0\times \alpha_1,\alpha_1\times \alpha_0)\\
			&=n-i(\Graph(M),\alpha_0\times \alpha_1,\alpha_1\times \alpha_0)\ \text{ by Corollary \ref{c:i-index-cal-I}.(III) }\\
			&=n-m^+(Q(\Graph(M),\alpha_0\times \alpha_1;\alpha_1\times \alpha_0))-\dim ML_1\cap L_0\
			\text{ by \eqref{e:i-index-calculation-a}},
		\end{split}
	\end{equation*}
	where $M=\gamma(S)$. By \eqref{e:i-index-calculation-b} of Proposition \ref{p:i-index-calculation}, we have \[0 \le i(\Graph(M),\alpha_0\times \alpha_1,\alpha_1\times \alpha_0)\le 2n.\]  Thus we get (\ref{e:index-L0-L1}).
	Note that in the notations $\Mas\{\cdot,\cdot\}$, $s(\cdot,\cdot;\cdot,\cdot)$ and triple index $i(\cdot,\cdot,\cdot)$ here and in
	\cite{zhou-zhu-wu}, we see $\Graph(\gamma)$ as its complexification in
	$\C^{2n}\times \C^{2n}$.
	Similarly, we have
	\begin{equation*}
		\begin{split}
			&\ \quad i_{L_1\times L_0}(\gamma)+\nu_{L_1\times L_0}(\gamma)-i_{L_0\times L_1}(\gamma)-\nu_{L_0\times L_1}(\gamma)\\
			&= s(\Graph(I_{2n}),\Graph(M);\alpha_0\times \alpha_1,\alpha_1\times \alpha_0) + \dim (ML_1\cap L_0)-\dim(ML_0\cap L_1)\\
			&= i(\Graph(I_{2n}),\alpha_0\times \alpha_1,\alpha_1\times \alpha_0)-i(\Graph(M),\alpha_0\times \alpha_1,\alpha_1\times \alpha_0)+ \dim (ML_1\cap L_0)-\dim (ML_0\cap L_1)\\
			&=n-m^+(Q(\Graph(M),\alpha_0\times \alpha_1;\alpha_1\times \alpha_0))-\dim(ML_0\cap L_1).
		\end{split}
	\end{equation*}
	Then by the direct computation (cf. Lemma \ref{l:quadratic-forms-calculation}(iii)), we get this lemma.
\end{proof}

\begin{corollary}\label{c:shift-two-times-index}
	Let $\gamma\in \mathcal{P}_S(2n)$. Set $\tilde \gamma(t)=N\gamma(S-t)\gamma(S)^{-1}N \text{ for } 0\le t\le S$.
	Then we have
	\[|i_{L_0}(\tilde\gamma^2)-i_{L_0}( \gamma^2)|\le n.\]
	Set $\gamma(S)=\begin{pmatrix}
		A & B \\
		C& D \\
	\end{pmatrix}\in {\rm Sp}(2n)$, where $A,B,C,D\in \R^{n\times n}$. Then
	\begin{gather*}
		i_{L_0}(\tilde\gamma^2)-i_{L_0}( \gamma^2)=n-m^+\begin{pmatrix}
			C^TA & A^TD \\
			D^TA& D^TB \\
		\end{pmatrix}-\dim \ker A.\\
		i_{L_0}(\tilde \gamma^2)+\nu_{L_0}( \tilde\gamma^2)-i_{L_0}(\gamma^2)-\nu_{L_0}(\gamma^2)=n-m^+\begin{pmatrix}
			C^TA & A^TD \\
			D^TA& D^TB\\
		\end{pmatrix}-\dim \ker D.
	\end{gather*}
\end{corollary}
\begin{proof}
	By \eqref{e:mas-index-define-R-gamma},
	we have $i_{L_0}(\tilde\gamma)=i_{L_0}( \gamma)$ and $i_{L_0\times L_1}(\tilde\gamma)=i_{L_1\times L_0}(\gamma)$.
	For $M\in \Sp(2n)$, we also have
	$\nu_{L_0}(NM^{-1}N)=\nu_{L_0}(M)$ and $\nu_{L_0\times L_1}(NM^{-1}N)=\nu_{L_1\times L_0}(M)$. Then by
	Proposition \ref{p:iteration-brake} for $k=2$
and	 Lemma \ref{l:shift-index}, we get our conclusion.
\end{proof}
\begin{remark}\label{r:shift-index}
Let $\gamma\in \mathcal{P}_S(2)$
	and assume $\gamma(S)=\left(
	\begin{array}{cc}
		1 & -2 \\
		1 & -1 \\
	\end{array}
	\right)\in {\rm Sp}(2)
	$. Then by Corollary \ref{c:shift-two-times-index}, $i_{L_0}({\tilde\gamma^2})-i_{L_0}(\gamma^2)=-1\neq 0$ and
	$i_{L_0}(\tilde\gamma^2)+\nu_{L_0}(\tilde\gamma^2)-i_{L_0}(\gamma^2)-\nu_{L_0}(\gamma^2)=-1\neq 0$.
\end{remark}
\subsection{Proof of Theorem \ref{t:super-quadratic-1}}
\label{ss:subharmonic}
For $j\in\N$, 
a $jT$-periodic solution $(x_j,jT)$ of \eqref{e:HS-jT} is also called a brake solution.
Such a periodic brake solution is called nondegenerate if $\nu_{L_0}(x_j)=0$.
For an integer $j\in\Z$ and a $kT$-periodic function $x_k$, the
phase shift $j*x_k$ is defined to be
\[(j*x_k)(t)=x_k(t+jT).\]
Since $H$ is $T$-periodic in $t$, whenever $j\in2\N$ and $x_j$ solves {\rm (HS)$_j$}, so do its phase shift
$\frac{j}{2}*x_j$.
For two brake solutions $(x_j,jT)$ of {\rm (HS)$_j$} and $(x_k,kT)$ of {\rm (HS)$_k$} in \eqref{e:HS-jT},
it might happen that $x_j$ is in fact a time-translate of $x_k$ for $j,k\in\N$.
For instance,
\[\frac{l}{2}*x_j= \frac{h}{2}*x_k,\]
for some $l\in\{0,j\}\cap 2\N$ and some $h\in\{0,k\}\cap 2\N$.

More generally, given two brake solutions $(x_j,jT)$ of {\rm (HS)$_j$} and $(x_k,kT)$ of {\rm (HS)$_k$} in \eqref{e:HS-jT}, we shall say they are geometrically distinct if
\[l*x_j\neq h*x_k \text{ for all } l,h\in\Z.\]
\begin{lemma}\label{l:semi-periodic-even-shift}
	For $j\in2\N$.
	If $(x_j,jT)$ is a $jT$-periodic brake solution of \eqref{e:HS-jT}, then
	$\frac{j}{2}*x_j$ is also a $jT$-periodic brake solution of \eqref{e:HS-jT} and we have
	\[i_{L_0}(x_j)=i_{L_0}(\frac{j}{2}*x_j), \ \ \ \nu_{L_0}(x_j)=\nu_{L_0}(\frac{j}{2}*x_j).\]
\end{lemma}
\begin{proof}
	Since we assume $j\in2\N$, direct verification shows that $\frac{j}{2}*x_j$ is also a $jT$-periodic brake solution of \eqref{e:HS-jT}.
	Let $B(t)=H''(t,x_j(t))$ for all $t\in\R$.
	Let $S=\frac{jT}{2}$.
	By \eqref{e:HS-jT} and {\bf(SH6)}, we have
    $x_j(t+S)=Nx_j(S-t)$ and
	$H''(t,Nx)=NH''(-t,x)N$ for $t\in\R$, $x\in\R^{2n}$. Then we have for $t\in\R$,
	\[B(t+S)=H''(t,x_j(t+S))=NH''(-t,x_j(S-t))N=NH''(S-t,x_j(S-t))N=NB(S-t)N.\]
	Let $\gamma_{x_j}$ be the associated symplectic path of $(x_j,jT)$. Then the associated symplectic path of $(\frac{j}{2}*x_j,jT)$ is $N\gamma_{x_j}(S-t)\gamma_{x_j}(S)^{-1}N$ for $t\in[0,S]$.
	Define $\gamma\in \mathcal{P}_{jT/2}(2n)$ by 
	$\gamma(t)=\gamma_{x_j}(t)$ for $t\in[0,jT/2]$.
	Thus by the definition of the Maslov-type index and \eqref{e:mas-index-define-R-gamma}, we have
	\[i_{L_0}(x_j)=i_{L_0}(\gamma)=i_{L_0}(N\gamma(S-t)\gamma(S)^{-1}N)=i_{L_0}(\frac{j}{2}*x_j).\]
	Since $\gamma(S)$ is a symplectic matrix, we have
	\[\nu_{L_0}(x_j)=\nu_{L_0}(\gamma(S))=\nu_{L_0}(N\gamma(S)^{-1}N)=\nu_{L_0}(\frac{j}{2}*x_j).\]
\end{proof}
\begin{lemma}\label{l:even-periodic-solution}
	Let $j,h\in\N$.
	If $(x_j,jT)$ and $(h*x_j,jT)$ are two $jT$-periodic brake solutions of \eqref{e:HS-jT},
	then $x_j$ is $2hT$-periodic, i.e., $x_j(t+2hT)=x_j(t)$ for $t\in\R$.
\end{lemma}
\begin{proof}
	Since $x_j(t+hT)=Nx_j(-t+hT)$ and $x_j(t)=Nx_j(-t)$, we have
	\[x_j(2hT)=Nx_j(0)=x_j(0).\]
	Since $H$ is $T$-periodic about $t\in\R$, $x_j$ and  $(2h)*x_j$ are the solutions of \eqref{e:HS-nonauto} with the same initial value. Thus $x_j=(2h)*x_j$.
\end{proof}
\begin{lemma}\label{l:shift-period}
	Let $j,k\in\N$ and $j<k$. Assume that $(x_j,jT)$ and $(x_k,kT)$ are two brake solutions of $\rm (HS)_j$ and $\rm (HS)_k$ in \eqref{e:HS-jT} respectively. If there exist $h,l\in\Z$ such that
	\[x_j(t+hT)= x_k(t+lT) \text{ for all } t\in\R,\]
	then there exists an $m\in\N$ such that
	\begin{equation}\label{e:even-periodic-solution}
		x_j(t+mT)=x_k(t)\ \ \ \ \ x_j(t+2mT)=x_j(t) \text{ for all } t\in\R.
	\end{equation}
\end{lemma}
\begin{proof}
	According to the assumption, there exits an $m\in\{0,1,...,j-1\}$ such that
	$m*x_j=x_k$. Then $x_j$ and $m*x_j$ are two $jT$-periodic brake solutions of \eqref{e:HS-jT}.
	By Lemma \ref{l:even-periodic-solution}, we get \eqref{e:even-periodic-solution}.
\end{proof}

Now we can prove our important proposition.
\begin{proposition}\label{p:index-geometrical-same}
	Let $j,k\in\N$ and $j<k$. Let $(x_j,jT)$ and $(x_k,kT)$ be two brake solutions of ${\rm(HS)_j}$ and ${\rm(HS)_k}$ in \eqref{e:HS-jT} respectively. Assume that $x_j$ and $x_k$ are not geometrically distinct, i.e.,  there exist $h,l\in\Z$ such that
	\[x_j(t+hT)= x_k(t+lT) \text{ for all } t\in\R.\]
	Then we have:
	\begin{enumerate}
		\item If $j$ is odd, then $x_j(t)=x_k(t)$ for all $t\in\R$.
		\item If $j\in 4\N-2$, then
		\begin{equation}\label{e:iteration-not4N}
			i_{L_0}(x_j)=i_{L_0}(x_k|_{[0,jT]}),\ \ \ \nu_{L_0}(x_j)=\nu_{L_0}(x_k|_{[0,jT]}).
		\end{equation}
		\item If $j\in 4\N$, then
		\[|i_{L_0}(x_j)-i_{L_0}(x_k|_{[0,jT]})|\le n.\]
	\end{enumerate}
	Let $\gamma_j\in \mathcal P_{\frac{jT}{2}}(2n)$ be the associated symplectic path of $x_j|_{[0,\frac{jT}{2}]}$.
	When  $k=pj$, $p\in \N$,
	we have
	\begin{equation}\label{e:index-brake-iteration-even}
		|i_{L_0}(\gamma^p_j)-i_{L_0}(x_{pj})|\le n.
	\end{equation}
\end{proposition}
\begin{proof}
	By Lemma \ref{l:shift-period},  there exits an $m\in\{0,1,...,j-1\}$ such that
	\eqref{e:even-periodic-solution} holds.
	Let $h$ be the greatest common divisor of $j$ and $2m$, which is denoted by $h=(j,2m)$.
	Then $h*x_j=x_j$.
	
	{\bf Case 1}: $j$ is odd. Then $h$ must be odd and $h|m$.
	Since $x_j(t+hT)=x_j(t)$, we have
	 \[x_j(t)=x_j(t+mT)=x_k(t)  \text{ for all } t\in\R.\]
	
	{\bf Case 2}:  $j\in4\N-2$.
	
	If $h\mid m$, then $x_k=m*x_j=x_j$.
	
	If $h\nmid m$, then $h$ is even, $\frac{h}{2}\mid m$. So $m=(2t'-1)\frac{h}{2}$ for some $t'\in\N$, thus
	$h\mid (m-\frac{h}{2})$. Then $x_k=m*x_j=(\frac{h}{2})*x_j$, 
	 and $j=h(2l-1)$ for some $l\in\N$.
	By Lemma \ref{l:semi-periodic-even-shift},
	we have
	\[i_{L_0}(x_j|_{[0,hT]})=i_{L_0}(x_k|_{[0,hT]}), \ \ \ \nu_{L_0}(x_j|_{[0,hT]})=\nu_{L_0}(x_k|_{[0,hT]}).\]
	Let $\gamma_{x_a}$ be the associated symplectic path of $x_a$ for $a=j,k$.
	Define $\gamma_a\in \mathcal{P}_{\frac{hT}{2}}(2n)$ by $\gamma_a(t)=\gamma_{x_a}(t)$ for $t\in[0,\frac{hT}{2}]$ and $a=j,k$.
	Then  
	\begin{gather*}
		i_{L_0}(\gamma_j)=i_{L_0}(x_j|_{[0,hT]})=i_{L_0}(x_k|_{[0,hT]})=i_{L_0}(\gamma_k),\quad  \nu_{L_0}(\gamma_j)=\nu_{L_0}(\gamma_k),\\
i_{L_0}(x_j)=i_{L_0}(\gamma^{2l-1}_j), \nu_{L_0}(x_j)=\nu_{L_0}(\gamma^{2l-1}_j), i_{L_0}(x_k|_{[0,jT]})=i_{L_0}(\gamma^{2l-1}_k),\nu_{L_0}(x_k|_{[0,jT]})=\nu_{L_0}(\gamma^{2l-1}_k).
\end{gather*}
By the homotopy invariance property (cf. Lemma \ref{l:index-invariance}), we have \[i_{e^{\sqrt{-1}\theta}}(\gamma_j^2)=i_{e^{\sqrt{-1}\theta}}(x_j|_{[0,hT]})=i_{e^{\sqrt{-1}\theta}}(x_k|_{[0,hT]})=
i_{e^{\sqrt{-1}\theta}}(\gamma_k^2).\]
	Then by the Bott-type iteration formula  Proposition \ref{p:iteration-brake} for the odd times brake iteration,
	we get (\ref{e:iteration-not4N}).
	
	{\bf Case 3}: $j\in4\N$. Then $h$ is even, $\frac{h}{2}\mid m$ and
	$j=hl$ for some $l\in\N$.
	Similarly to Cases 1 and 2, if $h\mid m$, then $x_k=m*x_j=x_j$; thus 
	\[i_{L_0}(x_j)=i_{L_0}(x_k|_{[0,jT]}) \quad\text{and}\quad \nu_{L_0}(x_j)=\nu_{L_0}(x_k|_{[0,jT]}).\]
	
If $h\nmid m$, then $m=(2t'-1)\frac{h}{2}$ for some $t'\in\N$, thus
$h\mid (m-\frac{h}{2})$. Then $x_k=m*x_j=(\frac{h}{2})*x_j$.
	Let $\gamma_{x_j}$ be the associated symplectic path of $(x_j,jT)$.
	Set 
	\begin{equation}\label{e:shift-brake}
	\gamma(t)=\gamma_{x_j}(t) \text{ for } 0\le t\le \frac{hT}{2} \text { and }
	\tilde \gamma(t)=N\gamma(\frac{hT}{2}-t)\gamma(\frac{hT}{2})^{-1}N \text{ for } 0\le t\le \frac{hT}{2}.
	\end{equation}
	Then $\tilde \gamma(t)$ is the associated symplectic path of $x_k|_{[0,\frac{hT}{2}]}$.  By the $N$-brake iteration (cf. Definition \ref{d:iteration-brake}), we have
	$i_{L_0}(x_j)=i_{L_0}(\gamma^l)$ and $i_{L_0}(x_k|_{[0,jT]})=i_{L_0}(\tilde \gamma^l)$.
	Then by Lemma \ref{l:shift-index} and Corollary \ref{c:shift-two-times-index}, we have
	\begin{equation}\label{e:shift-half-period}
	i_{L_0}(\gamma)=i_{L_0}(\tilde \gamma) \text{\quad and\quad } |i_{L_0\times L_1}(\tilde\gamma)-i_{L_0\times L_1}(\gamma)|=|i_{L_1\times L_0}(\gamma)-i_{L_0\times L_1}(\gamma)|\le n.
	\end{equation}
	By the homotopy invariance property, we have $i_{e^{\sqrt{-1}\theta}}(\tilde\gamma^2)=i_{e^{\sqrt{-1}\theta}}(\gamma^2).$
	Thus by the Bott-type iteration formula  for the brake iteration, since $l$ may be even, we just have
	\[|i_{L_0}(x_j)-i_{L_0}(x_k|_{[0,jT]})|=|i_{L_0}(\gamma^l)-i_{L_0}(\tilde \gamma^l)|\le n.\]

Assume $k=pj$ for some $p\in\N$.
To prove \eqref{e:index-brake-iteration-even}, we only need to consider the last situation of Case 3, since the others are 
all trivial.
Since $j=lh$ and $\gamma,\tilde\gamma\in \mathcal{P}_{\frac{hT}{2}}(2n)$ (cf. \eqref{e:shift-brake}), we readily get
 \[i_{L_0}(\gamma_j^p)=i_{L_0}(\gamma^{pl})\quad {and}\quad i_{L_0}(x_{pj})=i_{L_0}(\tilde\gamma^{pl}).\]  By 
 \eqref{e:shift-half-period} and the Bott-type iteration formula Proposition \ref{p:iteration-brake}, we get
	\eqref{e:index-brake-iteration-even}.
\end{proof}

\begin{remark}\label{r:time-shift-change}
	For $j,k\in\N$ and $j<k$, let $(x_j,jT)$ and $(x_k,kT)$ be two brake solutions of ${\rm(HS)_j}$ and ${\rm(HS)_k}$ in \eqref{e:HS-jT} and assume that $x_j$ and $x_k$ are not geometrically distinct.
	If $j\in 4\N$, we cannot say $x_j|_{[0,jT]}$ and $x_k|_{[0,jT]}$ must have the same Maslov-type indices
	with the Lagrangian boundary condition $L_0$.
	For example, for $j=4$, if $(x_4,4T)$ is also $2T$-periodic, then
	$\tilde x(t):=x(t+T)$, $t\in\R$, is also a $2T$-periodic brake solution. By Proposition \ref{p:index-geometrical-same}.2, we have
	\[i_{L_0}(x_4|_{[0,2T]})=i_{L_0}(\tilde x|_{[0,2T]}),\] 
	but we cannot say $x_4|_{[0,4T]}$ and $\tilde x|_{[0,4T]}$ must have the same Maslov-type indices.
	See Remark \ref{r:shift-index} for a counterexample.

	\end{remark}
	\begin{proof}[Proof of Theorem \ref{t:super-quadratic-1}]
		According to \cite[Theorem 3.1]{LiLiu10}, 
	 for any $j\in\N$ and $1\le j< \frac{2\pi}{T\|\tilde B\|_{C^0}}$, the
		Hamiltonian system \eqref{e:HS-jT} possesses a non-constant $jT$-periodic solution $x_j$
		which satisfies
		\begin{equation}\label{e:index-solution}
			i_{L_0}(x_j)\leq 1\le  i_{L_0}(x_j)+\nu_{L_0}(x_j).
		\end{equation}
		In \cite[(30)]{LiLiu10},  {(\bf SH5)} is needed.
		We suppose that $(x_j,jT)$ and $(x_{pj},pjT)$  are not geometrically distinct, for some $j,p\in\N$ and
		$jp< \frac{2\pi}{T\|\tilde B\|_{C^0}}$.
		For $j\in \N$, we divide into two cases.

		{\bf Case 1} $j$ is odd or belongs to $4\N-2$.
		By Proposition \ref{p:index-geometrical-same},
		we have
		\begin{equation}\label{e:same-Maslov-index-shift}
			i_{L_0}(x_{pj}|_{[0,jT]})= i_{L_0}(x_j),\ \ \  \nu_{L_0}(x_{pj}|_{[0,jT]})= \nu_{L_0}(x_j).
		\end{equation}
		Let $\gamma_{x_{pj}}$
		be the associated symplectic path of $(x_{pj},pjT)$. Set $S:=\frac{jT}{2}$ and
		let $\gamma_S=\gamma_{x_{pj}}|_{[0,S]}$.
		Then $\gamma_{x_{pj}}|_{[0,pjT/2]}=\gamma_S^p$ and we have
		\begin{equation}\label{e:iteration-index}
			\begin{split}
				&
				i_{L_0}(\gamma_S^p)=i_{L_0}(\gamma_{x_{pj}}|_{[0,pjT/2]})=i_{L_0}(x_{pj}), \\ &i_{L_0}(\gamma_S)=i_{L_0}(\gamma_{x_{pj}}|_{[0,jT/2]})=i_{L_0}(x_{pj}|_{[0,jT]}))=i_{L_0}(x_j),
				\\
				&\nu_{L_0}(\gamma_S)=\nu_{L_0}(\gamma_{x_{pj}}|_{[0,jT/2]})=\nu_{L_0}(x_{pj}|_{[0,jT]}))=\nu_{L_0}(x_j).
			\end{split}
		\end{equation}
		{\bf Claim:} $ p\le n+1.$
		Indeed, by Corollary \ref{0-1}, we have
		\begin{equation}\label{p-n}
			i_{L_0}(\gamma_S^p) \ge p( i_{L_0}(\gamma_S)+\nu_{L_0}(\gamma_S))-n.
		\end{equation}
		By \eqref{e:index-solution}, \eqref{e:iteration-index} and \eqref{e:same-Maslov-index-shift}, we have
		$i_{L_0}(\gamma_S^p)\le1$ and $i_{L_0}(\gamma_S)+\nu_{L_0}(\gamma_S)\ge1$.
		Based on \eqref{p-n}, we have
		\begin{equation}\label{ge-n}
			1 \ge i_{L_0}(\gamma_{x_{pj}}) \ge p-n.
		\end{equation}
		Thus we have $p\le n+1$.
		
		By \eqref{e:iteration-ineq}, we have
		\begin{equation}\label{e:iteration-ineq-copy}
			i_{L_0}(\gamma^p_S)\ge pi_{L_0}(\gamma_S)+(p-1)\nu_{L_0}(\gamma_S).
		\end{equation}
		If $x_j$
		is nondegenerate, then $\nu_{L_0}(x_j)$=0. Thus
		 according to \eqref{e:index-solution}, we get $i_{L_0}(\gamma_S)\ge 1$. From \eqref{e:iteration-ineq-copy}, we have $p\le1$.

		Moreover, if we assume $H$ satisfies {\bf(SH7)}, then
		 $i_{L_0}(\gamma_S)\ge 0$. Together with \eqref{e:iteration-ineq-copy}, we obtain $p\le 2$.
		In fact, if $i_{L_0}(\gamma_S)=1$, we get $p\le 1$.
		If $i_{L_0}(\gamma_S)=0$, then by \eqref{e:index-solution}, we get $\nu_{L_0}(\gamma_S)\ge1$; by \eqref{e:iteration-ineq-copy}, 
		we obtain $p\le2$.

		{\bf Case 2} $j \in 4\N$.
		Let $\gamma_{x_j}$
		be the associated symplectic path of $(x_j,jT)$. Set $S:=\frac{jT}{2}$ and
		let $\gamma_S=\gamma_{x_j}|_{[0,S]}$.
		By definition, $i_{L_0}(x_j)=i_{L_0}(\gamma_S)$ and $\nu_{L_0}(x_j)=\nu_{L_0}(\gamma_S)$.
		By Proposition \ref{p:index-geometrical-same},
		we have
		\begin{equation}\label{e:n-Maslov-index-shift}
			|i_{L_0}(x_{pj})-i_{L_0}(\gamma^p_S)|\le n.
		\end{equation}
		{\bf Claim:} $ p\le 2n+1.$
		Indeed,
		together with \eqref{e:index-solution} and \eqref{e:n-Maslov-index-shift}, we have
		\begin{equation}\label{e:iter-4N}
		i_{L_0}(\gamma_{S}^p)\le i_{L_0}(x_{pj})+n\le n+1\quad \text{and}\quad i_{L_0}(\gamma_{S})+\nu_{L_0}(\gamma_{S})\ge1.
		\end{equation}
		Based on \eqref{p-n}, we have
		\begin{equation}
			n+1 \ge i_{L_0}(\gamma_{S}^p) \ge p-n.
		\end{equation}
		Thus we have $p\le 2n+1$.
		
		If $x_j$
		is nondegenerate, then $\nu_{L_0}(x_j)$=0. Thus according to \eqref{e:index-solution}, we get $i_{L_0}(x_j)\ge 1$. From \eqref{e:iteration-ineq-copy} and \eqref{e:iter-4N}, we have $p\le n+1$.
		
		If we assume $H$ satisfies {\bf(SH7)}, then
		$i_{L_0}(\gamma_S)\ge 0$. Together with \eqref{e:iteration-ineq-copy}, we get $p\le n+2$.
		In fact, if $i_{L_0}(\gamma_S)=1$,  then by \eqref{e:iter-4N} and \eqref{e:iteration-ineq-copy}, we get $p\le n+1$.
		If $i_{L_0}(\gamma_S)=0$, then by \eqref{e:iter-4N},  we get $\nu_{L_0}(\gamma_S)\ge1$; by \eqref{e:iteration-ineq-copy}, we obtain $p\le n+2$.
		
The proof of Theorem \ref{t:super-quadratic-1} is complete.
\end{proof}

\section{The Morse index, the Maslov-type indices and the relative Morse index}\label{s:Morse-Maslov-indices}
\subsection{Relation between the Morse and the Maslov-type index of a solution}\label{ss:Morse-Maslov-indices}
In Section \ref{ss:morse-index-topology}, for a non-constant $T$-periodic solution $\bar x$ of \eqref{e:HS-T-K}, we define the quadratic form $q_{T/2}$ on the Hilbert space $\hat L_{\circ}^2(I_T;\R^{2n})$ and its Morse index $m^-(\bar x)$. To understand the Morse index $m^-(\bar x)$ more sufficiently, we will prove a index theorem relating the conjugate points along the curve $\bar x$.
According to the \textit{a priori} estimate Proposition \ref{p:aprior-estimate}, we can choose $K$ large enough such that 
$\bar x$ is a solution of \eqref{e:HS-jT} and $\hat H_K^{''}(t,\bar x(t))=H''(t,\bar x(t))$ for $t\in\R$.

Thus in this section, given  a $T$-periodic brake solution of \eqref{e:HS-jT} $(\bar x,T)$ such that $H_{qq}(t,\bar x(t))$ is positive definite for all $t\in\R$, 
we write $B(t)= H^{''}(t,\bar x(t))$ for short.
Similarly to \cite{Lo-Z-Zhu}, for $s>0$, we define the Hilbert spaces
\[\hat L^2_s=\{u\in L^2([-s,s];\R^{2n})|  u(-t)=Nu(t) \text{ for } t\in[-s,s]\}, \quad \hat L^2_{s,\circ}=\{u\in  \hat L^2_s| \int_{-s}^s u(t)dt=0\},\]
and a symmetric bilinear form on $\hat L^2_{s,\circ}$ by
\begin{equation}\label{e:quadratic-qs}
q_s(u,v)=\int_{-s}^s\frac{1}{2}[ (-J\Pi_su(t)+J\Lambda J\Pi_s^2u,v(t))+((B(t)+\Lambda)^{-1} u(t),v(t))
]dt,\end{equation}
where $\Lambda=\begin{pmatrix}
	\lambda I_n&0\\
	0&0
\end{pmatrix}\in \R^{2n\times 2n}$ such that $B(t)+\Lambda$ is positive definite for $t\in\R$.
Here  $\Pi_s$
denotes the primitive of $u$ with mean value zero:
\[\frac{d}{dt}\Pi_s u=u \qquad\text{ and }\qquad \int_{-s}^s (\Pi_s u)(t)dt=0.\]
By the almost same reason as in \cite[Lemma I.4.1]{Ek90} (cf. \eqref{e:Pi-example}),
$J\Pi_s$ is a compact self-adjoint operator from  $\hat L^2_{s,\circ}$ to itself,
and $\|\Pi_s\|\le \frac{s}{\pi}$.
Thus $q_s$ induces a Fredholm operator from $\hat L^2_{s,\circ}$ to itself, more precisely, a compact perturbation of an invertible operator.
According to \cite[Proposition I.4.2]{Ek90},
there is an orthogonal decomposition with respect to the
bilinear form $q_s$ on
$\hat L^2_{s,\circ}=E_+(s)\oplus E_0(s)\oplus E_-(s)$ such that
$E_+(s)$, $E_0(s)$ and $E_-(s)$ are the positive definite, null and negative definite spaces of $q_s$, respectively;
and $E_-(s)$ and $E_0(s)$ are finite-dimensional.
Denote by $i_s=\dim E_-(s)$ and $\nu_s=\dim E_0(s)$.
We generalize  \cite[Lemma 8.1]{Lo-Z-Zhu} to
\begin{lemma}\label{l:nullity}
	The nullity $\nu_s$ is  the number of linearly independent solutions of
	the boundary-value problem
	\begin{equation}\label{e:linear-brake}
		\left\{
		\begin{array}{ll}
			\dot y=JB(t)y\\
			y(0)\in L_0 \text { and } y(s)\in L_0.
		\end{array}
		\right.
	\end{equation}
\end{lemma}
\begin{proof}
	Recall $L_0=\{0\}\times\R^n$, the subspace of constant functions in $\hat L^2_s$.
	We have the $L^2$-orthogonal decomposition
	\[\hat L^2_s=\hat L^2_{s,\circ}\oplus^{\perp} L_0.\]
	The kernel of $q_s$ consists of all $u\in \hat L^2_{s,\circ}$ such that $q_s(u,v)=0$ for all
	$v\in \hat L^2_{s,\circ}$. This means there is some constant $\xi\in L_0$ such that
	\begin{equation*}
		-J\Pi_su+J\Lambda J\Pi_s^2u+(B(t)+\Lambda)^{-1} u=\xi,
	\end{equation*}
	which we rewrite as 
	\begin{equation}\label{e:nullity-equation}
		u=(B+\Lambda)(\xi+J\Pi_s u-J\Lambda\Pi_s J\Pi_su).
	\end{equation}
	Now define $y$ to be $\xi+J\Pi_s u-J\Lambda \Pi_sJ\Pi_su$.
	Then $y(0)\in L_0$, $y(s)\in L_0$, and 
	\eqref{e:nullity-equation} reads as 
	$\dot y=JB(t)y$, which is precisely the system \eqref{e:linear-brake}.
	
	On the other hand, for any  solution $y$ of \eqref{e:linear-brake}, we can extend $y$ into $\hat L^2_s$ as follows:
	\[y(t)=Ny(-t),\quad  \forall t\in [-s,0].\]
	Then $u=-J\dot y+\Lambda y\in  \hat L_{s,\circ}^2$ belongs to the kernel of $q_s$. Thus the lemma is proved.
\end{proof}

Then we have the following index theorem, which is a direct
consequence of \cite[Theorem 1.11]{Uhlen73}.
\begin{lemma}\label{l:morse-index}
	Fix a number $S\in(0,\frac{T}{2}]$. We have
	\[i_S=\sum_{0<s<S}\nu_s,\]
	where $\{s\in (0,S);\nu_s\neq 0\}$ consists of finite points.
\end{lemma}
\begin{proof}
Since  $\left(\frac{\partial^2 H}{\partial q_i \partial q_j}(t,\bar x)\right)$ is positive definite, for $\xi\in L_0=\{0\}\times \R^n$,
	$ H''(t,\bar x(t))\xi=0$ implies $\xi=0$.
	The detailed proof will be given in Appendix 
	\ref{a:morse}.
\end{proof}
By definition $i_{T/2}=m^-(q_{T/2})=m^-(\bar x)$. 
Now we give the relation of the Maslov-type index and the Morse index for a brake solution of the partially positive definite reversible Hamiltonian system.
\begin{proposition}\label{p:morse-maslov-index}
	Let $(\bar x,T)$ be a $T$-periodic brake solution of \eqref{e:HS-jT} such that $H_{qq}(t,\bar x(t))$ is positive definite for all $t\in\R$. Then we have
	\[i_{L_0}(\bar x)=m^-(q_{T/2}).\]
\end{proposition}
\begin{proof}
	Let $\gamma_{\bar x}$ be the associated symplectic path of $(\bar x,T)$ and write $\gamma=\gamma_{\bar x}|_{[0,T/2]}$ for short.
	Then by definition, we have $i_{L_0}(\bar x)=i_{L_0}(\gamma)$.
Then the restriction of $B(t)=H''(t,\bar x(t))$ on $L_0=\{0\}\times \R^n$ is positive definite for any $t\in\R$.
	Together with \cite[Proposition 2.5]{zhou-zhu-wu} and \cite[Lemma 3.1]{duistermaat}, we obtain
	\begin{equation*}
		\begin{split}
			i_{L_0}(\gamma)&=\Mas\{\Graph(\gamma),\widetilde\alpha_0\}-n,\\
			&=\dim L_0 +\sum_{0<t<T/2}\dim_{\C} (\Graph(\gamma(t))\cap\widetilde \alpha_0)-n\\
			&=\sum_{0<t<T/2}\dim (\gamma(t)L_0\cap L_0)\\
			&=\sum_{0<t<T/2} \nu_{L_0}(\gamma(t))
		\end{split}
	\end{equation*}
	Since every solution of \eqref{e:linear-brake} is of the form
	$\gamma(t)y(0)$ where $y(0)\in L_0$ and $\gamma(s)y(0)\in L_0$.
	By Lemma \ref{l:nullity}, we have
	$\nu_{L_0}(\gamma(s))=\nu_s$.
	Then together with Lemma \ref{l:morse-index}, we get our proposition.
\end{proof}
\begin{remark}
	We can also use the relative Morse index to prove Proposition \ref{p:morse-maslov-index}; please see the following Subsection \ref{ss:relative-morse-index}.
\end{remark}
For short write and recall
\[\hat W_{2T}:=\hat W^{1,2}_{2T}(I_{2T};\R^{2n})=\{x\in W^{1,2}(\R;\R^{2n}); x(-t)=Nx(t), x(t+2T)=x(t), \;\text{ for }\; t\in \R\}.\]
Let $(e_i)$, $i=1,2,\cdots,2n$, be the standard basis for $\R^{2n}$.
Define $\eta_{jk}=\sin \frac{\pi jt}{T}e_k+\cos \frac{\pi jt}{T}e_{k+n}$, for $j\in\Z$ and $k\in\{1,2,\cdots,n\}$.
Then we represent $x\in\hat W_{2T}$ by its Fourier series
\[x=\sum^n_{k=1}\sum_{j\in\Z}a_{jk}\eta_{jk},\]
where $a_{jk}\in\R$ and $\sum^n_{k=1}\sum_{j\in\Z}(1+j^2)|a_{jk}|^2<\infty$ and
write
\[\hat W_{2T}=\Span_{\R}\{\eta_{j1},\eta_{j2},\cdots,\eta_{jn}; j\in \Z\}.\]
Let
\begin{equation*}
	\begin{split}
		\hat W_{T}&=\Span_{\R}\{\eta_{j1},\eta_{j2},\cdots,\eta_{jn}; j\in 2\Z\}\subset\hat W_{2T},\\
		\hat W^{-1}_T&=\Span_{\R}\{\eta_{j1},\eta_{j2},\cdots,\eta_{jn}; j\in 2\Z+1\}\subset\hat W_{2T}.
	\end{split}
\end{equation*}
Given a $T$-periodic solution $x^*$ of \eqref{e:HS-jT},
define $B^*(t)=(H^{''}(t, x^*(t))+\Lambda)^{-1}$, where 
$\Lambda=\begin{pmatrix}
	\lambda I_n&0\\
	0&0
\end{pmatrix}$ and we choose suitable $\lambda\ge0$ such that 
$B^*$ is positive definite.
We define the quadratic form on $\hat W_{2T}$ by
\[Q_T(y,y)=\frac{1}{2}\int_{-T}^T[(J\dot y-\Lambda y,y)+(B^*(t)(J\dot y-\Lambda y),(J\dot y-\Lambda y))]dt, \text{ for } y\in \hat W_{2T}. \]
Then we readily have
\begin{lemma}\label{l:decomposition}
	We have the following $L^2$-orthogonal decomposition with respect to both the inner product and the quadratic form $Q_{T}$ on $\hat W_{2T}$ 
	\[\hat W_{2T}=\hat W_T\oplus \hat W^{-1}_T. \]
\end{lemma}
\begin{proof}
Note that for $x\in \hat W_T$, $x(t+T)=x(t)$; for $x\in \hat W^{-1}_T$, $x(t+T)=-x(t)$.
$B^*(t+T)=B^*(t)$ for $t\in\R$.
\end{proof}
By $j_{-1}(x^*)$ we denote the Morse index of the quadratic form $Q_T$ restricted to $\hat W^{-1}_T$, i.e., the 
maximal dimension of a subspace of  $\hat W^{-1}_T$ on which $Q_T$ is negative definite. Then we have
\begin{lemma}\label{l:morse-decomposition}
	\[m^-(Q_T)=i_{L_0}(x^*)+j_{-1}(x^*).\]
\end{lemma}
\begin{proof}
	Note $-J\frac{d}{dt}+\Lambda$ induces isomorphisms from $\hat W_{2T}$ to $\hat L^2_{\circ}(I_{2T};\R^{2n})$ and
	from $\hat W_{T}$ to $\hat L^2_{\circ}(I_T;\R^{2n})$.
	Together with Lemma \ref{l:decomposition} and Proposition \ref{p:morse-maslov-index}, we obtain
	\[m^-(Q_T)=m^-(q_{T/2})+j_{-1}(x^*)=i_{L_0}(x^*)+j_{-1}(x^*).\]
\end{proof}
\subsection{The relative Morse index}\label{ss:relative-morse-index}
Now we prove the following two lemmas, which will be used in the proof of Lemma \ref{l:i+nu-ge-1}.
In this subsection, let $X$ be a Hilbert space with an inner product $\langle \cdot,\cdot\rangle$ and the norm $\|x\|\equiv\langle x,x\rangle ^{1/2}$ for all $x\in X$.
For any bounded linear operator $P$ on $X$, we also denote the norm of $P$ by $\|P\|$.
\begin{lemma}\label{l:relative-morse-index}
	Let $X$ be a Hilbert space.
	Let $A$ be a closed self-adjoint Fredholm operator in $X$, and $B$ be a bounded selfadjoint operator on $X$ relatively compact with respect to $A$. 
	Then the relative Morse index $I(A,A-B)$ is well-defined (cf. \cite[Definition 2.8]{Zhu-Long99}).
	Assume that $B$ is semi-positive definite on $X$, i.e., $\langle Bx,x\rangle\ge 0$ for any $x\in X$. Then we have
	\begin{enumerate}
		\item[$1^{\circ}$.] $I(A,A-B)=\displaystyle\sum_{0\le s<1}(\dim \ker (A-sB)-\lim_{t\to s^+}\dim \ker (A-tB))$.
		\item[$2^{\circ}$.]
		\begin{equation*}
			I(A,A-B)+\dim \ker (A-B)-\dim \ker A=\sum_{0< s\le 1}(\dim \ker (A-sB)-\lim_{t\to s^-} \dim \ker (A-tB)).
		\end{equation*}
	\end{enumerate}
\end{lemma}
\begin{proof}
	1:
	According to the definitions of the spectral flow and the relative Morse index $I(A,A-B)=-\SF\{A-sB\}$, $0\le s\le1$,
	we can choose a partition, $\{0=s_0<s_1<\cdots <s_k=1\}$ of the interval $[0,1]$ and positive real numbers $\varepsilon_j$, $j=1,2,\ldots,k$, such that for each $j=1,2,\ldots,k$, the function $s\mapsto E_j(s):=1_{[-\varepsilon_j,\varepsilon_j]}(A-sB)$ is continuous and of finite rank on $[s_{j-1},s_j]$,
	where $1_{[-\varepsilon_j,\varepsilon_j]}(A-sB)=\frac{1}{2\pi i}\int_{\Gamma_j}(z-(A-sB))^{-1}dz$, $\Gamma_j$ is the positively oriented circle with center at $0$ and radius $\varepsilon_j$. Then
	\begin{equation}\label{e:def-relative-morse-index}
		I(A,A-B)=\sum_{j=1}^k (m^-(E_j(s_j)(A-s_jB)E_j(s_j))-m^-(E_j(s_{j-1})(A-s_{j-1}B)E_j(s_{j-1}))).
	\end{equation}
	
	Now we only need to prove for each $j=1,2,\ldots,k$,
	\begin{equation}\label{e:relative-morse-index-subinterval}
		\begin{split}
			&\quad \quad \  m^-(E_j(s_j)(A-s_jB)E_j(s_j))-m^-(E_j(s_{j-1})(A-s_{j-1}B)E_j(s_{j-1}))\\
			&=\sum_{s_{j-1}\le s <s_j}(\dim \ker (A-sB)-\lim_{t\to s^+}\dim \ker (A-tB)).
		\end{split}
	\end{equation}
	
	Fix a $j\in \{1,2,\ldots,k\}$.
	By the definition of the spectral flow and spectral theory,
	for $s\in[s_{j-1},s_j]$, $\sigma(A-sB)\cap[-\varepsilon_j,\varepsilon_j]$ consists of the eigenvalues
	\[-\varepsilon_j<\lambda_1(s)\le\lambda_2(s)\le\cdots \le\lambda_m(s)< \varepsilon_j,\]
	where $m=\dim \im E_j(s_j)=\dim \im E_j(s_{j-1})$.
	By Theorem \ref{t:continuity-eigemvalue}, for each $l=1,2,\ldots,m$, $\lambda_l(s)$ is continuous in
	$s\in[s_{j-1},s_j]$.
	
	To prove \eqref{e:relative-morse-index-subinterval},
	we prove that for each $l=1,2,\ldots,m$, $\lambda_l(s)$ is  monotone non-increasing in $s\in[s_{j-1},s_j]$.
	Fix an $l\in\{1,2,\ldots,m\}$. For any $t\in [s_{j-1},s_j]$, let $\{t_n\in [s_{j-1},s_j]; n\in\N\}$ be a sequence converging to $t$ such that
	\begin{equation*}
		a_l(t):=\liminf_{s\to t}\frac{\lambda_l(s)-\lambda_l(t)}{s-t}=\lim_{n\to+\infty}\frac{\lambda_l(t_n)-\lambda_l(t)}{t_n-t}.
	\end{equation*}

	For each $n\in\N$, we pick up an $x(t_n)\in \ker(A-t_nB-\lambda_l(t_n)I)$ with $\|x(t_n)\|=1$. Then we have
	$x(t_n)\in \im E_j(t_n)$. So $\{x(t_n);n\in \N\}$ has a convergent subsequence still denoted by $x(t_n)$ such that
	$\displaystyle\lim_{n\to +\infty}x(t_n)=x(t)$. (In fact, one way to deduce that is: We have $x(t_n)=E_j(t_n)x(t_n)$ and $E_j(t)x(t_n)\in \im E_j(t)$, thus $\|x(t_n)-E_j(t)x(t_n)\|\le\|E_j(t_n)-E_j(t)\|\|x(t_n)\|$.
	 Since $\displaystyle \lim_{n\to+\infty}\|E_j(t_n)-E_j(t)\|\to0$  and $\im E_j(t)$ is finite-dimensional, we get the convergent subsequence.) 
	 Since $\displaystyle \lim_{n\to+\infty}\lambda_l(t_n)=\lambda_l(t)$, by the Closed Graph Theorem,  $x(t)\in \ker(A-tB-\lambda_l(t)I)$. So we have
	\begin{equation*}
		\begin{split}
			a_l(t) & =\lim_{n\to+\infty}\frac{\lambda_l(t_n)-\lambda_l(t)}{t_n-t} \\
			& =\lim_{n\to+\infty}\frac{\lambda_l(t_n)-\lambda_l(t)}{t_n-t}\langle x(t_n),x(t)\rangle\\
			&=\lim_{n\to+\infty}\frac{\langle \lambda_l(t_n)x(t_n),x(t)\rangle-\langle x(t_n),\lambda_l(t)x(t)\rangle}{t_n-t}\\
			&=\lim_{n\to+\infty}\frac{\langle (A-t_nB)x(t_n),x(t)\rangle-\langle x(t_n),(A-tB)x(t)\rangle}{t_n-t}\\
			&=\lim_{n\to+\infty}\frac{\langle(A-t_nB-(A-tB))x(t_n),x(t)\rangle}{t_n-t}\\
			&=\lim_{n\to+\infty} -\langle Bx(t_n),x(t)\rangle\\
			&=-\langle Bx(t),x(t)\rangle\\
			&\le 0,
		\end{split}
	\end{equation*}
	where in the last inequality we used $B$ is semi-positive definite.
	According to the knowledge of real analysis, since $\displaystyle\liminf_{s\to t}\frac{\lambda_l(s)-\lambda_l(t)}{s-t}\le 0$,
	we have $\lambda_l(s)\le \lambda_l(t)$ for $s_{j-1}\le s\le t\le s_j$.
	Thus \eqref{e:relative-morse-index-subinterval} is followed. Then by \eqref{e:def-relative-morse-index},
	we get $1^{\circ}$.
	
	2: $2^{\circ}$ follows from $1^{\circ}$. In fact, since $B$ is semi-positive and relatively compact with respect to $A$, according to the proof of $1^{\circ}$,
	there exists a partition, $\{0= t_0<t_1<\cdots<t_m=1\}$ of $[0,1]$ such that for each $j=1,2,\ldots,m$, $\dim \ker(A-sB)$ is constant on $t_{j-1}<s<t_j$. Then $\displaystyle\lim_{s\to t^+_{j-1}}\dim \ker (A-sB)=\lim_{s\to t^-_j}\dim \ker (A-sB)$ for $j=1,2,\ldots,m$.
	So
	\begin{equation*}
		\begin{split}
			&\quad\  I(A,A-B)+\dim \ker(A-B)-\dim \ker A   \\
			&= \sum_{j=1}^{m}(\dim \ker (A-t_{j-1}B)-\lim_{s\to t^+_{j-1}}\dim \ker (A-sB))+\dim \ker(A-B)-\dim \ker A\\
			&= \sum_{j=1}^{m}(\dim \ker (A-t_jB)-\lim_{s\to t^-_j}\dim \ker (A-sB))\\
			&= \sum_{0< s\le 1}(\dim \ker (A-sB)-\lim_{t\to s^-} \dim \ker (A-tB)). 
		\end{split}
	\end{equation*}
\end{proof}

\begin{lemma}\label{l:ker-A-sB-constant}
	Let $X$ be a Hilbert space.
	Let $A$ be a closed self-adjoint operator in $X$, and $B$ be a bounded selfadjoint operator on $X$. Assume that $B$ is semi-positive definite on $X$, i.e., $\langle Bx,x\rangle\ge 0$ for any $x\in X$.
	Let $U$ be a non-trivial interval in $\R$ and assume that  $A-sB$ is Fredholm for any $s\in U$.
	If $\dim \ker(A-sB)=\const$ for all $s\in U$,
	then $\ker (A-sB)=\ker A\cap \ker B$ for $s\in U$.
\end{lemma}
\begin{proof}
	It is obvious that $\ker (A-sB)\supset\ker A\cap \ker B$.
	Now we prove that $\ker (A-sB)\subset\ker A\cap \ker B$.
	Since $\dim \ker(A-sB)=\const$ on $s\in U$.
	Then the family of subspaces of $X$, $\ker (A-sB)$ is continuous in $s\in U$ in the gap topology (cf. \cite[Appendixes A.2 and A.3]{BoZh14}). 
	
	Let $u\in \ker (A-sB)$ for some $s\in U$.
	Then there exists a sequence $s_k\in U$, $k\in \N $, such that $s_k\neq s$, $\displaystyle \lim_{k\to +\infty}s_k=s$, and a sequence
	$u_k\in \ker (A-s_kB)$, $k\in \N $ such that $\displaystyle \lim_{k\to +\infty}u_k =u$ in X.
	Then
	\begin{equation*}
		\begin{split}
			0&=\langle (A-s_kB)u_k,u\rangle \\
			&=\langle Au_k,u\rangle-\langle s_kBu_k,u\rangle\\
			&=\langle sBu_k,u\rangle-\langle s_kBu_k,u\rangle.
		\end{split}
	\end{equation*}
	Thus $\langle Bu_k,u\rangle=0$, then $\langle Bu,u\rangle=0$.
	Since for any $v$, $t\in\R$, $2t\langle Bu,v\rangle+t^2\langle Bv,v\rangle=\langle B(u+tv),u+tv\rangle\ge0$, we get $Bu=0$. Thus $Au=0$ and we are done.
\end{proof}

\section{Proof of the main theorems }\label{s:partial-convex-symmetric}
\subsection{Brake subharmonics for partially convex reversible nonautonomous Hamiltonian systems}
\begin{lemma}\label{l:simple-period}
	Let $\bar u$ be a mountain-pass essential point of $\psi$ on $\hat L^{\beta}_{\circ}(I_{2^kT};\R^{2n})$ (cf. \eqref{e:reduced-functional}), where $k\ge1$, i.e.,
	\[\psi(u)=\int_{-2^{k-1}T}^{2^{k-1}T}\left[ \frac{1}{2}(-J\Pi u+J\Lambda J\Pi^2u,u)+F^*(t;u)\right]dt\]
	for $u\in \hat L^{\beta}_{\circ}(I_{2^kT};\R^{2n})$, here $\Pi u$ denote the primitive of $u$ whose mean over $I_{2^kT}$ is zero and $\Lambda=\begin{pmatrix}
		\lambda I_n&0\\
		0&0
		\end{pmatrix}\in \R^{2n\times 2n}$.
	 Suppose $\bar x\in \hat W_{2^kT}$ satisfying $-J\frac{d}{dt}\bar x+\Lambda \bar x=\bar u$ is a brake solution of the corresponding Hamiltonian system. Assume
	$i_{L_0}(\bar x)=1$. Then $2^kT$ is $\bar x$'s simple period.
\end{lemma}
\begin{proof}
	Arguing  by contradiction, we 
	assume there is a positive integer $j<k$ such that $\bar x$ is $2^jT$-periodic. Then it must be $2^{k-1}T$-periodic.
	Let $\gamma_{\bar x}$ be the associated symplectic path of $\bar x$.
Set $S=2^{k-2}T$ and define $\gamma\in\mathcal{P}_S(2n)$ by
$\gamma(t)=\gamma_{\bar x}(t)$ for $t\in[0,S]$.
Then \[1=i_{L_0}(\bar x)=i_{L_0}(\gamma^2).\]
	By Proposition \ref{p:iteration-brake-inequality}, we get $i_{L_0}(\gamma)=0$.
Define the $2^{k-1}T$-periodic brake solution by $x^*=\bar x|_{[-S,S]}$.
Then $i_{L_0}(x^*)=i_{L_0}(\gamma)$.
In view of Lemma \ref{l:morse-decomposition},
$j_{-1}(x^*)=1$, for $m^-(Q_{2S})=i_{L_0}(\bar x)=1$.
Thus the Morse index of the corresponding $q_{2S}$ defined in Section \ref{s:Morse-Maslov-indices} is one; and $\Span_{R}\{e\}$ is  one of the maximal negative definite subspace of the quadratic form $q_{2S}$,
where $e(t+2S)=-e(t)$.
According to Proposition \ref{p:morse-index--1}, we get $m^-(\bar x)\neq 1$.
Since $m^-(\bar x)=i_{L_0}(\bar x)=1$, we get a contradiction.
Thus we get this lemma.
\end{proof}

Now we give the \textit{a priori} estimate of the solutions we found by the mountain-pass theorem.
\begin{proposition}\label{p:aprior-estimate}
	Let $\bar u$  be the critical point of $\psi$ (cf. \eqref{e:reduced-functional}) we found by Mountain Pass Theorem \ref{t:critical-point-mp} and $\bar x\in \hat W_{T}$ with $-J\frac{d}{dt}\bar x+\Lambda \bar x=\bar u$ be the solution of \eqref{e:HS-T-K} we found by Proposition \ref{p:dual-critical-solution}. 
	Then there is a constant $C>0$ independent of $K$ and $\lambda$, such that
	\[\|\bar x\|_{L^{\infty}}\le C.\]
	\end{proposition}
\begin{proof}
	First, we choose $K>\max\{\bar r,\bar R_1\}$, where $\bar R_1=\max\{R_1,1\}$.
By Proposition \ref{p:dual-action-principle}, we obtain $\Phi_K(\bar x)+\psi(\bar u)=0$.
By {\bf(SH3)} and \eqref{e:H-super-quadratic},
we have
\[ x\cdot \hat H'_K(t, x)\ge \alpha \hat H_K(t,x)-a_5\ge a_7|x|^{\alpha}-a_6\]
for all $x\in\R^{2n}$,
where $a_5,a_6,a_7$ are constants independent of $K$ and $a_7>0$. Moreover $a_5,a_6,a_7$ depend only on $\bar r,\alpha, T$ and $C_1,C_2$, where $C_1=\max\{H(t,x);t\in I_T;|x|\le \bar r\}$ and 
$C_2=\max\{|H'(t,x)|;t\in I_T;|x|\le \bar r\}$.
Together with the  fact that $\bar x$ is the solution of the nonautonomous system \eqref{e:HS-T-K}, we have
\begin{equation}\label{e:critical-value-big}
	\begin{split}
		d &= \psi (\bar u)=-\Phi_K(\bar x)=\int_{T/2}^{T/2}-\frac{1}{2}(J\dot {\bar x},\bar x)-\hat H_K(t,\bar x)\ dt \\
		&=\int_{-T/2}^{T/2} \frac{1}{2}(\hat H_K'(t,\bar x),\bar x)-\hat H_K(t,\bar x)\ dt\\
		&\ge \int_{-T/2}^{T/2} (\frac{1}{2}-\frac{1}{\alpha})(\hat H_K'(t,\bar x),\bar x)dt-\frac{a_5}{\alpha}T.
	\end{split}
\end{equation}
Hence 
\[\|\bar x\|_{L^{\alpha}}\le \left[ \frac{2\alpha}{(\alpha-2) a_7}(d+\frac{a_5}{\alpha}T)+\frac{a_6}{a_7}T\right]^{1/\alpha}.\]

By {\bf(SH5)} for $\hat H_K$ and \eqref{e:critical-value-big}, we obtain 
\begin{equation*}
	\begin{split}
		\|\dot {\bar x}\|_{L^1}&= \|J\hat H'_K(t,\bar x)\|_{L^1}\\
		&\le \bar \theta  \|(\hat H'_K(t,\bar x),\bar x)\|_{L^1}+C_3T\\
		&\le \frac{2\alpha \bar \theta}{\alpha-2}(d+\frac{a_5T}{\alpha})+C_3(\bar \theta\bar r+1)T,
\end{split}
\end{equation*}
where $C_3=\max\{|H'(t, x)|;t\in I_T, |x|\le \max \{\bar r,\bar R_1\}\}$.
Together with the Sobolev embedding inequality, we obtain
\begin{equation}\label{e:L-infity-estimate}
	\|\bar x\|_{L^{\infty}}\le \|\dot {\bar x}\|_{L^1}+a_{11}\|\bar x\|_{L^{\alpha}}\le a_1d+a_2,
\end{equation}
where $a_1,a_2,a_{11}$ are constants independent of $K$.

Now we estimate the critical value $d$. As shown in \eqref{e:psi-small-0}, 
we can choose $u_1(t)=\frac{2\pi}{T}\begin{pmatrix}
	\sin \frac{2\pi t}{T} y_0\\
	-\cos \frac{2\pi t}{T} y_0
\end{pmatrix}+\lambda \begin{pmatrix}
	\sin \frac{2\pi t}{T} y_0\\
	0
\end{pmatrix}\in \R^{2n}$ such that $\psi(u_1)<0$, where $y_0\in\R^{n}$ is a constant vector independent of $\lambda $ and $K$.
By \eqref{e:mountain-pass}, we have
\begin{equation*}\label{e:critical-value-less}
d\le \max_{0\le s\le 1}\psi(su_1).
\end{equation*}
By \eqref{e:psi-upper-bound},
\[\psi(su_1)\le -\pi s^2|y_0|^2
+\sqrt{2}c_3\frac{(2\pi)^{^{\beta}}}{T^{\beta-1}}s^{\beta}|y_0|^{\beta}+
a_4T.\]
Similar to \cite[page 479 (21)]{AubinEke84},
the right-hand is maximal for $s=c_4\frac{\beta^{\frac{1}{2-\beta}}}{|y_0|}$, where $c_4$ is a constant independent of $K$ and $\lambda$.
Hence 
\begin{equation}\label{e:critical-value-estimate}
	d\le (c_5-c_6\beta)\beta^{\frac{\beta}{2-\beta}}+a_4T,
	\end{equation}
where $c_5,c_6,a_4$ are constants independent of $K$ and $\lambda$.
Combing  \eqref{e:L-infity-estimate} and \eqref{e:critical-value-estimate}, we get
\[\|\bar x\|_{L^{\infty}}\le C,\]
where $C$ is a constant independent of $K$ and $\lambda$.
\end{proof}
\begin{proof}[Proof of Theorem \ref{t:super-quadratic-2}]
	We first modify $H$ to $\hat H_K$, for  sufficient large $K$ such that $K>C$ in Proposition \ref{p:aprior-estimate}. 
	Then according to Theorems \ref{t:mountain-condition} and  \ref{t:critical-point-mp}, Proposition \ref{p:morse-maslov-index} and 
	Theorem \ref{t:two-path-components}, for any $T>0$ and $j\in\N$, there exists a non-constant $jT$-periodic solution $x_j$ of {$\rm(HS)_j$} \eqref{e:HS-jT} with Maslov-type indices satisfying 
	\begin{equation}\label{e:index-estimate}
	i_{L_0}(x_j)\le 1\le i_{L_0}(x_j)+\nu_{L_0}(x_j),
	\end{equation}
which reformulates the Morse index \eqref{e:morse-index-less1-greater1}.
	Note $-J\frac{d}{dt}x_j+\Lambda x_j$ are the mountain-pass essential points of the corresponding dual functions
	$\psi$.
	Assume $j\nmid4$. If $x_j$ and $x_{jp}$ are not  geometrically distinct, then 
	according to the partially strictly convex case of Theorem \ref{t:super-quadratic-1}, 
	we get $p\le 2$. Assume $p=2$.
	Let $\gamma_{x_{2j}}$ be the associated symplectic path of $x_{2j}$.
	Set $S=\frac{jT}{2}$ and define $\gamma\in\mathcal{P}_{S}(2n)$ by $\gamma(t)=\gamma_{x_{2j}}(t)$ for $t\in[0,S]$.
	Then we have $i_{L_0}(x_j)=i_{L_0}(\gamma)$, $i_{L_0}(x_{2j})=i_{L_0}(\gamma^2)$.
According to Proposition \ref{p:iteration-brake-inequality} and \eqref{e:index-estimate}, we get $i_{L_0}(\gamma^2)=1$ and 
	$i_{L_0}(\gamma)=0$. By the same argument in Lemma \ref{l:simple-period}, we get a contradiction.
	Thus $p\neq2$, and $p=1$.
	The proof of Case 1 of Theorem \ref{t:super-quadratic-2} is complete.
	
	Assume $j=2^l$ for some $l\in \N$. That $x_j$ is nondegenerate means $\nu_{L_0}(x_j)=0$.
According to  \eqref{e:index-estimate}, we get $i_{L_0}(x_j)=1$.
Then this case also follows from Lemma \ref{l:simple-period}.
	
	
	\end{proof}
\subsection{Rabinowitz minimal period conjecture for brake solutions of partially convex Hamiltonian systems}
Given  a non-constant $T$-periodic solution of \eqref{e:HS-T} and define
\[T_0=\inf\{a>0| \bar{x}(t+a)=\bar{x}(t), \forall t\in\R\}.\]
Then $T_0>0$ and $T_0=\frac{T}{k}$ for some $k\in\N$.
$T_0$ is the minimal positive period of $\bar x$.
\begin{proposition}\label{p:critical}
	Suppose $H$ satisfies {\rm (H1)-(H4)} and {\rm (H0)$^{q+}$}.
	Let \[\psi(u)=\int_{-\frac{T}{2}}^{\frac{T}{2}}\left[ \frac{1}{2}(-J\Pi u+J\Lambda J\Pi^2u,u)+F^*(u)\right]dt\]
	as in \eqref{e:reduced-functional} for our autonomous system, where $F(x)=\hat H_K(x)+\frac{1}{2}(\Lambda x,x)$ is strictly convex.
	Then $\psi$ has a nonzero critical point $\bar u$; and 
	 there is a non-constant period brake solution $\bar x$
	 of the following \eqref{e:HS-T-K-auto} such that $-J\frac{d}{dt}\bar x +\Lambda \bar x=\bar u$ and
	\begin{equation}\label{e:a-priori-estimate-solution}
		|\bar x(t)|\le \max\{C,\bar r\} \text{\ \ \  for all } t\in\R,
	\end{equation}
	where $C$ is a constant depending only on $T,\alpha,\bar r$, $C_1=\displaystyle \max_{|x|\le \bar r}|H'(x)|$ and $C_2= \displaystyle\max_{|x|\le \bar r}H(x)$.
\end{proposition}
\begin{proof}
	According to Theorems \ref{t:mountain-condition} and  \ref{t:critical-point-mp},
	$\psi$
	has a nonzero critical point. Denote the critical point by $\bar u$.
	We have $d=\psi(\bar u)>0$ and $0 \xrightarrow[\psi]{} \bar u$.
	Then Proposition \ref{p:dual-action-principle} shows that
	there is a constant $\xi\in \{0\}\times \R^{n}$
	such that $\bar x(t)=J(\Pi \bar u)(t)-J\Lambda J(\Pi^2\bar u)(t)+\xi$, $t\in \R$, solves \eqref{e:HS-T-K-auto}.
	
	By similar and careful computations in \cite{Ek79},
	we have $\hat H_K(\bar x)\le c(T,\bar r)$, where $c(T,\bar r)$ is a constant depending only on
	the period $T$,  $\bar r$ in condition {\rm (H3)} and also the Hamiltonian function $H$.
	
	Now we deduce that. 
	By \cite[Theorem II.4.2]{Ek90}, we obtain $\Phi_K(\bar x)+\psi(\bar u)=0$. Together with {\rm(H3)} and that fact $\bar x$ is a $T$-periodic solution of the autonomous system \eqref{e:HS-T-K-auto}, we have
	\begin{equation*}
		\begin{split}
			d &= \int_0^T-(\frac{1}{2}J\dot {\bar x},\bar x)-\hat H_K(\bar x)\ dt \\
			&=\int_0^T \frac{1}{2}(\hat H_K'(\bar x),\bar x)-\hat H_K(\bar x)\ dt\\
			&= \int_{\{t\in[0,T];|\bar x(t)|\le \bar r\}} \frac{1}{2}(\hat H_K'(\bar x),\bar x) dt+
			\int_{\{t\in[0,T];|\bar x(t)|> \bar r\}} \frac{1}{2}(\hat H_K'(\bar x),\bar x)dt
			-\int_0^T\hat H_K(\bar x)dt\\
			&\ge \int_{\{t\in[0,T];|\bar x(t)|\le \bar r\}} \frac{1}{2}(\hat H_K'(\bar x),\bar x) dt+
			\frac{\alpha}{2}\int_{\{t\in[0,T];|\bar x(t)|> \bar r\}} \hat H_K(\bar x)dt
			-\int_0^T \hat H_K(\bar x)dt\\
			&= \int_{\{t\in [0,T]; |\bar x(t)|\le \bar r\}}\frac{1}{2}(\hat H_K'(\bar x),\bar x)dt-
			\frac{\alpha}{2}\int_{\{t\in [0,T];\bar x(t)\le \bar r\}}\hat H_K(\bar x)dt
			+
			(\frac{\alpha}{2}-1)\int_0^T \hat H_K(\bar x)dt\\
			&\ge -\frac{1}{2}C_1\bar rT-\frac{\alpha}{2} C_2T +(\frac{\alpha}{2}-1)T \hat H_K(\bar x(t))
		\end{split}
	\end{equation*}
	for $t\in\R$,
	where $C_1=\displaystyle\max_{|x|\le \bar r}|H'(x)|$ and $C_2= \displaystyle\max_{|x|\le \bar r}H(x)$, since
	$H(x)=H_K(x)$ for $|x|\le K+1$ and $K>\bar r$.
	Then we get for $t\in\R$
	\begin{equation}\label{e:critical-value-s}
		\hat H_K(\bar x(t))\le \frac{2}{\alpha-2}(\frac{d}{T}+\frac{1}{2}C_1\bar r+\frac{\alpha}{2} C_2).
	\end{equation}
	
	We recall \eqref{e:H-super-quadratic}, that is, 
	\begin{equation}\label{e:H-bigger-r}
		\hat H_K(x)\ge a_3|x|^{\alpha}-a_4 \text{ for any } x\in\R^{2n},
	\end{equation}
	where $a_3,a_4>0$ and are independent of $K$.
	
	Combining \eqref{e:critical-value-s}, \eqref{e:H-bigger-r} and 
\eqref{e:critical-value-estimate},
there exists a positive constant $C$ depending on $\bar r$,
	$C_1$, $C_2$ and $\alpha,T$, such that the solution $\bar x$ satisfies \eqref{e:a-priori-estimate-solution}. 
\end{proof}

\begin{lemma}\label{l:i+nu-ge-1}
	Let $(\bar x,T)$ be a a non-constant $T$-periodic
	brake solution of \eqref{e:HS-T} with $\Phi(\bar x)<0$ and let $\gamma_{\bar x}$ be the associated symplectic path of $\bar x$.
	For $t\in [0,T/2]$, set $\gamma(t)=\gamma_{\bar x}(t)$ and $\gamma\in\mathcal P_{\frac{T}{2}}(2n)$. 
	Then we have
	\begin{gather}
		i_{L_1}(\gamma)+\nu_{L_1}(\gamma) \ge 1 \label{e:i-L1-semipositve},\\
		i^{\Lo}_1(\gamma^2)+\nu_1(\gamma^2)\ge n+1
		\label{e:i-long-semipositve}
	\end{gather}
	if
	$H_{pp}(\bar x(t))$ is semi-positive definite and $H_{qq}(\bar x(t))$ is positive definite for $t\in\R$.
\end{lemma}
\begin{proof}
Denote $B_2(t)=\begin{pmatrix}
	H_{pp}(\bar x(t))&H_{pq}(\bar x(t))\\
	H_{qp}(\bar x(t))&H_{qq}(\bar x(t))
	\end{pmatrix}$ and $B_1(t)=\begin{pmatrix}
0&H_{pq}(\bar x(t))\\
H_{qp}(\bar x(t))&\frac{H_{qq}(\bar x(t))}{2}
	\end{pmatrix}$ for $t\in\R$.
	Then $B_2(t)-B_1(t)=\begin{pmatrix}
			H_{pp}(\bar x(t))&0\\
		0&\frac{H_{qq}(\bar x(t))}{2}
	\end{pmatrix}$ is semi-positive definite for $t\in\R$.

Let
$\check L^2(I_T;\R^{2n})=\{u\in L^2(I_T;\R^{2n});u(-t)=-Nu(t)  \text{ for  }t\in I_T\}$  and
\[\check W_T^{1,2}(I_T;\R^{2n})=\{x\in W^{1,2}(I_T;\R^{2n});x(-t)=-Nx(t) \text{ for } t\in I_T \text{ and } x(\frac{T}{2})=x(-\frac{T}{2})\}.\]
We
define a closed self-adjoint Fredholm operator $\check A=-J\frac{d}{dt}$ in $\check L^2(I_T;\R^{2n})$ with domain 
$\dom(\check A)=\check W_T^{1,2}(I_T;\R^{2n})$ and the bounded self-adjoint operators  $\check B_i$ on $\check L^2(I_T;\R^{2n})$ by
\[(\check B_ix,y)_{L^2}=\int_{-\frac{T}{2}}^{\frac{T}{2}}(B_i(t)x(t),y(t))dt,  \quad \text{ for }  x,y\in\check L^2(I_T;\R^{2n}),\ i=1,2.\]
Then $\check B_i$ is relatively compact with respect to $\check A$ for $i=1,2$.
According to the definition of the relative Morse index and the path-additive property of the spectral flow, we have
	\begin{equation}\label{e:relative-B1-B2}
	I(\check A,\check A-\check B_2)=I(\check A,\check A-\check B_1)+I(\check A-\check B_1,\check A-\check B_2).
	\end{equation}

Let $\gamma_1\in \mathcal{P}_{\frac{T}{2}}(2n)$ be the matrizant of linear system \eqref{e:linear-path} with $B(t)=B_1(t)$.
By the relation between the Maslov-type index and the relative Morse index (cf. \cite[the proof of Lemma 3.1]{FZZZ22}), we have
\begin{gather*}
i_{L_1}(\gamma)+n=I(\check A,\check A-\check B_2), \quad \nu_{L_1}(\gamma)=\dim \ker (\check A-\check B_2);\\
i_{L_1}(\gamma_1)+n=I(\check A,\check A-\check B_1), \quad \nu_{L_1}(\gamma_1)=\dim \ker (\check A-\check B_1).
\end{gather*}
Together with \eqref{e:relative-B1-B2}, we get
	\begin{equation}\label{e:iL1-nuL1-gamm}
	i_{L_1}(\gamma)+\nu_{L_1}(\gamma)=I(\check A-\check B_1,\check A-\check B_2)+\dim\ker (\check A-\check B_2)-\dim \ker (\check A-\check B_1)+i_{L_1}(\gamma_1)+\nu_{L_1}(\gamma_1).
	\end{equation}

	According to the form of $B_1$, by the theory of linear differential equations, we obtain  $\gamma_1(t)=\begin{pmatrix}a(t)&b(t)\\
	0&d(t)
	\end{pmatrix}$, where $a(t),b(t),d(t)\in\R^{n\times n}$. In fact, 
\begin{equation*}
	\begin{cases}
		\dot{\gamma_1}(t)=JB_1(t)\gamma_1(t)=\begin{pmatrix}
			-H_{qp}(\bar x(t))&-\frac{H_{qq}(\bar x(t))}{2}\\
				0&H_{pq}(\bar x(t))
		\end{pmatrix}\gamma_1(t), \\
		\gamma_1(0)=I_{2n}.
	\end{cases}
\end{equation*}
	According to  \eqref{e:i-nu-L0-L1}, 
	we obtain \begin{equation*}
		i_{L_1}(\gamma_1)+\nu_{L_1}(\gamma_1)=i_{L_0}(\gamma_1)+\nu_{L_0}(\gamma_1)+m^+\begin{pmatrix}
			0&0\\
			0&d^T(\frac{T}{2})b(\frac{T}{2})
			\end{pmatrix}-n+n\ge i_{L_0}(\gamma_1)+\nu_{L_0}(\gamma_1).
	\end{equation*}
	Together with Lemma \ref{l:partial-positive-L0-L1}.(a), 
	we get 
	\begin{equation}\label{e:gamma_1-i+nu-ge0}
		i_{L_1}(\gamma_1)+\nu_{L_1}(\gamma_1)\ge 0.
	\end{equation}

	Note that $\check A-\check B_2=\check A-\check B_1-(\check B_2-\check B_1)$.
	Since $\check B_2-\check B_1$ is semi-positive definite on $\check L^2(I_T;\R^{2n})$, by Lemma \ref{l:relative-morse-index}, we obtain
	\begin{equation}\label{e:relative-morse-index-ker}
		\begin{split}
		I(\check A-\check B_1&,\check A-\check B_2)+\dim\ker(\check A-\check B_2)-\dim\ker(\check A-\check B_1)
		\\
		&=\sum_{0<s\le1}(\dim\ker(\check A-\check B_1-s(\check B_2-\check  B_1))-\lim_{u\to s^-}\dim\ker(\check A-\check B_1-u(\check B_2-\check B_1))).
		\end{split}
	\end{equation}

	Now we prove that $I(\check A-\check B_1,\check A-\check B_2)+\dim\ker(\check A-\check B_2)-\dim\ker(\check A-\check B_1)\ge 1$. According to \eqref{e:relative-morse-index-ker},
	the left hand of \eqref{e:relative-morse-index-ker} is greater than or equal to $0$.
	If the left hand of \eqref{e:relative-morse-index-ker} is equal to $0$,
	then by \eqref{e:relative-morse-index-ker}, $\dim\ker(\check A-\check B_2)=\dim\ker(\check A-\check B_1-s(\check B_2-\check B_1))$ for $s$ in a neighborhood of $1\in(0,1]$; thus by Lemma \ref{l:ker-A-sB-constant}, $\ker (\check A-\check B_2)=\ker(\check A-\check B_1)\cap \ker(\check B_2-\check B_1)$. Since $\bar x$ is a non-constant periodic solution of the  \textit{autonomous} system \eqref{e:HS-T}, we have $\dot {\bar x} \in\ker (\check A-\check B_2)$; thus $\dot {\bar x}\in\ker (\check B_2-\check B_1)$.
	Write $\bar x(t)=\begin{pmatrix}
		p(t)\\
		q(t)
	\end{pmatrix}$, where $p(t),q(t)\in\R^n$. Since $H_{qq}(\bar x(t))$ is positive definite for $t\in\R$, we have
	$\dot q=0$, thus $q(t)\equiv C\in\R^n$.
	Then \eqref{e:phi-autonomous} becomes
	\[\Phi(\bar x)=\int_{-\frac{T}{2}}^{\frac{T}{2}}\frac{1}{2}\begin{pmatrix}0&-I_n\\
	I_n&0\end{pmatrix}\begin{pmatrix}\dot p(t)\\ \dot q(t)\end{pmatrix}\cdot \begin{pmatrix}p(t)\\q(t)\end{pmatrix}dt+\int_{-\frac{T}{2}}^{\frac{T}{2}}H(\bar x(t))dt=0+\int_{-\frac{T}{2}}^{\frac{T}{2}}H(\bar x(t))dt\ge0,\] which contradicts the assumption
	 $\Phi(\bar x)<0$.
	 Thus the right hand of \eqref{e:relative-morse-index-ker} is strictly greater than $0$.
	 Together with \eqref{e:iL1-nuL1-gamm} and \eqref{e:gamma_1-i+nu-ge0}, we get
	 \eqref{e:i-L1-semipositve}.

	
\end{proof}
Now we give the proof of our main result.
\begin{proof}[Proof of Theorem \ref{t:super-mini-brake}]
	Step 1: To apply Proposition \ref{p:critical}, $H$ will be modified as in Section \Ref{ss:dual-action}.
	We recall:
		choose the integer $m$ such that $2<2+\frac{1}{2m}\le\mu$ and
	denote by $\alpha =2+\frac{1}{2m}$.
	Let $K>\bar r$ and $\chi\in C^{\infty}(\R,\R)$ such that
	$\chi(y)=1$ if $y\le K+1$, $\chi(y)=0$ if $y\ge K+2$, and $\chi'(y)<0$
	if $y\in (K+1,K+2)$.
	Set
	\[H_K(x)=\chi(|x|)H(z)+(1-\chi(|x|))R|x|^{\alpha},\]
	where $R\ge \displaystyle\max_{K+1\le|x|\le K+2}\frac{H(x)}{|x|^{\alpha}}$.
	Then $H_K$ satisfies {\rm (H1)--(H4)} with the same $\bar r$ and $\mu$ replaced by $\alpha$
	in {\rm(H3)}.
	Note that there exists a positive constant $M_K$ such that
	$H''_K(x)y\cdot y\ge -M_K |y|^2$ for $K+1\le |x|\le K+2$.
	Define a special convex $C^2$ function $l$ on $\R^{2n}$ by
	\begin{equation*}
		l(x)=\left\{
		\begin{array}{ll}
			(|x|^2-2K|x|+K^2)^{\frac{\alpha}{2}}, & \hbox{for $|x|\ge K$;} \\
			0, & \hbox{for $|x|\le K$.}
		\end{array}
		\right.
	\end{equation*}
	Then we have
	\[l''(x)y\cdot y\ge \alpha (|x|^2-2K|x|+K^2)^{\frac{\alpha}{2}-1}\frac{|x|-K}{|x|}|y|^2 \text{ for } |x|\ge K.\]
	For $|x|\ge K+1$, we have
	$l''(x)y\cdot y\ge \alpha\frac{1}{K+1}|y|^2$.
	Let $\hat H_K(x)=H_K(x)+\frac{K+1}{\alpha}(M_K+1)l(x)$ for $x\in\R^{2n}$.
	Then $\hat H_K\in C^2(\R^{2n},\R)$ satisfies  {\rm (H1)--(H4)} with $\mu=\alpha$ and {\rm (H0)$^{q+}$}, the following {\rm(H5)},
	and $\hat H_K=H$
	for $|x|\le K$.
	
	Set a strictly convex function
	\begin{equation}\label{e:def-F-convex}
	F(x)=\hat H_K(x)+\frac{1}{2}(\Lambda x,x), \quad \text{where}\quad \Lambda=\begin{pmatrix}
		\lambda I_n&0\\
		0&0
	\end{pmatrix}\in \R^{2n\times 2n}.
	\end{equation}
	In fact, since the Hessian of $\hat H_K$ is positive definite for $|x|\ge K+1$ and {\rm (H0)$^{q+}$}, we can choose sufficient large $\lambda$ as in
	Lemma \ref{l:positive-definite-lambda} such that
	the Hessian of $F$ in \eqref{e:def-F-convex} is positive definite for $x\neq 0$.
	Let $\psi(u)=\int_{-T/2}^{T/2}\left[ \frac{1}{2}(-J\Pi u+J\Lambda J\Pi^2u, u)+F^*(u)\right]dt$ be the reduced dual action functional
	on the space $\hat L^{\beta}(I_T;\R^{2n})$ where $\beta\in(1,2)$ satisfies $\frac{1}{\beta}+\frac{1}{\alpha}=1$ as in Section \ref{ss:dual-action}.
	By Theorems \ref{t:critical-point-mp} and  \ref{t:mountain-condition},  we found a nonzero mountain-pass essential point $\bar u$
	with $\psi(\bar u)>0$.  This leads to a  non-constant $T$-period brake solution $\bar x$ of
	\begin{equation}\label{e:HS-T-K-auto}
		\begin{split}
			&\dot x(t)=J\hat H_K'(x(t)) \\
			&x(-t)=Nx(t),\ \  t\in\R
		\end{split}
	\end{equation}
	such that $-J\frac{d}{dt}\bar x+\Lambda \bar x=\bar u$ and $\bar x$ satisfies \eqref{e:a-priori-estimate-solution}. 
	
	Denote by $K=\max\{C,\bar r\}$,
	where $C$ is a positive number depending only on $\bar r,\alpha,T$ and the maximal value of $H(x)$ and $|H'(x)|$ on $\{x\in\R^{2n};|x|\le \bar r\}$ as in \eqref{e:a-priori-estimate-solution}.
	Then $|\bar x(t)|\le K$ for $t\in \R$. 
	Thus  $\hat H_K(\bar x)=H(\bar x)$ and $\bar x$ solves the original \eqref{e:HS-T}.
	By \eqref{e:func-dual-func}, $\Phi(\bar x)=\Phi_K(\bar x)=-\Psi_K(\bar x)=-d<0$.

	Step 2: We will prove that the mimimal period of $\bar x$ is $T$.
	
	By Theorem \ref{t:two-path-components}, we have $m(\bar x)\le 1$.
	We know that $\bar u\neq 0$; so $\bar x$ is non-constant. Suppose $\bar x$ has period $T/k$. We first claim that:
	if $m(\bar x)=0$, then $k=1$ and $\bar x$ has minimal period $T$. In fact, by Proposition \ref{p:morse-maslov-index}, we have $i_{L_0}(\bar x)=0$. 
	Since $\bar x$ satisfies {\rm (H0)$^{p\ge0}$} or $n=1$, by 	Lemma
	\ref{l:i+nu-ge-1} or Proposition \ref{p:C=0} and Proposition \ref{p:index=0}, we get $k=1$.	
	
	If $m(\bar x)=1$, by Lemma \ref{l:i+nu-ge-1}  or Proposition \ref{p:C=0} and Corollary \ref{c:index=1}, we get $k\le 2$.
Assume $k=2$. Then $\bar x$ is $T/2$-periodic.
	Similar to that in the proof of Lemma \ref{l:simple-period}, 
	 there exists
	an eigenvector $e$ associated with the negative eigenvalue
	of $\psi_{\infty}''(\bar u)$
	such that $e(t+\frac{T}{2})=-e(t)$.
	In fact, let $(x^*,T/2)$ be the $\frac{T}{2}$-periodic brake solution induced by $(\bar x,T)$, i.e., $x^*=\bar x|_{[-T/4,T/4]}$.
	By Proposition \ref{p:morse-maslov-index}, we have $i_{L_0}(\bar x)=1$.
	 By Proposition \ref{p:iteration-brake-inequality},
	\[i_{L_0}(\bar x)\ge 2i_{L_0}(x^*)+\nu_{L_0}(x^*).\]
	 Thus we get $i_{L_0}(x^*)=0$.
	By Lemma \ref{l:morse-decomposition}, we have
	$j_{-1}(x^*)=1$.
	If $y\in \hat W_{T}$, then $-J\dot y+\Lambda y\in \hat L_{\circ}^2(I_T;\R^{2n})$ and
	$\psi_{\infty}''(\bar u)(-J\dot y+\Lambda y,-J\dot y+\Lambda y)=Q_{\frac{T}{2}}(y,y)$.
	Since 	$j_{-1}(x^*)=1$, we get a $y\in \hat W^{-1}_{\frac{T}{2}}$. Denote by  $e=-J\dot y+\Lambda y$. Then
	$e\in \hat L_{\circ}^2(I_T;\R^{2n})$ is an
	eigenvector  associated with the negative eigenvalue
	of $\psi_{\infty}''(\bar u)$ such that $e(t+\frac{T}{2})=-e(t)$.
	According to Proposition \ref{p:morse-index--1}, 
	we get a contradiction. So $\bar x$ has mimimal period $T$.
\end{proof}
In the proof of 
Theorem 
\ref{t:super-mini-brake}, we have proved that
under the assumptions of {\rm (H1)-(H4)}, {\rm (H0)$^{q+}$} and
\begin{enumerate}
	\item [\rm (H5)] $\displaystyle \limsup_{|x|\to \infty}\frac{H(x)}{|x|^{\mu}}<\infty$, where $\mu>2$,
\end{enumerate}
for any $T>0$, \eqref{e:HS-T} has a periodic brake solution $\bar x$ with minimal period $T$ provided
this solution $\bar x$  further satisfies either {\rm (H0)$^{p\ge 0}$} or $n=1$.

\begin{proof}[Proof of Theorem \ref{t:potential-well1}]
	We divide into two steps to prove that: under the conditions of {\rm (H1)-(H2), (H3)$'$, (H4), (H5)} and {\rm (H0)$^{q+}$}, 
	for any $T>0$, \eqref{e:HS-T} has a periodic brake solution $\bar x$ with minimal period $T$ provided
	this solution $\bar x$  further satisfies either {\rm (H0)$^{p\ge 0}$} or $n=1$.

Step 1: 
Under the assumptions of {\rm (H1)-(H5)}, {\rm (H0)$^{q+}$},
for any $T>0$,  we have found  a periodic brake solution $\bar x$ of \eqref{e:HS-T}  with minimal period $T$ provided
this solution $\bar x$  further satisfies either {\rm (H0)$^{p\ge 0}$} or $n=1$.
	Note that the Morse index of $\bar x$ is less than $1$.
	We {\bf claim} for some $t_0\in\R$, 
	\begin{equation}\label{e:H''-barx-t0}
	H''(\bar x(t_0))\le \frac{4\pi}{T}I_{2n}.
	\end{equation}
	If \eqref{e:H''-barx-t0} did not hold, then for any $t\in\R$,
	$H''(\bar x(t))>\frac{4\pi}{T}I_{2n}$, and equivalently 
	\[ \text{ for any } t\in \R,\quad  H''(\bar x(t))^{-1} <\frac{T}{4\pi}I_{2n}.\]
	Since $\bar x$ is $T$-periodic, by the compactness of $I_T=[-T/2,T/2]$, there is some $\epsilon >0$ such that
	\[\text{ for any } t\in I_T, H''(\bar x(t))^{-1} <(\frac{T}{4\pi}-\epsilon)I_{2n}.\]
Consider the quadratic form  on $\hat L^2_{\circ}(I_T;\R^{2n})$
\[q_{T/2}(u,v)=\int_{-T/2}^{T/2}[ (-J(\Pi u)(t),v(t))+ [H''(\bar x(t))]^{-1} u(t),v(t))
]dt,\] where $\Pi\colon L^2_{\circ}(I_T;\R^{2n}) \to L^2_{\circ}(I_T;\R^{2n})$ is the primitive.
Let $u_1=\begin{pmatrix}
	\frac{2\pi}{T}\sin \frac{2\pi t}{T} y_0\\
	-\frac{2\pi}{T}\cos \frac{2\pi t}{T} y_0
\end{pmatrix}
$ and $u_2=\begin{pmatrix}
	\frac{4\pi}{T}\sin \frac{4\pi t}{T} y_1\\
	-\frac{4\pi}{T}\cos \frac{4\pi t}{T} y_1
\end{pmatrix}
$, where $y_0,y_1\in\R^n$. Then we have $u_1$ and $u_2$ are linear independent and
$q(u_1,u_1)<0,q(u_2,u_2)<0$.
So $m^-(\bar x)\ge 2n$. We have a contradiction.

Step 2: 
By virtue of {\rm (H3)$'$}, when $|x|\to \infty $ or $x\to \partial \Omega$, we have $H(x)\to +\infty$.
Set $\Omega_1=\{x\in \Omega; H''(x)\le \frac{4\pi}{T}I_{2n}\}$
and denote $h_0=\sup\{H(x);x\in \Omega_1\}$.
Then $\Omega_1$ and $h_0$ are both bounded and  
\[H''(x)>\frac{4\pi}{T}I_{2n} \text{\ for \ } x\in \Omega \text{ and } H(x)>h_0.\]
Set $\Omega_0=\{x\in\Omega; H(x)\le h_0\}$.
Then $\Omega_0$ is a bounded set with brake symmetry and
\[H''(x)>\frac{4\pi}{T}I_{2n} \quad \text{ for } x\in \Omega\setminus  \Omega_0.\]
There exists a $K>0$ such that $\Omega_0\subset \{x\in \R^{2n};|x|\le K\}$.
Set 
\begin{equation}\label{e:tilde-H}
\tilde H(x)=\chi(|x|)H(x)+(1-\chi(|x|))R|x|^4+\frac{K+1}{4}(M_K+1+\frac{4\pi}{T})l(x),\text{ for } x\in\R^{2n},
\end{equation}
where $\chi$, $R$, $M_K$ and $l$ are defined in the Step 1 of the proof of Theorem \ref{t:super-mini-brake} with $\alpha=4$.
Then  $\tilde H$ satisfies {\rm(H1)--(H5)} with $\mu=4$, {\rm (H0)$^{q+}$} and
\begin{gather*}
	\tilde H(x)=H(x) \text{ for } x\in \Omega_0,\\
	\tilde H''(x)>\frac{4\pi}{T}I_{2n} \text{ for } x\notin \Omega_0.
\end{gather*}
In fact, for $|x|\ge K$, $\tilde H''(x)>\frac{4\pi}{T}I_{2n}$; for $x\in\Omega $ and $|x|\le K$, $\tilde H(x)=H(x)$.

By Step 1, for any $T>0$, we get a $T$-periodic brake solution $\bar x$ of \begin{equation}\label{e:HS-T-tilde}
	\begin{split}
		&\dot x(t)=J\tilde H'(x(t)) \\
		&x(-t)=Nx(t),\ \  t\in\R
	\end{split}
\end{equation}such that the minimal period of
$\bar x$ is $T$ and
\[ \tilde H''(\bar x(t_0))\le \frac{4\pi}{T}I_{2n} \text{ for some } t_0\in\R.\]

This implies $\bar x(t_0)\in\Omega_0$, thus  $\tilde H(\bar x(t))=\tilde H(\bar x(t_0))\le h_0$ for $t\in \R$, for 
$\title H$ is an integral of the motion.
Since $\lim\limits_{x\to \partial \Omega }H(x)=+\infty$, 
$\bar x(t)\in \Omega$ for $t\in\R$; thus $\bar x(t)\in \Omega_0$.
Then 
$H(\bar x )=\tilde H(\bar x)$.
We get that $\bar x$ solves the original problem \eqref{e:HS-T}.
\end{proof}
\begin{appendices}
\section{Appendix}\label{app}
\subsection{The Morse index theory}\label{a:morse}
We will use the following Morse index theorem (cf. \cite[Theorem 1.11]{Uhlen73}) to prove Lemma \ref{l:morse-index}.
\begin{theorem}\label{t:morse-index-theorem}
	Let $Q$ be a bilinear form on a Hilbert space $H$, and
	$\{0\}=H_0\subset H_t\subset H_1=H$, $0\le t \le 1$, an increasing family of closed Hilbert spaces. If 
	\begin{enumerate}
		\item [(i)] $Q$ satisfies the unique continuation property,
		\item [(ii)] the Morse negative index of $Q$ is finite and the associated linear transformation \underline{Q}$\colon H\to H$ is a Fredholm map,
		\item [(iii)] $\overline{\cup_{t<k} H_t }=H_k=\cap_{t>k}H_t$,
	\end{enumerate}
	then there is only a finite number of conjugate points where $n(t)\neq 0$ and the Morse negative index
	 \[m^-(Q) =\sum_{0< t<1}n(t).\]
\end{theorem}
We recall the Morse negative index of $Q$ is the dimension of any maximal subspace of $H$ on which $Q$ is negative definite and
\begin{definition}\label{d:UCP}
	$Q$ satisfies the \textit{unique continuation property} with respect to the family $H_t$ if $N_t\cap N_s=0$ for $t\neq s$, where $N_t=\ker Q|_{H_t}=\{u\in H_t; Q(u,v)=0 \text{ for any } v\in H_t\}$ and $n(t)=\dim N_t$.
\end{definition}
\begin{proof}[Proof of Lemma \ref{l:morse-index}]
Let  
$H_t=\hat L^2_{\circ}(I_t;\R^{2n})$, where $I_t=[-t/2,t/2]$ and $0\le t\le T$, and $H=H_T$. In fact, we identify the element of $H_t$ as its zero extension on $I_T\setminus I_t$, which belongs to $H$.
That is for $u\in H_t$, the restriction on $I_T\setminus I_t$ is zero.
We only need to prove the family of bilinear forms $q_s$ defined in \eqref{e:quadratic-qs} ($s\in[0,\frac{T}{2}]$) satisfies the unique continuation property.
That is $\ker q_s\cap\ker q_t=N_{2s}\cap N_{2t}=0$ for $t\neq s$.
Let $s_1<s_2$ and $u\in  \ker q_{s_1}\cap\ker q_{s_2}=N_{2s_1}\cap N_{2s_2}$. 
This means for $t\in[s_1,s_2]$, $u(t)=0$ and 
there is some constant $\xi_2\in L_0=\{0\}\times \R^n$ such that
	\begin{equation*}
		 -J\Pi_{s_2}u+J\Lambda J\Pi_{s_2}^2u+(B(t)+\Lambda)^{-1}u=\xi_2.
	\end{equation*}
	Let $y=J\Pi_{s_2}u-J\Lambda J \Pi^2_{s_2}u+\xi_2 \in \hat W_{2s_2}^{1,2}(I_{2s_2};\R^{2n}).$
	Then
	\begin{equation}\label{e:linear-equa-a}
		\dot y=JH''(t,\bar x)y.
	\end{equation}
Since the restriction of $u$ on $[-s_2,-s_1]\cup[s_1,s_2]$ is zero,
by the definition of $\Pi_{s_2}$, $J(\Pi_{s_2}u)(t)\in L_0$ is a constant  for $t\in[s_1,s_2]$. 

{\bf Claim:} There is a constant $\xi\in L_0$ such that $y(t)=\xi$ for $t\in [s_1,s_2]$.

Together with \eqref{e:linear-equa-a}, we get
$H''(t,\bar x)\xi=0$. Since $H_{qq}(t,\bar x)$ is positive definite and $\xi\in L_0=\{0\}\times \R^n$, we get $\xi=0$.
By the uniqueness of solutions of ODEs for the initial-value problem, we get $y(t) \equiv 0$ for $t\in I_{2s_2}$. Since $u=(H''(t,\bar x)+\Lambda)y$, we get $u=0$.

Now we prove the claim.
Note that for $u=\begin{pmatrix}
	p\\q
\end{pmatrix}\in H_{2s_2}$, 
 $-J\Lambda J\Pi_{s_2}^2u(t)=\begin{pmatrix}
 	0\\ \lambda \Pi^2q
 	\end{pmatrix}$.
 Since $q(-t)=q(t)$ and $\int_{-s_2}^{s_2}q(t)dt=0$, we have $(\Pi q)(t)=\int^t_{-s_2}q(s')ds'$.
 $\Pi^2q$ equals to $\int^s_{-s_2}\int^t_{-s_2} q(s') ds' dt$ plus a constant belonging to $\R^n$.
 Since $\int_{-s_1}^{s_1}q(t)dt=0$ and the restriction $q|_{[-s_2,-s_1]\cup[s_1,s_2]}=0$, we have
 for $s\in[s_1,s_2]$
 \begin{equation*}
 	\begin{split}
 		\int^s_{-s_2}\int^t_{-s_2} q(s') ds' dt&=\int^{s_1}_{-s_2}\int^t_{-s_2} q(s') ds' dt+\int^{s}_{s_1}\int^t_{-s_2} q(s') ds' dt\\
 		&=\int^{s_1}_{-s_2}\int^t_{-s_2} q(s') ds' dt \in \R^n.
 	\end{split}
 	\end{equation*}
\end{proof}

\subsection{Calculation of the H\"{o}rmander index}\label{a:hormander}
To calculate the H\"{o}rmander index via the triple index, we recall the Hermitian form $Q(\cdot,\cdot;\cdot)$ in \cite[\S 3.1]{zhou-zhu-wu}.
	Let $(V,\omega)$ be a complex symplectic vector space, where $\omega$ is a symplectic
	form on $V$, i.e., a non-degenerate skew-Hermitian form.
	Let $\lambda$ be a linear subspace of $V$.
	The annihilator $\lambda^{\omega}$ of $\lambda$ is defined by
	\[\lambda^{\omega}=\{u\in V;\omega(u,v)=0, \text{ for any } v\in \lambda\}.\]
	A linear subspace $\lambda$ of $V$ is called Lagrangian if $\lambda^{\omega}=\lambda$.
	\begin{definition}\label{d:hermitian-form}
	Given three Lagrangian subspaces $\alpha,\beta,\gamma$ of $V$, we define an Hermitian form
	$Q:=Q(\alpha,\beta;\gamma)$ on
	$\alpha\cap (\beta+\gamma)$ by
	\[Q(x_1,x_2):=\omega(x_1,y_2)=\omega(z_1,y_2)=\omega(x_1,z_2)\]
	for all $x_j=-y_j+z_j\in \alpha\cap (\beta+\gamma)$, where $y_j\in\beta$, $z_j\in\gamma$, $j=1,2$.
	\end{definition}
	Now we calculate some Hermitian forms $Q(\cdot,\cdot;\cdot)$, which can be seen as symmetric matrices.
	\begin{lemma}\label{l:quadratic-forms-calculation}
		Let $(V,\omega)=(\C^{2n}\oplus\C^{2n},(-\omega_0)\oplus \omega_0)$, where \[\omega_0(x_1,x_2)=J\bar x_1\cdot x_2\] for $x_1,x_2\in \C^{2n}$ and $\bar x_1$ means the complex conjugate of the vector $x_1$. Let $\alpha_0=\{0\}\times \C^n$, $\alpha_1=\C^n\times \{0\}$, and  $\tilde \alpha_j=\alpha_j\times\alpha_j$ for $j=0,1$.
		Let $M=\begin{pmatrix}
			A&B\\
			C&D
		\end{pmatrix}\in \Sp(2n)$, where $A,B,C,D\in \R^{n\times n}$. Then
		\begin{enumerate}
			\item [\rm(i)] The matrix of the Hermitian form $Q(\Graph(M),\alpha_0\times \alpha_1;\widetilde\alpha_0)$ is $D^TB$.
			\item [\rm(ii)] The matrix of the Hermitian form $Q(\Graph(M),\widetilde\alpha_1;\widetilde\alpha_0)$ is $\begin{pmatrix}
				C^TA &C^TB\\
				B^TC&D^TB
			\end{pmatrix}$.
			\item [\rm(iii)] The matrix of the Hermitian  form $Q(\Graph(M),\alpha_0\times \alpha_1;\alpha_1\times \alpha_0)$ is $\begin{pmatrix}
				C^TA &A^TD\\
				D^TA&D^TB
			\end{pmatrix}$.
		\end{enumerate}
Furthermore assume $C=0$.
\begin{enumerate}
	\item [\rm(iv)] The matrix of the Hermitian form $Q(\Graph(I_{2n}),\widetilde \alpha_0,M\alpha_0\times \alpha_0)$ is $0$.
	\item [\rm(v)]Let $M_1=NM^{-1}NM$.
	Then \begin{equation}\label{e:NM-1NM}
	M_1=\begin{pmatrix}
		I_n& 2B^TD\\
		0&I_n
	\end{pmatrix}.\end{equation}
	If $B$ is invertible, then the matrix of the Hermitian form $Q(\Graph(M),\widetilde \alpha_0,M_1\alpha_0\times \alpha_0)$ is $-\frac{3}{2}(D^TB)^{-1}$.
\end{enumerate}
\end{lemma}
	\begin{proof}
		Since $M^TJM=J$, we have
		\[C^TA=A^TC, \qquad D^TB=B^TD, \qquad D^TA-B^TC=I_n.\]
		And we can work in the complex symplectic space
		$(\C^{2n}\oplus\C^{2n},\omega)$, where  \[\omega(u_1,u_2)=\begin{pmatrix}
			-J&0\\
			0&J
		\end{pmatrix}\begin{pmatrix}
		\bar x_1\\
		\bar y_1
		\end{pmatrix}\cdot \begin{pmatrix}
		x_2\\
		y_2
		\end{pmatrix},\] for $u_i=\begin{pmatrix}
		x_i\\y_i
		\end{pmatrix}$, $x_i,y_i\in \C^n$, $i=1,2$.
		
	{\rm(i)}  \[\Graph(M)\cap (\alpha_0\times \alpha_1+\widetilde \alpha_0)=\Graph(M)\cap(\alpha_0\times \C^{2n}),\] whose complex dimension is $n$ and elements are $\begin{pmatrix}
		0\\
		y\\
		By\\
		Dy
	\end{pmatrix}$, $y\in\C^n$.
	For $y_i\in\C^n$ and $u_i=\begin{pmatrix}
		0\\
		y_i\\
		By_i\\
		Dy_i 
	\end{pmatrix}$, $i=1,2$, by Definition \ref{d:hermitian-form} we have
	\[Q(u_1,u_2)=\begin{pmatrix}
		0&I_n&0&0\\
		-I_n&0&0&0\\
		0&0&0&-I_n\\
		0&0&I_n&0
	\end{pmatrix}\begin{pmatrix}
0\\
\bar y_1\\B \bar y_1\\
D \bar y_1
	\end{pmatrix}\cdot \begin{pmatrix}
	0\\
	y_2\\0\\
	Dy_2
	\end{pmatrix}=D^TB\bar y_1\cdot y_2.\]
	
{\rm(ii)}  \[\Graph(M)\cap (\widetilde \alpha_1+\widetilde \alpha_0)=\Graph(M)\cap(\C^{2n}\times \C^{2n}),\] whose complex dimension is $2n$ and elements are $\begin{pmatrix}
		x\\
		y\\
		Ax+By\\
		Cx+Dy
	\end{pmatrix}$, $x,y\in\C^n$.
	For $x_i,y_i\in\C^n$ and $u_i=\begin{pmatrix}
		 x_i\\
		 y_i\\
		A x_i+B y_i\\
		C x_i+D y_i
	\end{pmatrix}$, $i=1,2$, by Definition \ref{d:hermitian-form} we have
	\[Q(u_1,u_2)=\begin{pmatrix}
		0&I_n&0&0\\
		-I_n&0&0&0\\
		0&0&0&-I_n\\
		0&0&I_n&0
	\end{pmatrix}\begin{pmatrix}
		\bar x_1\\
		\bar y_1\\A\bar x_1+B\bar y_1\\
		C\bar x_1+D\bar y_1
	\end{pmatrix}\cdot \begin{pmatrix}
		0\\
		y_2\\0\\
		Cx_2+Dy_2
	\end{pmatrix}=\begin{pmatrix}
	C^TA&C^TB\\
	D^TA-I_n&D^TB
	\end{pmatrix}\begin{pmatrix}
	\bar x_1\\
	\bar y_1
	\end{pmatrix}\cdot \begin{pmatrix}
	x_2\\
	y_2
	\end{pmatrix}.
\]

{\rm(iii)}  \[\Graph(M)\cap (\alpha_0\times \alpha_1+\alpha_1\times \alpha_0)=\Graph(M)\cap(\C^{2n}\times \C^{2n}),\] whose complex dimension is $2n$ and elements are $\begin{pmatrix}
	x\\
	y\\
	Ax+By\\
	Cx+Dy
\end{pmatrix}$, $x,y\in\C^n$.
For $x_i,y_i\in\C^n$ and $u_i=\begin{pmatrix}
	 x_i\\
 y_i\\
	A x_i+B y_i\\
	C x_i+Dy_i
\end{pmatrix}$, $i=1,2$, by Definition \ref{d:hermitian-form} we have
\[Q(u_1,u_2)=\begin{pmatrix}
	0&I_n&0&0\\
	-I_n&0&0&0\\
	0&0&0&-I_n\\
	0&0&I_n&0
\end{pmatrix}\begin{pmatrix}
	\bar x_1\\
	\bar y_1\\A\bar x_1+B\bar y_1\\
	C\bar x_1+D\bar y_1
\end{pmatrix}\cdot \begin{pmatrix}
	x_2\\
	0\\0\\
	Cx_2+Dy_2
\end{pmatrix}=\begin{pmatrix}
	C^TA&C^TB+I_n\\
	D^TA&D^TB
\end{pmatrix}\begin{pmatrix}
	\bar x_1\\
	\bar y_1
\end{pmatrix}\cdot \begin{pmatrix}
	x_2\\
	y_2
\end{pmatrix}.
\]
{\rm(iv)}
\[\Graph(I_{2n})\cap (\widetilde\alpha_0+M\alpha_0\times \alpha_0)=\Graph(I_{2n})\cap\widetilde\alpha_0,\]
whose complex dimension is $n$ and elements are $\begin{pmatrix}
	0\\
	y\\
	0\\
	y
\end{pmatrix}$, $y\in\C^n$.
For $y_i\in\C^n$ and $u_i=\begin{pmatrix}
	0\\
	y_i\\
0\\
y_i
\end{pmatrix}$, $i=1,2$, by Definition \ref{d:hermitian-form} we have
\[Q(u_1,u_2)=\begin{pmatrix}
	0&I_n&0&0\\
	-I_n&0&0&0\\
	0&0&0&-I_n\\
	0&0&I_n&0
\end{pmatrix}\begin{pmatrix}
	 0\\
	\bar y_1\\0\\ \bar y_1
\end{pmatrix}\cdot \begin{pmatrix}
	0\\
	0\\0\\
	0
\end{pmatrix}=0.\]
(v)
Since $M^{-1}=-JM^TJ$, by direct computation we get \eqref{e:NM-1NM}.
\[\Graph(M)\cap (\widetilde \alpha_0+M_1\alpha_0\times \alpha_0)=\Graph(M)\cap(\R^{2n}\times \alpha_0),\]
whose elements are
\[\begin{pmatrix}
	x\\-B^{-1}Ax\\0\\-DB^{-1}Ax
\end{pmatrix}=
\begin{pmatrix}
0\\-B^{-1}Ax-\frac{1}{2}(B^TD)^{-1}x\\0\\-DB^{-1}Ax
\end{pmatrix}
+\begin{pmatrix}
x\\\frac{1}{2}(B^TD)^{-1}x\\0\\0
\end{pmatrix}, \quad x\in\R^n.\]
For $x_i\in\C^n$ and $u_i=\begin{pmatrix}
	x_i\\
	-B^{-1}Ax_i\\
	0\\
-DB^{-1}Ax_i
\end{pmatrix}$, $i=1,2$, by Definition \ref{d:hermitian-form} we have
\[Q(u_1,u_2)=\begin{pmatrix}
	0&I_n&0&0\\
	-I_n&0&0&0\\
	0&0&0&-I_n\\
	0&0&I_n&0
\end{pmatrix}\begin{pmatrix}
	\bar x_1\\
	-B^{-1}A \bar x_1\\0\\ -DB^{-1}A\bar x_1
\end{pmatrix}\cdot \begin{pmatrix}
	x_2\\
	\frac{1}{2}(B^TD)^{-1}x_2\\0\\
	0
\end{pmatrix}=-\frac{3}{2}(D^TB)^{-1}\bar x_1\cdot x_2.\]

The proof is complete.
	\end{proof}
Together with  \eqref{e:i-index-calculation-a} of Proposition \ref{p:i-index-calculation}, we have	
\begin{corollary}\label{c:i-index-cal-I}
	\begin{enumerate}[\rm(I)]
\item  $i(\Graph(I_{2n}),\alpha_0\times \alpha_1;\widetilde\alpha_0)=n.$
\item $i(\Graph(I_{2n}),\tilde \alpha_1,\tilde\alpha_0)=n.$
\item $i(\Graph(I_{2n}),\alpha_0\times \alpha_1;\alpha_1\times \alpha_0)=n$.
		\end{enumerate}
\end{corollary}
\begin{proof}
	{\rm (I)} By \eqref{e:i-index-calculation-a} and Lemma \ref{l:quadratic-forms-calculation}.(i), we have
	\begin{equation*}
		\begin{split}
	&\quad\;i(\Graph(I_{2n}),\alpha_0\times\alpha_1,\widetilde\alpha_0)\\
	&=m^+(Q(\Graph(I_{2n}),\alpha_0\times \alpha_1;\widetilde\alpha_0)+\dim_{\C}(\Graph(I_{2n})\cap\widetilde\alpha_0)-\dim_{\C}(\Graph(I_{2n})\cap(\alpha_0\times \alpha_1)\cap \widetilde\alpha_0)\\
	&=n.
	\end{split}
	\end{equation*}
	{\rm(II)} Similarly by \eqref{e:i-index-calculation-a} and Lemma \ref{l:quadratic-forms-calculation}.(ii), we have
	\begin{equation*}
		\begin{split}
			i(\Graph(I_{2n}),\widetilde \alpha_1,\widetilde\alpha_0)&=m^+(Q(\Graph(I_{2n}),\tilde \alpha_1;\widetilde\alpha_0)+\dim_{\C}(\Graph(I_{2n})\cap\widetilde\alpha_0)-\dim_{\C}(\Graph(I_{2n})\cap\widetilde \alpha_1\cap \widetilde\alpha_0)\\
			&=n.
		\end{split}
	\end{equation*}
		{\rm(III)} By \eqref{e:i-index-calculation-a} and Lemma \ref{l:quadratic-forms-calculation}.(iii)
		\begin{equation*}
			\begin{split}
				 i(\Graph(I_{2n}),\alpha_0\times  \alpha_1,\alpha_1\times\alpha_0)
				&=m^+(Q(\Graph(I_{2n}),\alpha_0\times  \alpha_1;\alpha_1\times\alpha_0))
					 +\dim_{\C}(\Graph(I_{2n})\cap(\alpha_1\times\alpha_0))\\&\quad\
			 -\dim_{\C}(\Graph(I_{2n})\cap(\alpha_0\times  \alpha_1)\cap(\alpha_1\times\alpha_0))\\
				&=n.
			\end{split}
		\end{equation*}
	
\end{proof}
\subsection{Appendix: Continuity of eigenvalues}\label{app:continuity-eigemvalue}
First we recall two well-known facts about the eigenvalues of a symmetric operator on a finite-dimensional inner space with an inner product $\langle\cdot,\cdot\rangle$.
\begin{proposition}\label{p:maxmin-eigenvalues}
	(cf. \cite[Section I.6.10]{Ka95})
	Let $T$ be a symmetric operator on an $n$-dimensional vector space $V$ with an inner product $\langle\cdot,\cdot\rangle$.
	Let
	\[\mu_1\le \mu_2\le\cdots \le\mu_n\]
	be the repeated eigenvalues of $T$ arranged in the ascending order.
	For each subspace $M$ of $V$, set
	\[\mu[M]=\mu[T,M]=\min_{u\in M,\|u\|=1}\langle Tu,u\rangle.\]
	The maximin principle asserts that for $j\in\{1,2,...,n\}$
	\[\mu_j=\max_{\codim M=j-1}\mu[M].\]
\end{proposition}

\begin{proposition}(cf. \cite[page 139]{BoWo93})\label{p:continuous-eigenvalue}
	Let $T,T'$ be two symmetric operators on an $n$-dimensional vector space $V$ with an inner product $\langle\cdot,\cdot\rangle$.
	Let
	\[\mu_1\le \mu_2\le\cdots \le\mu_n \ \ \text{ and }\ \ \mu'_1\le \mu'_2\le\cdots \le\mu'_n\]
	be the repeated eigenvalues of $T$ and $T'$ respectively, arranged in the ascending order.
	Then one has
	\[|\mu_j-\mu'_j|\le \|T-T'\|\ \ \ \ \text{for each }\ \ j\in\{1,2,...,n\}.\]
\end{proposition}
\begin{proof}
	We follow the lines in \cite[page 139]{BoWo93}.
	According to Proposition \ref{p:maxmin-eigenvalues},
	for $j\in\{1,2,\ldots,n\}$, each $(j-1)$-dimensional subspace $M$ of $V$,
	\[\mu_j=\max_{\codim M=j-1}\mu[T,M] \ \ \ \text{ and }\ \ \ \mu'_j=\max_{\codim M=j-1}\mu[T',M].\]
	For each $\varepsilon >0$ choose a vector $w\in M$ with $\|w\|=1$ such that
	\[\langle Tw,w\rangle-\mu[T,M]<\varepsilon.\]
	It follows that
	\begin{equation*}
		\begin{split}
			\mu[T',M]& \le \langle T'w,w\rangle \\
			& =\langle Tw,w\rangle+\langle(T'-T)w,w\rangle\\
			&\le \mu[T,M]+\|T-T'\|+\varepsilon;
		\end{split}
	\end{equation*}
	hence
	\[\mu[T',M]\le \mu[T,M]+\|T-T'\|.\]
	In the same way we get
	\[\mu[T,M]\le \mu[T',M]+\|T-T'\|.\]
Thus the proposition follows.
\end{proof}

Now we consider the infinite-dimensional Hilbert space.
Let $X$ be a Hilbert space.
Let $A$ be a closed operator in $X$.
Let $\mathcal{L}(X)$ denote the set of bounded linear operators on $X$.
Denote by 
$\rho(A)=\{\lambda\in \C; A-\lambda\  \text{ is invertible with }\ (A-\lambda)^{-1}\in\mathcal{L}(X)\}$ the resolvent set of $A$. Denote by $\sigma(A)=\C\setminus \rho(A)$ the spectrum of $A$.
Let $(A_s)_{s\in[0,1]}$ be a family of closed operator in $X$.
We recall: $A_s$ converges to $A_t$ as $s\to t$ in the gap topology (cf. \cite[Theorem IV.2.23]{Ka95})
is equivalent to that $\|(A_s-\lambda)^{-1}-(A_t-\lambda)^{-1}\|\to0$  as $s\to t$ for $\lambda\in \sigma(A_t)$.
\begin{theorem}\label{t:continuity-eigemvalue}
	Let $(A_s)_{s\in[0,1]}$ be a continuous family of closed self-adjoint Fredholm operators in $X$. Fix a $t\in[0,1]$. Let $a\in \rho(A_t)$ be a positive real number such that the intersection $\sigma(A_t)\cap[-a,a]$ consists only of a finite system of eigenvalues
	\[-a<\lambda_k(t)\le \lambda_{k+1}(t)\le\cdots\le \lambda_m(t)<a\]
	(all the eigenvalues repeated according to their multiplicity). Then,
	for all $s\in\R$ and $|s-t|$ sufficiently small, the intersection of
	$\sigma(A_s)\cap[-a,a]$ consists of the same number of eigenvalues
	\begin{equation}\label{e:eigenvalue-ascending}
		-a<\lambda_k(s)\le \lambda_{k+1}(s)\le\cdots\le \lambda_m(s)<a,
	\end{equation}
	and one has
	\begin{equation}\label{e:eigenvalue-estimate}
		\lim_{s\to t}|\lambda_j(s)-\lambda_j(t)|= 0, \text{ \rm for each } j\in\{k,\ldots,m\}.
	\end{equation}
\end{theorem}
\begin{proof}
	Let $E(t)=\frac{1}{2\pi i}\int_{\Gamma}(\lambda-A_t)^{-1}d\lambda$,
	where $\Gamma=\{ae^{i\theta}; 0\le \theta\le 2\pi\}$ with positive direction. 
	For $s\in[0,1]$ being close enough to $t$, $\Gamma\subset \rho(A_s)$ and there is a unitary operator $U(s)$ maps $\im E(t)$ onto $\im E(s)$ (cf. \cite[Lemma I.4.10]{Ka95}). Then we consider the two maps on the finite dimensional space
	$\im E(t)$: $E(t)A_tE(t)$ and $E(t)U(s)^{-1}A_sU(s)E(t)$. 
	By the closedness of  $A_s$ (cf. \cite[Theorem III.6.17]{Ka95}),
	\[E(s)A_sE(s)=\frac{1}{2\pi i}\int_{\Gamma}A_s(\lambda-A_s)^{-1}d\lambda=\frac{1}{2\pi i}\int_{\Gamma}\lambda(\lambda-A_s)^{-1}d\lambda;\]
	thus
	\begin{equation*}
		\begin{split}
			&\quad\quad E(s)A_sE(s)-E(t)A_tE(t) \\
			&=\frac{1}{2\pi i}\int_{\Gamma}\lambda((\lambda-A_s)^{-1}-(\lambda-A_t)^{-1})d\lambda.
		\end{split}
	\end{equation*}
	Since $(\lambda-A_t)^{-1}-(\lambda-A_s)^{-1}\to 0\in \mathcal{L}(X)$ as $s\to t$ and
	$\Gamma$ is compact, we get \[\|E(t)A_tE(t)-E(s)A_sE(s)\|\to0 \text{ as } s\to t.\]
	Since
	$E(t)U(s)^{-1}A_sU(s)E(t)=E(t)U(s)^{-1}E(s)A_sE(s)U(s)E(t)$ and $U(s)\to \Id$ (cf. loc. cit.) as $s\to t$, finally we get
	\begin{equation}\label{e:EU-1-AUE-converge}
		\|E(t)U(s)^{-1}A_sU(s)E(t)-E(t)A_tE(t)\|\to0 \text{ as } \ s\to t.
	\end{equation}

	Since $E(s)U(s)=U(s)E(t)$, $E(t)U(s)^{-1}A_sU(s)E(t)=U(s)^{-1}E(s)A_sE(s)U(s)$.
	Noting that $U(s)$ is unitary, we have
	\[\sigma(A_s)\cap[-a,a]=\sigma(E(t)U(s)^{-1}A_sU(s)E(t))\cap[-a,a],\]
	for $s\in\R$ and $|s-t|$ sufficient small.
	Applying Proposition \ref{p:continuous-eigenvalue} 
	for the finite-rank self-adjoint operators $E(t)U(s)^{-1}A_sU(s)E(t)$ on $\im E(t)$, we have
	\begin{equation}\label{e:lambda-j-continuous}
		|\lambda_j(s)-\lambda_j(t)|
		\le \|E(t)U(s)^{-1}A_sU(s)E(t)-E(t)A_tE(t)\|, \ \text{ for } j\in\{k,\ldots,m\},
	\end{equation}
	where $\{\lambda_j(s); \lambda_j(s)\in \sigma(A_s)\cap[-a,a]\}$ satisfy \eqref{e:eigenvalue-ascending}.
	By \eqref{e:lambda-j-continuous} and \eqref{e:EU-1-AUE-converge},
	we get (\ref{e:eigenvalue-estimate}).
\end{proof}
\subsection{Symmetric matrices}
\begin{lemma}\label{l:positive-definite-lambda}
	Let $B=\begin{pmatrix}
		A&C^T\\
		C&D
	\end{pmatrix}$ be a symmetric $2n\times 2n$-matrix with $D\in \GL(n,\R)$.
	Assume that $D$ is positive definite.
	Then there is a $\lambda>0$ such that $B+\Lambda$ is positive definite,
	where $\Lambda=\begin{pmatrix}
		\lambda I_n&0\\
		0&0
	\end{pmatrix}$.
\end{lemma}
\begin{proof}
	Let $P=\begin{pmatrix}
		I_n&-C^TD^{-1}\\
		0&I_n
	\end{pmatrix}$.
	Then $PBP^T=\begin{pmatrix}
		A-C^TD^{-1}C&0\\
		0&D
	\end{pmatrix}$.
	Then we can choose $\lambda >0$ sufficient large, e.g., 
	\[\lambda > \lambda_{\max}(C^TD^{-1}C-A),\]
	where $\lambda_{\max}(C^TD^{-1}C-A)$ denotes the maximum eigenvalue of $C^TD^{-1}C-A$.
	Then
	$PBP^T+\Lambda$ is positive definite.
	Since $P\Lambda P^T=\Lambda$,
	$B+\Lambda$ is positive definite.
\end{proof}
\subsection{Homotopy invariance of the Maslov index}
We recall
\begin{definition}[cf.{\cite[Definition 5.1.2]{Lo}}]\label{d:omega-homotopic}
For $\tau>0$ and $e^{\sqrt{-1}\theta}\in\U(1)$, where $\U(1)$ is the unitary group of dimension $1$.
Given two paths $\gamma_0,\gamma_1\in \mathcal{P}_{\tau}(2n)$, we say $\gamma_0$ and $\gamma_1$ are $e^{\sqrt{-1}\theta}$-{\bf homotopic}, if there exists a $f\in C([0,1]\times [0,\tau],\Sp(2n))$ such that
\begin{gather*}
	f(0,t)=\gamma_0(t),\quad f(1,t)=\gamma_1(t),\quad \text{for}\quad t\in [0,\tau],\\
\text{and}\quad	f(s,0)=I_{2n},\quad \nu_{e^{\sqrt{-1}\theta}}(f(s,\tau))\, \text{is constant for } s\in[0,1].
\end{gather*}
We write $\gamma_0\sim_{e^{\sqrt{-1}\theta}}\gamma_1$.
\end{definition}
Then we have
\begin{lemma}\label{l:index-invariance}\label{l:homotopy-invar}
Given a $\tau>0$, set $S_{\tau}=\R/\tau \Z$.
Let $B\in C(S_{\tau},\mathcal{L}_s(\R^{2n}))$ be a $\tau$-periodic path of symmetric matrices.
Given $t_0\in[0,\tau)$, define $B_1(t)=B(t+t_0)$ for $t\in\R$.
Let $\gamma,\gamma_1\in \mathcal{P}_{\tau}(2n)$ be the matrizant of the linear system \eqref{e:linear-path} with 
$B$ and $B_1$ respectively. Then we have
\[i_{e^{\sqrt{-1}\theta}}(\gamma)=i_{e^{\sqrt{-1}\theta}}(\gamma_1),\quad \nu_{e^{\sqrt{-1}\theta}}(\gamma)=\nu_{e^{\sqrt{-1}\theta}}(\gamma_1),\]
here $i_1(\gamma)$ means $i^{\Lo}_1(\gamma)$.
\end{lemma}
\begin{proof}
	By definition, 
	$\gamma_1(t)=\gamma(t+t_0)\gamma(t_0)^{-1}$.
	Note that since $B(t+\tau)=B(t)$ for $t\in\R$,
	\[\gamma(t+\tau)=\gamma(t)\gamma(\tau).\]
So $\gamma_1(\tau)=\gamma(\tau+t_0)\gamma(t_0)^{-1}=\gamma(t_0)\gamma(\tau)\gamma(t_0)^{-1}$.
Thus $\nu_{e^{\sqrt{-1}\theta}}(\gamma)=\nu_{e^{\sqrt{-1}\theta}}(\gamma_1)$.
We define the $e^{\sqrt{-1}\theta}$-homotopy $f\in C([0,1]\times [0,\tau],\Sp(2n))$ by
\[f(s,t)=\gamma(st_0+t)\gamma(st_0)^{-1}.\]
Then $\gamma\sim_{e^{\sqrt{-1}\theta}}\gamma_1$.
By \cite[Theorem 6.2.3]{Lo}, we get our lemma.
\end{proof}
\end{appendices}

\end{document}